\documentclass[oneside,english,a4paper]{amsart}
\usepackage[T1]{fontenc}
\usepackage[latin9]{inputenc}
\usepackage{verbatim}
\usepackage{amstext}
\usepackage{amsthm}
\usepackage{amssymb}
\usepackage{stmaryrd}
\usepackage[all]{xy}

\makeatletter
\numberwithin{equation}{section}
\numberwithin{figure}{section}
\numberwithin{table}{section}
\theoremstyle{plain}
\newtheorem{thm}{\protect\theoremname}[section]
\theoremstyle{plain}
\newtheorem{lem}[thm]{\protect\lemmaname}
\theoremstyle{plain}
\newtheorem{cor}[thm]{\protect\corollaryname}
\theoremstyle{definition}
\newtheorem{defn}[thm]{\protect\definitionname}
\theoremstyle{remark}
\newtheorem{rem}[thm]{\protect\remarkname}
\theoremstyle{plain}
\newtheorem{prop}[thm]{\protect\propositionname}
\theoremstyle{definition}
\newtheorem{example}[thm]{\protect\examplename}
\theoremstyle{remark}
\newtheorem{notation}[thm]{\protect\notationname}

\subjclass[2010]{Primary: 14D23, Secondary: 13F25, 14H30}

\usepackage{extarrow}
\usepackage{enumerate}
\usepackage{fontenc}
\usepackage[T1]{fontenc}
\usepackage{graphicx}
\usepackage[colorlinks=true, linkcolor=blue]{hyperref}
\usepackage{latexsym}
\usepackage{mathrsfs}
\usepackage{mathtext}
\usepackage{stmaryrd}
\usepackage{tikz}
\usepackage{textcomp}
\usepackage{upgreek}
\usepackage{xy}

\usetikzlibrary{arrows}
\usetikzlibrary{matrix}
\usetikzlibrary{shapes}
\usetikzlibrary{snakes}
\usetikzlibrary{matrix}

\DeclareMathOperator{\Add}{\textup{Add}}
\DeclareMathOperator{\Aff}{\textup{Aff}}
\DeclareMathOperator{\Alg}{\textup{Alg}}
\DeclareMathOperator{\Ann}{\textup{Ann}}
\DeclareMathOperator{\Arr}{\textup{Arr}}
\DeclareMathOperator{\Art}{\textup{Art}}
\DeclareMathOperator{\Ass}{\textup{Ass}}
\DeclareMathOperator{\Aut}{\textup{Aut}}
\DeclareMathOperator{\Autsh}{\underline{\textup{Aut}}}
\DeclareMathOperator{\Bi}{\textup{B}}
\DeclareMathOperator{\CAdd}{\textup{CAdd}}
\DeclareMathOperator{\CAlg}{\textup{CAlg}}
\DeclareMathOperator{\CMon}{\textup{CMon}}
\DeclareMathOperator{\CPMon}{\textup{CPMon}}
\DeclareMathOperator{\CRings}{\textup{CRings}}
\DeclareMathOperator{\CSMon}{\textup{CSMon}}
\DeclareMathOperator{\CaCl}{\textup{CaCl}}
\DeclareMathOperator{\Cart}{\textup{Cart}}
\DeclareMathOperator{\Cl}{\textup{Cl}}
\DeclareMathOperator{\Coh}{\textup{Coh}}
\DeclareMathOperator{\Coker}{\textup{Coker}}
\DeclareMathOperator{\Cov}{\textup{Cov}}
\DeclareMathOperator{\Der}{\textup{Der}}
\DeclareMathOperator{\Div}{\textup{Div}}
\DeclareMathOperator{\End}{\textup{End}}
\DeclareMathOperator{\Endsh}{\underline{\textup{End}}}
\DeclareMathOperator{\Ext}{\textup{Ext}}
\DeclareMathOperator{\Extsh}{\underline{\textup{Ext}}}
\DeclareMathOperator{\FAdd}{\textup{FAdd}}
\DeclareMathOperator{\FCoh}{\textup{FCoh}}
\DeclareMathOperator{\FGrad}{\textup{FGrad}}
\DeclareMathOperator{\FLoc}{\textup{FLoc}}
\DeclareMathOperator{\FMod}{\textup{FMod}}
\DeclareMathOperator{\FPMon}{\textup{FPMon}}
\DeclareMathOperator{\FRep}{\textup{FRep}}
\DeclareMathOperator{\FSMon}{\textup{FSMon}}
\DeclareMathOperator{\FVect}{\textup{FVect}}
\DeclareMathOperator{\Fibr}{\textup{Fibr}}
\DeclareMathOperator{\Fix}{\textup{Fix}}
\DeclareMathOperator{\Fl}{\textup{Fl}}
\DeclareMathOperator{\Fr}{\textup{Fr}}
\DeclareMathOperator{\Funct}{\textup{Funct}}
\DeclareMathOperator{\GAlg}{\textup{GAlg}}
\DeclareMathOperator{\GExt}{\textup{GExt}}
\DeclareMathOperator{\GHom}{\textup{GHom}}
\DeclareMathOperator{\GL}{\textup{GL}}
\DeclareMathOperator{\GMod}{\textup{GMod}}
\DeclareMathOperator{\GRis}{\textup{GRis}}
\DeclareMathOperator{\GRiv}{\textup{GRiv}}
\DeclareMathOperator{\Gal}{\textup{Gal}}
\DeclareMathOperator{\Gl}{\textup{Gl}}
\DeclareMathOperator{\Grad}{\textup{Grad}}
\DeclareMathOperator{\Hilb}{\textup{Hilb}}
\DeclareMathOperator{\Hl}{\textup{H}}
\DeclareMathOperator{\Hom}{\textup{Hom}}
\DeclareMathOperator{\Homsh}{\underline{\textup{Hom}}}
\DeclareMathOperator{\ISym}{\textup{Sym}^*}
\DeclareMathOperator{\Imm}{\textup{Im}}
\DeclareMathOperator{\Irr}{\textup{Irr}}
\DeclareMathOperator{\Iso}{\textup{Iso}}
\DeclareMathOperator{\Isosh}{\underline{\textup{Iso}}}
\DeclareMathOperator{\Ker}{\textup{Ker}}
\DeclareMathOperator{\LAdd}{\textup{LAdd}}
\DeclareMathOperator{\LAlg}{\textup{LAlg}}
\DeclareMathOperator{\LMon}{\textup{LMon}}
\DeclareMathOperator{\LPMon}{\textup{LPMon}}
\DeclareMathOperator{\LRings}{\textup{LRings}}
\DeclareMathOperator{\LSMon}{\textup{LSMon}}
\DeclareMathOperator{\Left}{\textup{L}}
\DeclareMathOperator{\Lex}{\textup{Lex}}
\DeclareMathOperator{\Loc}{\textup{Loc}}
\DeclareMathOperator{\M}{\textup{M}}
\DeclareMathOperator{\ML}{\textup{ML}}
\DeclareMathOperator{\MLex}{\textup{MLex}}
\DeclareMathOperator{\Map}{\textup{Map}}
\DeclareMathOperator{\Mod}{\textup{Mod}}
\DeclareMathOperator{\Mon}{\textup{Mon}}
\DeclareMathOperator{\Ob}{\textup{Ob}}
\DeclareMathOperator{\Obj}{\textup{Obj}}
\DeclareMathOperator{\PDiv}{\textup{PDiv}}
\DeclareMathOperator{\PGL}{\textup{PGL}}
\DeclareMathOperator{\PML}{\textup{PML}}
\DeclareMathOperator{\PMLex}{\textup{PMLex}}
\DeclareMathOperator{\PMon}{\textup{PMon}}
\DeclareMathOperator{\Pic}{\textup{Pic}}
\DeclareMathOperator{\Picsh}{\underline{\textup{Pic}}}
\DeclareMathOperator{\Pro}{\textup{Pro}}
\DeclareMathOperator{\Proj}{\textup{Proj}}
\DeclareMathOperator{\QAdd}{\textup{QAdd}}
\DeclareMathOperator{\QAlg}{\textup{QAlg}}
\DeclareMathOperator{\QCoh}{\textup{QCoh}}
\DeclareMathOperator{\QMon}{\textup{QMon}}
\DeclareMathOperator{\QPMon}{\textup{QPMon}}
\DeclareMathOperator{\QRings}{\textup{QRings}}
\DeclareMathOperator{\QSMon}{\textup{QSMon}}
\DeclareMathOperator{\R}{\textup{R}}
\DeclareMathOperator{\Rep}{\textup{Rep}}
\DeclareMathOperator{\Rings}{\textup{Rings}}
\DeclareMathOperator{\Riv}{\textup{Riv}}
\DeclareMathOperator{\SFibr}{\textup{SFibr}}
\DeclareMathOperator{\SMLex}{\textup{SMLex}}
\DeclareMathOperator{\SMex}{\textup{SMex}}
\DeclareMathOperator{\SMon}{\textup{SMon}}
\DeclareMathOperator{\SchI}{\textup{SchI}}
\DeclareMathOperator{\Sh}{\textup{Sh}}
\DeclareMathOperator{\Soc}{\textup{Soc}}
\DeclareMathOperator{\Spec}{\textup{Spec}}
\DeclareMathOperator{\Specsh}{\underline{\textup{Spec}}}
\DeclareMathOperator{\Stab}{\textup{Stab}}
\DeclareMathOperator{\Supp}{\textup{Supp}}
\DeclareMathOperator{\Sym}{\textup{Sym}}
\DeclareMathOperator{\TMod}{\textup{TMod}}
\DeclareMathOperator{\Top}{\textup{Top}}
\DeclareMathOperator{\Tor}{\textup{Tor}}
\DeclareMathOperator{\Vect}{\textup{Vect}}
\DeclareMathOperator{\alt}{\textup{ht}}
\DeclareMathOperator{\car}{\textup{char}}
\DeclareMathOperator{\codim}{\textup{codim}}
\DeclareMathOperator{\degtr}{\textup{degtr}}
\DeclareMathOperator{\depth}{\textup{depth}}
\DeclareMathOperator{\divis}{\textup{div}}
\DeclareMathOperator{\et}{\textup{et}}
\DeclareMathOperator{\ffpSch}{\textup{ffpSch}}
\DeclareMathOperator{\h}{\textup{h}}
\DeclareMathOperator{\ilim}{\displaystyle{\lim_{\longrightarrow}}}
\DeclareMathOperator{\ind}{\textup{ind}}
\DeclareMathOperator{\indim}{\textup{inj dim}}
\DeclareMathOperator{\lf}{\textup{LF}}
\DeclareMathOperator{\op}{\textup{op}}
\DeclareMathOperator{\ord}{\textup{ord}}
\DeclareMathOperator{\pd}{\textup{pd}}
\DeclareMathOperator{\plim}{\displaystyle{\lim_{\longleftarrow}}}
\DeclareMathOperator{\pr}{\textup{pr}}
\DeclareMathOperator{\pt}{\textup{pt}}
\DeclareMathOperator{\rk}{\textup{rk}}
\DeclareMathOperator{\tr}{\textup{tr}}
\DeclareMathOperator{\type}{\textup{r}}
\DeclareMathOperator*{\colim}{\textup{colim}}
\usepackage[colorinlistoftodos]{todonotes}

\usepackage{enumitem}
\setenumerate[1]{label={\arabic*)}} 

\theoremstyle{plain}
\newtheorem{thma}{Theorem}
\newtheorem{thmb}{Theorem}

\makeatother

\usepackage{babel}
\providecommand{\corollaryname}{Corollary}
\providecommand{\definitionname}{Definition}
\providecommand{\examplename}{Example}
\providecommand{\lemmaname}{Lemma}
\providecommand{\notationname}{Notation}
\providecommand{\propositionname}{Proposition}
\providecommand{\remarkname}{Remark}
\providecommand{\theoremname}{Theorem}

\begin{document}
\title{Moduli of formal torsors}
\author{Fabio Tonini, Takehiko Yasuda}
\address{Universitá degli Studi di Firenze, Dipartimento di Matematica e Informatica
\textquoteright Ulisse Dini\textquoteright , Viale Giovanni Battista
Morgagni, 67/A, 50134 Firenze, Italy}
\email{fabio.tonini@unifi.it}
\address{Department of Mathematics, Graduate School of Science, Osaka University,
Toyonaka, Osaka 560-0043, Japan}
\email{takehikoyasuda@math.sci.osaka-u.ac.jp}
\keywords{power series field, the Artin-Schreier theory, Deligne-Mumford stacks,
torsors}
\begin{abstract}
We construct the moduli stack of torsors over the formal punctured
disk in characteristic $p>0$ for a finite group isomorphic to the
semidirect product of a $p$-group and a tame cyclic group. We prove
that the stack is a limit of separated Deligne-Mumford stacks with
finite and universally injective transition maps.
\end{abstract}

\maketitle
\global\long\def\FLAlg{\textup{FLAlg}}%

\global\long\def\PLAlg{\textup{PLAlg}}%

\global\long\def\A{\mathbb{A}}%

\global\long\def\Ab{(\textup{Ab})}%

\global\long\def\C{\mathbb{C}}%

\global\long\def\Cat{(\textup{cat})}%

\global\long\def\Di#1{\textup{D}(#1)}%

\global\long\def\E{\mathcal{E}}%

\global\long\def\F{\mathbb{F}}%

\global\long\def\GCov{G\textup{-Cov}}%

\global\long\def\Gcat{(\textup{Galois cat})}%

\global\long\def\Gfsets#1{#1\textup{-fsets}}%

\global\long\def\Gm{\mathbb{G}_{m}}%

\global\long\def\GrCov#1{\textup{D}(#1)\textup{-Cov}}%

\global\long\def\Grp{(\textup{Grps})}%

\global\long\def\Gsets#1{(#1\textup{-sets})}%

\global\long\def\HCov{H\textup{-Cov}}%

\global\long\def\MCov{\textup{D}(M)\textup{-Cov}}%

\global\long\def\MHilb{M\textup{-Hilb}}%

\global\long\def\N{\mathbb{N}}%

\global\long\def\PGor{\textup{PGor}}%

\global\long\def\PGrp{(\textup{Profinite Grp})}%

\global\long\def\PP{\mathbb{P}}%

\global\long\def\Pj{\mathbb{P}}%

\global\long\def\Q{\mathbb{Q}}%

\global\long\def\RCov#1{#1\textup{-Cov}}%

\global\long\def\RR{\mathbb{R}}%

\global\long\def\Sch{\textup{Sch}}%

\global\long\def\WW{\textup{W}}%

\global\long\def\Z{\mathbb{Z}}%

\global\long\def\acts{\curvearrowright}%

\global\long\def\alA{\mathscr{A}}%

\global\long\def\alB{\mathscr{B}}%

\global\long\def\arr{\longrightarrow}%

\global\long\def\arrdi#1{\xlongrightarrow{#1}}%

\global\long\def\catC{\mathscr{C}}%

\global\long\def\catD{\mathscr{D}}%

\global\long\def\catF{\mathscr{F}}%

\global\long\def\catG{\mathscr{G}}%

\global\long\def\comma{,\ }%

\global\long\def\covU{\mathcal{U}}%

\global\long\def\covV{\mathcal{V}}%

\global\long\def\covW{\mathcal{W}}%

\global\long\def\duale#1{{#1}^{\vee}}%

\global\long\def\fasc#1{\widetilde{#1}}%

\global\long\def\fsets{(\textup{f-sets})}%

\global\long\def\iL{r\mathscr{L}}%

\global\long\def\id{\textup{id}}%

\global\long\def\la{\langle}%

\global\long\def\odi#1{\mathcal{O}_{#1}}%

\global\long\def\ra{\rangle}%

\global\long\def\rig{\mathbin{\!\!\pmb{\fatslash}}}%

\global\long\def\set{(\textup{Sets})}%

\global\long\def\sets{(\textup{Sets})}%

\global\long\def\shA{\mathcal{A}}%

\global\long\def\shB{\mathcal{B}}%

\global\long\def\shC{\mathcal{C}}%

\global\long\def\shD{\mathcal{D}}%

\global\long\def\shE{\mathcal{E}}%

\global\long\def\shF{\mathcal{F}}%

\global\long\def\shG{\mathcal{G}}%

\global\long\def\shH{\mathcal{H}}%

\global\long\def\shI{\mathcal{I}}%

\global\long\def\shJ{\mathcal{J}}%

\global\long\def\shK{\mathcal{K}}%

\global\long\def\shL{\mathcal{L}}%

\global\long\def\shM{\mathcal{M}}%

\global\long\def\shN{\mathcal{N}}%

\global\long\def\shO{\mathcal{O}}%

\global\long\def\shP{\mathcal{P}}%

\global\long\def\shQ{\mathcal{Q}}%

\global\long\def\shR{\mathcal{R}}%

\global\long\def\shS{\mathcal{S}}%

\global\long\def\shT{\mathcal{T}}%

\global\long\def\shU{\mathcal{U}}%

\global\long\def\shV{\mathcal{V}}%

\global\long\def\shW{\mathcal{W}}%

\global\long\def\shX{\mathcal{X}}%

\global\long\def\shY{\mathcal{Y}}%

\global\long\def\shZ{\mathcal{Z}}%

\global\long\def\st{\ | \ }%

\global\long\def\stA{\mathcal{A}}%

\global\long\def\stB{\mathcal{B}}%

\global\long\def\stC{\mathcal{C}}%

\global\long\def\stD{\mathcal{D}}%

\global\long\def\stE{\mathcal{E}}%

\global\long\def\stF{\mathcal{F}}%

\global\long\def\stG{\mathcal{G}}%

\global\long\def\stH{\mathcal{H}}%

\global\long\def\stI{\mathcal{I}}%

\global\long\def\stJ{\mathcal{J}}%

\global\long\def\stK{\mathcal{K}}%

\global\long\def\stL{\mathcal{L}}%

\global\long\def\stM{\mathcal{M}}%

\global\long\def\stN{\mathcal{N}}%

\global\long\def\stO{\mathcal{O}}%

\global\long\def\stP{\mathcal{P}}%

\global\long\def\stQ{\mathcal{Q}}%

\global\long\def\stR{\mathcal{R}}%

\global\long\def\stS{\mathcal{S}}%

\global\long\def\stT{\mathcal{T}}%

\global\long\def\stU{\mathcal{U}}%

\global\long\def\stV{\mathcal{V}}%

\global\long\def\stW{\mathcal{W}}%

\global\long\def\stX{\mathcal{X}}%

\global\long\def\stY{\mathcal{Y}}%

\global\long\def\stZ{\mathcal{Z}}%

\global\long\def\then{\ \Longrightarrow\ }%

\global\long\def\L{\textup{L}}%

\global\long\def\l{\textup{l}}%

\section*{Introduction}

The main subject of this paper is the moduli space of formal torsors,
that is, $G$-torsors (also called principal $G$-bundles) over the
formal punctured disk $\Spec k((t))$ for a given finite group (or\emph{
}étale finite group scheme) $G$ and field $k$. More precisely, we
are interested in a space over a field $k$ whose $k$-points are
$G$-torsors over $\Spec k((t))$. Since torsors may have non-trivial
automorphisms, this space should actually be a stack in groupoids
and it should not be confused with $\Bi G=[\Spec k((t))/G]$, which
is a stack defined over $k((t))$.

The case where the characteristic of $k$ and the order of $G$ are
coprime is called tame and the other case is called wild. The two
cases are strikingly different: in the tame case the moduli space
is expected to be zero-dimensional, while in the wild case it is expected
to be infinite-dimensional. 

An important work on this subject is Harbater's one \cite{Harbater1980}.
He constructed the coarse moduli space for pointed formal torsors
when $k$ is an algebraically closed field of characteristic $p>0$
and $G$ is a $p$-group. This coarse moduli space is isomorphic to
the inductive limit $\varinjlim_{n}\A^{n}$ of affine spaces such
that the transition map $\A^{n}\to\A^{n+1}$ is the composition of
the closed embedding $\A^{n}\hookrightarrow\A^{n+1}$ and the Frobenius
map of $\A^{n+1}$. In particular it is neither a scheme nor an algebraic
space, but an ind-scheme. Some of the differences between Harbater's
space and our space are explained in Remark \ref{rem:Harbater vs us}.
As a consequence Harbater shows that there is a bijective correspondence
between $G$-torsors over the affine line $\A^{1}$ and over $\Spec k((t))$.
In this direction an important development has been given by Gabber
and Katz in \cite{Katz}. Later Pries \cite{Pries} and Obus-Pries
\cite{Obus-Pries} constructed moduli/parameter spaces for groups
$\Z/p\rtimes C$ and $\Z/p^{m}\rtimes C$ with $C$ a tame cyclic
group respectively and Fried-Mezard \cite{Fried-Mezard} constructed
a parameter space of (not necessarily Galois) covers of $\Spec k((t))$
with given ramification data; all these works assumed $k$ to be algebraically
closed. 

In recent works \cite{Yasuda-p-cyclic,Yasuda-toward} of the second
named author, an unexpected relation of this moduli space to singularities
of algebraic varieties was discovered. He has formulated a conjectural
generalization of the motivic McKay correspondence by Batyrev \cite{Batyrev}
and Denef-Loeser \cite{Denef-Loeser} to arbitrary characteristics,
which relates a motivic integral over the moduli space of formal torsors
with a stringy invariant of wild quotient singularities. The motivic
integral can be viewed as the motivic counterpart of mass formulas
for local Galois representations, see \cite{Wood-Yasuda-I,Wood-Yasuda-II}.
The first and largest problem for other groups is the construction
of the moduli space. From the arithmetic viewpoint, the case where
$k$ is finite is the most interesting, which motivates us to remove
the ``algebraically closed'' assumption in earlier works.

The main result of this paper is to construct the moduli stack of
formal torsors and to show that it is a limit of Deligne-Mumford stacks
(DM stacks for short) when $k$ is an arbitrary field of characteristic
$p>0$ and $G$ is an étale group scheme over $k$ which is geometrically
the semidirect product $H\rtimes C$  of a $p$-group $H$ and a cyclic
group $C$ of order coprime with $p$. This is an important step towards
the general case, because, if $k$ is algebraically closed, then connected
$G$-torsors over $\Spec k((t))$ (or equivalently Galois extensions
of $k((t))$ with group $G$) exist only for semidirect products as
before. Moreover any $G'$-torsor for a general $G'$ is induced by
some connected $G$-torsor along an embedding $G\hookrightarrow G'$. 

To give the precise statement of the result, we introduce the following
notation. We denote by $\Delta_{G}$ the category fibered in groupoids
over the category of affine $k$-schemes such that for a $k$-algebra
$B$, $\Delta_{G}(\Spec B)$ is the category of $G$-torsors over
$\Spec B((t))$. The following is the precise statement of the main
result:

\begin{thma}\label{A} Let $k$ be a field of positive characteristic
$p$ and $G$ be a finite and étale group scheme over $k$ such that
$G\times_{k}\overline{k}$ is a semidirect product $H\rtimes C$  of
a $p$-group $H$ and a cyclic group $C$ of rank coprime with $p$. 

\begin{enumerate}

\item \label{thma-general} Then there exists a direct system $\stX_{*}$
of separated DM stacks with finite and universally injective transition
maps, with a direct system of finite and étale atlases (see \ref{def:system of atlases}
for the definition) $X_{n}\arr\stX_{n}$ from affine schemes and with
an isomorphism $\varinjlim_{n}\stX_{n}\simeq\Delta_{G}$.

\item \label{thma-p-gp} If $G$ is a constant $p$-group then the
stacks $\stX_{n}$ can be chosen to be  smooth and integral. More
precisely there is a strictly increasing sequence $v\colon\N\arr\N$
such that $X_{n}=\A^{v_{n}}$, the maps $\A^{v_{n}}\arr\stX_{n}$
are finite and étale of degree $\sharp$G and the transition maps
$\A^{v_{n}}\arr\A^{v_{n+1}}$ are composition of the inclusion $\A^{v_{n}}\arr\A^{v_{n+1}}$
and the Frobenius $\A^{v_{n+1}}\arr\A^{v_{n+1}}$.

\item \label{thma-abelian-p} If $G$ is an abelian constant group
of order $p^{r}$ then we also have an equivalence 
\[
\left(\varinjlim_{n}\A^{v_{n}}\right)\times\Bi G\simeq\Delta_{G}
\]
and the map from $\varinjlim_{n}\A^{v_{n}}$ to the sheaf of isomorphism
classes of $\Delta_{G}$, which is nothing but the rigidification
$\Delta_{G}\rig G$ (see Appendix \ref{sec:Rigidification}), is an
isomorphism.\end{enumerate}\end{thma}

As a consequence of assertion \ref{thma-general} of this theorem
(and \ref{prop:limit of stacks is stack}) the fibered category $\Delta_{G}$
is a stack. 

We now explain the outline of our construction. We first consider
the case of a constant group scheme of order $p^{r}$. Following Harbater's
strategy, we prove the theorem in this case by induction. We obtain
the explicit description of $\Delta_{G}$ as in assertion \ref{thma-abelian-p}
when $G=\Z/p\Z$ by the Artin-Schreier theory (Theorem \ref{thm:The case of Zp});
this is one of the two base cases. It is not difficult to generalize
it to the case $G\simeq(\Z/p\Z)^{n}$ (Lemma \ref{lem:direct system of Fp vector spaces}),
which forms the initial step of induction. Since a general $p$-group
has a central subgroup $H\subset G$ isomorphic to $(\Z/p\Z)^{n}$,
we have a natural map $\Delta_{G}\longrightarrow\Delta_{G/H}$, enabling
the induction to work. We then use the fact (Proposition \ref{prop:factorizing through the rigidification })
that this map factors into the rigidification $\Delta_{G}\longrightarrow\Delta_{G}\rig H$
and an $X_{H}$-torsor $\Delta_{G}\rig H\longrightarrow\Delta_{G/H}$
with $X_{H}=\Delta_{H}\rig H$, to construct a direct system for $\Delta_{G}$
from one for $\Delta_{G/H}$.

Next we consider the other base case, the case of the group scheme
$\mu_{n}$ of $n$-th roots of unity with $n$ coprime to $p$. In
this case, we have the following explicit description of $\Delta_{G}$,
including also the case of characteristic zero:

\begin{thmb}\label{cyclic} Let $k$ be a field and $n\in\N$ such
that $n\in k^{*}$. We have an equivalence
\[
\bigsqcup_{q=0}^{n-1}\Bi(\mu_{n})\arr\Delta_{\mu_{n}}
\]
where the map $\Bi(\mu_{n})\arr\Delta_{\mu_{n}}$ in the index $q$
maps the trivial $\mu_{n}$-torsor to the $\mu_{n}$-torsor $\frac{k((t))[Y]}{(Y^{n}-t^{q})}\in\Delta_{\mu_{n}}(k)$.
\end{thmb}

When $G$ is a constant group of the form $H\rtimes C$ and $k$ contains
all $n$-th roots of unity, then $C\simeq\mu_{n}$ and there exists
a map $\Delta_{G}\longrightarrow\Delta_{\mu_{n}}$. Using Theorem
\ref{A} for $p$-groups, we show that the fiber products $\Delta_{G}\times_{\Delta_{\mu_{n}}}\Spec k$
with respect to $n$ maps $\Spec k\to\Delta_{\mu_{n}}$ induced from
the equivalence in Theorem \ref{cyclic} are limits of DM stacks.
Finally, to conclude that $\Delta_{G}$ itself is a limit of DM stacks
and also to reduce the problem to the case of a constant group, we
need a proposition (Proposition \ref{lem:descent of ind along torsors})
roughly saying that if $\shY$ is a $G$-torsor over a stack $\shX$
for a constant group $G$ and $\shY$ is a limit of DM stacks, then
$\shX$ is also a limit of DM stacks. This innocent-looking proposition
turns out to be rather hard to prove and we will make full use of
2-categories.

The moduli stack of formal torsors introduced in this paper is used
in \cite{Formal-Torsors-II} to construct a moduli space in a weaker
sense for general finite étale group schemes and in \cite{Yasuda-motivic}
to develop the motivic integration over wild DM stacks. Moreover it
is showed in \cite{Formal-Torsors-II} that the motivic integral in
the conjecture mentioned above on the McKay correspondence makes rigorous
sense and this conjecture is finally proved in \cite{Yasuda-motivic}.
According to discussion and observation in \cite{alpha_p}, it appears
quite meaningful to generalize the McKay correspondence further to
nonreduced finite group schemes. For this reason, the moduli problem
of formal torsors for such group schemes would be important as a future
study.

Notice that points of $\Delta_{G}$ over a field $L$, namely $G$-torsors
over $L((t))$, can also be seen as (not necessarily connected) Galois
extensions of $L((t))$ and, taking integers, as special covers of
$L[[t]]$ with an action of $G$. It is therefore natural to ask and
indeed this has been our initial approach to the problem, if one can
define a moduli space of special $G$-covers of $B[[t]]$ for varying
$B$ or, more precisely, give a different moduli interpretation of
$G$-torsors of $B((t))$ in terms of covers of $B[[t]]$, in the
spirit of \cite{Tonini-monoidal} and \cite{Tonini-diagonalizable}.
We don't have a precise answer to this question, but in \cite{Formal-Torsors-II,Yasuda-motivic}
we give partial answers.

The paper is organized as follows. In Section \ref{sec:Notation-and-Terminology}
we set up notation and terminology frequently used in the paper. In
Section \ref{sec:Preliminaries} we collect basic results on power
series rings, finite and universally injective morphisms and torsors.
In Section \ref{sec:system of DM stacks}, after introducing a few
notions and proving a few easy results, the rest of the section is
devoted to the proof of the proposition mentioned above (Proposition
\ref{lem:descent of ind along torsors}). Section \ref{sec:formal torsors}
is the main body of the paper, where we prove Theorems \ref{A} and
\ref{cyclic}. The proof of Theorem \ref{cyclic} is given on page
\pageref{pf: B}, the one of Theorem \ref{A}, \ref{thma-p-gp} and \ref{thma-abelian-p}
is given on page \pageref{pf: A p} and the one of Theorem \ref{A},
\ref{thma-general} is given on page \pageref{pf: A-general}. Lastly
we include two Appendices about limits of fibered categories, implicitly
used in Theorem \ref{A}, and rigidification, an operation introduced
in \cite{Abramovich2007} for algebraic stacks and that we extends
to more general stacks.

\subsection*{Acknowledgments}

We thank Alexis Bouthier, Ted Chinburg, Ofer Gabber, David Harbater,
Florian Pop, Shuji Saito, Takeshi Saito, Melanie Matchett Wood and
Lei Zhang for stimulating discussion and helpful information. The
second author was supported by JSPS KAKENHI Grant Numbers JP15K17510
and JP16H06337. Parts of this work were done when the first author
was staying as the Osaka University and when the second author was
staying at the Max Planck Institute for Mathematics and the Institut
des Hautes Études Scientifiques. We thank hospitality of these institutions. 

\section{\label{sec:Notation-and-Terminology}Notation and terminology}

Given a ring $B$ we denote by $B((t))$ the ring of Laurent series
$\sum_{i=r}^{\infty}b_{i}t^{i}$ with $b_{i}\in B$ and $r\in\Z$,
that is, the localization $B[[t]]_{t}=B[[t]][t^{-1}]$ of the formal
power series ring $B[[t]]$ with coefficients in $B$. This should
not be confused with the fraction field of $B[[t]]$ (when $B$ is
a domain). 

By a fibered category over a ring $B$ we always mean a category fibered
in groupoids over the category $\Aff/B$ of affine $B$-schemes.

Recall that a finite map between fibered categories is by definition
affine and therefore represented by finite maps of algebraic spaces.

By a vector bundle on a scheme $X$ we always mean a locally free
sheaf of finite rank. A vector bundle on a ring $B$ is a vector bundle
on $\Spec B$ or, before sheafification, a projective $B$-module
of finite type.

If $\shC$ is a category, $X\colon\shC\arr(\textup{groups})$ is a
functor of groups and $S$ is a set we denote by $X^{(S)}\colon\shC\arr(\textup{groups})$
the functor so defined: if $c\in\shC$ then $X^{(S)}(c)$ is the set
of functions $u\colon S\arr X(c)$ such that $\{s\in S\st u(s)\neq1_{X(c)}\}$
is finite.

We recall that for a morphism $f\colon\shX\to\shY$ of fibered categories
over a ring $B$, $f$ is faithful (resp. fully faithful, an equivalence)
if and only if for every affine $B$-scheme $U$, $f_{U}\colon\shX(U)\to\shY(U)$
is so (see \cite[\href{http://stacks.math.columbia.edu/tag/003Z}{003Z}]{SP014}).
A morphism of fibered categories is called a \emph{monomorphism }if
it is fully faithful. We also note that every representable (by algebraic
spaces) morphism of stacks is faithful (\cite[\href{http://stacks.math.columbia.edu/tag/02ZY}{02ZY}]{SP014}).

A map $f\colon\stY\arr\stX$ between fibered categories over $\Aff/k$
is a torsor under a sheaf of groups $\shG$ over $\Aff/k$ if it is
given a $2$-Cartesian diagram\[   \begin{tikzpicture}[xscale=1.5,yscale=-1.2]     \node (A0_0) at (0, 0) {$\stY$};     \node (A0_1) at (1, 0) {$\Spec k$};     \node (A1_0) at (0, 1) {$\stX$};     \node (A1_1) at (1, 1) {$\Bi \shG$};     \path (A0_0) edge [->]node [auto] {$\scriptstyle{}$} (A0_1);     \path (A0_0) edge [->]node [auto] {$\scriptstyle{}$} (A1_0);     \path (A0_1) edge [->]node [auto] {$\scriptstyle{}$} (A1_1);     \path (A1_0) edge [->]node [auto] {$\scriptstyle{}$} (A1_1);   \end{tikzpicture}   \] 

By a stack we mean a stack over the category $\Aff$ of affine schemes
with respect to the fppf topology, unless a different site is specified.

We often abbreviate ``Deligne-Mumford stack'' to ``DM stack''.

\section{\label{sec:Preliminaries}Preliminaries}

In this section we collect some general results that will be used
later.

\subsection{Some results on power series}
\begin{lem}
\label{lem:open subset of B=00005B=00005Bt=00005D=00005D}Let $C$
be a ring, $J\subseteq C$ be an ideal and assume that $C$ is $J$-adically
complete. If $U\subseteq\Spec C$ is an open subset containing $\Spec(C/J)$
then $U=\Spec C$.
\end{lem}

\begin{proof}
Let $U=\Spec C-V(I)$, where $I\subseteq C$ is an ideal. The condition
$\Spec(C/J)\subseteq U$ means that $I+J=C$. In particular there
exists $g\in I$ and $j\in J$ such that $g=1+j$. Since $j$ is nilpotent
in all the rings $C/J^{n}$ we see that $g$ is invertible in all
the rings $C/J^{n}$, which easily implies that $g$ is invertible
in $C$. Thus $I=C$.
\end{proof}
\begin{lem}
\label{lem:lifting for fomally etale maps and series} Let $R$ be
a ring, $X$ be a quasi-affine scheme formally étale over $R$, $C$
be an $R$-algebra and $J$ be an ideal such that $C$ is $J$-adically
complete. Then the projection $C\arr C/J^{n}$ induces a bijection
\[
X(C)\arr X(C/J^{n})\text{ for all }n\in\N
\]
\end{lem}

\begin{proof}
Since $X$ is formally étale the projections $C/J^{m}\arr C/J^{n}$
for $m\geq n$ induce bijections 
\[
X(C/J^{m})\arr X(C/J^{n})
\]
Thus it is enough to prove that if $Y$ is any quasi-affine scheme
over $R$ then the natural map
\[
\alpha_{Y}\colon Y(C)\arr\varprojlim_{n\in N}Y(C/J^{n})
\]
is bijective. This is clear when $Y$ is affine. Let $B=\Hl^{0}(\odi Y)$,
so that $Y$ is a quasi-compact open subset of $U=\Spec B$. The fact
that $\alpha_{U}$ is an isomorphism tells us that $\alpha_{Y}$ is
injective. To see that it is surjective we have to show that if $B\arr C$
is a map such that all $\Spec C/J^{n}\arr\Spec B$ factors through
$Y$ then also $\phi\colon\Spec C\arr\Spec B$ factors through $Y$.
But the first condition implies that $\phi^{-1}(Y)$ is an open subset
of $\Spec C$ containing $\Spec C/J$. The equality $\phi^{-1}(Y)=\Spec C$
then follows from \ref{lem:open subset of B=00005B=00005Bt=00005D=00005D}.
\end{proof}
\begin{cor}
\label{cor:shrinking to etale neighborhood} Let $B$ be a ring, $f\colon Y\arr\Spec B[[t]]$
an étale map, $\xi\colon\Spec L\arr\Spec B$ a geometric point and
assume that the geometric point $\Spec L\arr\Spec B\arr\Spec B[[t]]$
is in the image on $f$. Then there exists an étale neighborhood $\Spec B'\arr\Spec B$
of $\xi$ such that $\Spec B'[[t]]\arr\Spec B[[t]]$ factors through
$Y\arr\Spec B[[t]]$.
\end{cor}

\begin{proof}
We can assume $Y$ affine, say $Y=\Spec C$. Set $B'=C/tC$, so that
the induced map $f_{0}\colon Y_{0}=\Spec B'\arr\Spec B$ is étale.
By hypothesis the geometric point $\Spec L\arr\Spec B$ is in the
image of $f_{0}$ and therefore factors through $f_{0}$. Moreover
the map $Y_{0}\arr Y$ gives an element of $Y(B')$ which, by \ref{lem:lifting for fomally etale maps and series},
lifts to an element of $Y(B'[[t]])$, that is a factorization of $\Spec B'[[t]]\arr\Spec B[[t]]$
through $Y\arr\Spec B[[t]]$.
\end{proof}
\begin{lem}
\label{lem:change of rings for series}Let $R$ be a ring, $S$ be
an $R$-algebra and consider the map
\[
\omega_{S/R}\colon R[[t]]\otimes_{R}S\to S[[t]]
\]
The image of $\omega_{S/R}$ is the subring of $S[[t]]$ of series
$\sum s_{n}t^{n}$ such that there exists a finitely generated $R$
submodule $M\subseteq S$ with $s_{n}\in M$ for all $n\in\N$. 

If any finitely generated $R$ submodule of $S$ is contained in a
finitely presented $R$ submodule of $S$ then $\omega_{S/R}$ is
injective.
\end{lem}

\begin{proof}
The claim about the image of $\omega_{S/R}$ is easy.

Given an $R$-module $M$ we define $M[[t]]$ as the $R$-module $M^{\N}$.
Its elements are thought of as series $\sum_{n}m_{n}t^{n}$ and $M[[t]]$
has a natural structure of $R[[t]]$-module. This association extends
to a functor $\Mod A\to\Mod A[[t]]$ which is easily seen to be exact.
Moreover there is a natural map
\[
\omega_{M/R}\colon R[[t]]\otimes_{R}M\to M[[t]]
\]
Since both functors are right exact and $\omega_{M/R}$ is an isomorphism
if $M$ is a free $R$-module of finite rank, we can conclude that
$\omega_{M/R}$ is an isomorphism if $M$ is a finitely presented
$R$-module. Let $\shP$ be the set of finitely presented $R$ submodules
of $S$. By hypothesis this is a filtered set. Passing to the limit
we see that the map

\[
\omega_{S/R}\colon R[[t]]\otimes_{R}S\simeq\varinjlim_{M\in\shP}(R[[t]]\otimes_{R}M)\arr\varinjlim_{M\in\shP}M[[t]]=\bigcup_{M\in\shP}M[[t]]\subseteq S[[t]]
\]
is injective.
\end{proof}
\begin{lem}
\label{lem: constant sheaves and t} Let $N$ be a finite set and
denote by $\underline{N}\colon\Aff/\Z\arr\sets$ the associated constant
sheaf. Then the maps
\[
\underline{N}(B)\arr\underline{N}(B[[t]])\arr\underline{N}(B((t)))
\]
are bijective. In other words if $B$ is a ring and $B((t))\simeq C_{1}\times\cdots\times C_{l}$
(resp. $B[[t]]=C_{1}\times\cdots\times C_{l}$) then $B\simeq B_{1}\times\cdots\times B_{l}$
and $C_{j}=B_{j}((t))$ (resp. $C_{j}=B_{j}[[t]]$).
\end{lem}

\begin{proof}
Notice that $\underline{N}$ is an affine scheme étale over $\Spec\Z$.
Since $A=B[[t]]$ is $t$-adically complete we obtain that $\underline{N}(B[[t]])\arr\underline{N}(B[[t]]/tB[[t]])$
is bijective thanks to \ref{lem:lifting for fomally etale maps and series}.
Since $B\arr B[[t]]/tB[[t]]$ is an isomorphism we can conclude that
$\underline{N}(B)\arr\underline{N}(B[[t]])$ is bijective. 

Let $n$ be the cardinality of $N$ and $C$ be a ring. An element
of $\underline{N}(C)$ is a decomposition of $\Spec C$ into $n$-disjoint
open and closed subsets. In particular if $n=2$ then $\underline{N}(C)$
is the set of open and closed subsets of $\Spec C$. Taking this into
account it is easy to reduce the problem to the case $n=2$. In this
case another way to describe $\underline{N}$ is $\underline{N}=\Spec\Z[x]/(x^{2}-x)$,
so that $\underline{N}(C)$ can be identified with the set of idempotents
of $C$. Consider the map $\alpha_{B}\colon\underline{N}(B[[t]])\arr\underline{N}(B((t)))$,
which is injective since $B[[t]]\arr B((t))$ is so. If, by contradiction,
$\alpha_{B}$ is not surjective, we can define $k>0$ as the minimum
positive number for which there exist a ring $B$ and $a\in B[[t]]$
such that $a/t^{k}\in\underline{N}(B((t)))$ and $a/t^{k}\notin B[[t]]$.
Let $B,a$ as before and set $a_{0}=a(0)$. It is easy to check that
$a_{0}^{2}=0$ in $B$. Set $C=B/\langle a_{0}\rangle$. By \ref{lem:change of rings for series}
we have that $B[[t]]/a_{0}B[[t]]=C[[t]]$ and that $B((t))/a_{0}B((t))=C((t))$.
Thus we have a commutative diagram   \[   \begin{tikzpicture}[xscale=2.7,yscale=-1.2]     \node (A0_0) at (0, 0) {$\underline N(B[[t]])$};     \node (A0_1) at (1, 0) {$\underline N(B((t)))$};     \node (A1_0) at (0, 1) {$\underline N(C[[t]])$};     \node (A1_1) at (1, 1) {$\underline N(C((t)))$};     \path (A0_0) edge [->]node [auto] {$\scriptstyle{\alpha_B}$} (A0_1);     \path (A0_0) edge [->]node [auto] {$\scriptstyle{}$} (A1_0);     \path (A0_1) edge [->]node [auto] {$\scriptstyle{\beta}$} (A1_1);     \path (A1_0) edge [->]node [auto] {$\scriptstyle{\alpha_C}$} (A1_1);   \end{tikzpicture}   \] 
in which the vertical maps are bijective, since the topological space
of a spectrum does not change modding out by a nilpotent. By construction
$\beta(a/t^{k})=a'/t^{k-1}$ where $a'=(a-a_{0})/t$. By minimality
of $k$ we must have that $a'/t^{k-1}\in C[[t]]$. Since the vertical
maps in the above diagram are bijective we can conclude that also
$a/t^{k}\in B[[t]]$, a contradiction.
\end{proof}
\begin{lem}
\label{lem:Hom sheaves}Let $M,N$ be vector bundles on $B((t))$.
Then the functor
\[
\Homsh_{B((t))}(M,N)\colon\Aff/B\arr\sets\comma C\longmapsto\Hom_{C((t))}(M\otimes C((t)),N\otimes C((t)))
\]
is a sheaf in the fpqc topology.
\end{lem}

\begin{proof}
Set $H=\Hom_{B((t))}(M,N)$, which is a vector bundle over $B((t))$.
Moreover if $C$ is a $B$-algebra we have $H\otimes_{B((t))}C((t))\simeq\Homsh_{B((t))}(M,N)(C)$
because $M$ and $N$ are vector bundles. By \ref{lem:check stack on finite coverings}
we have to prove descent on coverings indexed by a finite set and,
by \ref{lem: constant sheaves and t}, it is enough to consider a
faithfully flat map $B\arr C$. If $i_{1},i_{2}\colon C\arr C\otimes_{B}C$
are the two inclusions, descent corresponds to the exactness of the
sequence

\[
0\arr H\arr C((t))\otimes H\arrdi{(i_{1}-i_{2})\otimes\id_{H}}(C\otimes_{B}C)((t))\otimes H
\]
Since this sequence is obtained applying $\otimes_{B((t))}H$ to the
exact sequence $0\arr B((t))\arr C((t))\arrdi{i_{1}-i_{2}}(C\otimes_{B}C)((t))$
and $H$ is flat we get the result.
\end{proof}

\subsection{Finite and universally injective morphisms}
\begin{defn}
A map $\stX\arr\stY$ between algebraic stacks is \emph{universally
injective }(resp. \emph{universally bijective}, a\emph{ universal
homeomorphism})\emph{ }if for all maps $\stY'\arr\stY$ from an algebraic
stack the map $|\stX\times_{\stY}\stY'|\arr|\stY'|$ on topological
spaces is injective (resp. bijective, an homeomorphism).
\end{defn}

\begin{rem}
In order to show that a map $\stX\arr\stY$ is universally injective
(resp. universally bijective, a universal homeomorphism) it is enough
to test on maps $\stY'\arr\stY$ where $\stY'$ is an affine scheme.
Indeed injectivity and surjectivity can be tested on the geometric
fibers. Moreover if $\bigsqcup_{i}Y_{i}\arr\stY$ is a smooth surjective
 map and the $Y_{i}$ are affine then $|\stX|\arr|\stY|$ is open
if $|\stX\times_{\stY}Y_{i}|\arr|Y_{i}|$ is open for all $i$. In
particular if $\stX\arr\stY$ is representable then it is universally
injective (resp. bijective, a universal homeomorphism) if and only
if it is represented by map of algebraic spaces which are universally
injective (resp. universally bijective, universal homeomorphisms)
in the usual sense.
\end{rem}

\begin{prop}
\label{prop:finite and universally injective maps} Let $f\colon\stX\arr\stY$
be a map of algebraic stacks. Then $f$ is finite and universally
injective if and only if it is a composition of a finite universal
homeomorphism and a closed immersion. More precisely, if $\shI=\Ker(\odi{\stY}\arr f_{*}\odi{\stX})$,
then $\stX\arr\Spec(\odi{\stY}/\shI)$ is finite and a universal homeomorphism.
\end{prop}

\begin{proof}
The if part in the statement is clear. So assume that $f$ is finite
and universally injective and consider the factorization $\stX\arrdi g\stZ=\Spec(\odi{\stY}/\shI)\arrdi h\stY$.
Since $f$ is finite, the map $g$ is finite and surjective. Since
$h$ is a monomorphism, given a map $U\arr\stZ$ from a scheme we
have that $\stX\times_{\stZ}U\arr\stX\times_{\stY}U$ is an isomorphism,
which implies that $g$ is also universally injective as required.
\end{proof}
\begin{rem}
The following properties of morphisms of schemes are stable by base
change and fpqc local on the base: finite, closed immersion, universally
injective, surjective and universal homeomorphism (see \cite[\href{http://stacks.math.columbia.edu/tag/02WE}{02WE}]{SP014}).
In particular for representable maps of algebraic stacks those properties
can be checked on an atlas.
\end{rem}

\begin{rem}
\label{rem:product of universally injective finite} Let $f\colon\stS'\arr\stS$
be a map of algebraic stacks, $\stU,\stV$ and $\stU',\stV'$ algebraic
stacks with a map to $\stS$ and $\stS'$ respectively and $u\colon\stU'\arr\stU$,
$v\colon\stV'\arr\stV$ be $\stS$-maps. If $f,u,v$ are finite and
universally injective then so is the induced map $\stU'\times_{\stS'}\stV'\arr\stU\times_{\stS}\stV$.
The map $(\stU\times_{\stS}\times\stV)\times_{\stS}\stS'\arr\stU\times_{\stS}\stV$
is finite and universally injective. The map $\stU'\arr\stU\times_{\stS}\stS'$
is also finite and universally injective because $\stU\times_{\stS}\stS'\arr\stU$
and $\stU'\arr\stU$ are so (use \cite[\href{http://stacks.math.columbia.edu/tag/01S4}{01S4}]{SP014}
for the universal injectivity). Thus we can assume $\stS=\stS'$ and
$f=\id$. In this case it is enough to use the factorization $\stU'\times_{\stS}\stV'\arr\stU\times_{\stS}\stV'\arr\stU\times_{\stS}\stV$.
\end{rem}

\subsection{Some results on torsors}

In what follows, actions of groups (or sheaves of groups) are supposed
to be right actions. Recall that for a sheaf $\shG$ of groups on
a site $\shC$, $\Bi\shG$ denotes the category of $\shG$-torsors
over objects of $\shC$, and that given a map $\shG\to\shH$ of sheaves
of groups, then there exists a functor $\Bi\shG\to\Bi\shH$ sending
a $\shG$-torsor $\shP$ to the $\shH$-torsor $(\shP\times\shH)/\shG$. 
\begin{lem}
\label{lem:central subgroups and torsors} Let $\shG$ be a sheaf
of groups on a site $\shC$ and $\shH$ be a sheaf of subgroups of
the center $Z(\shG)$. Then $\shH$ is normal in $\shG$, the map
$\mu\colon\shG\times\shH\arr\shG$ restriction of the multiplication
is a morphism of groups and the first diagram   \[   \begin{tikzpicture}[xscale=2.1,yscale=-1.2]     \node (A0_0) at (0, 0) {$\Bi (\shG\times \shH)$};     \node (A0_1) at (1, 0) {$\Bi(\shG)$};     \node (A0_2) at (2, 0) {$\shG\times\shH$};     \node (A0_3) at (3, 0) {$\shG$};     \node (A1_0) at (0, 1) {$\Bi(\shG)$};     \node (A1_1) at (1, 1) {$\Bi(\shG/\shH)$};     \node (A1_2) at (2, 1) {$\shG$};     \node (A1_3) at (3, 1) {$\shG/\shH$};     \path (A0_0) edge [->]node [auto] {$\scriptstyle{}$} (A0_1);     \path (A0_1) edge [->]node [auto] {$\scriptstyle{}$} (A1_1);     \path (A1_0) edge [->]node [auto] {$\scriptstyle{}$} (A1_1);     \path (A0_3) edge [->]node [auto] {$\scriptstyle{}$} (A1_3);     \path (A0_2) edge [->]node [auto] {$\scriptstyle{\pr_1}$} (A1_2);     \path (A0_0) edge [->]node [auto] {$\scriptstyle{}$} (A1_0);     \path (A0_2) edge [->]node [auto] {$\scriptstyle{\mu}$} (A0_3);     \path (A1_2) edge [->]node [auto] {$\scriptstyle{}$} (A1_3);   \end{tikzpicture}   \] induced
by the second one is $2$-Cartesian. A quasi-inverse $\Bi(\shG)\times_{\Bi(\shG/\shH)}\Bi(\shG)\arr\Bi(\shG\times\shH)=\Bi(\shG)\times\Bi(\shH)$
is obtained as follows: given $(\shP,\shQ,\lambda)\in\Bi(\shG)\times_{\Bi(\shG/\shH)}\Bi(\shG)$
(so that $\lambda\colon\shP/\shH\arr\shQ/\shH$ is a $\shG/\shH$-equivariant
isomorphism) we associate $(\shP,\shI_{\lambda})$, where $\shI_{\lambda}$
is the fiber of $\lambda$ along the map $\Isosh^{\shG}(\shP,\shQ)\arr\Isosh^{\shG/\shH}(\shP/\shH,\shQ/\shH)$
and the action of $\shH$ is given by $\shH\arr\shG\arr\Autsh(\shQ)$.
\end{lem}

\begin{proof}
Let $(\shP,\shQ,\lambda)\in\Bi(\shG)\times_{\Bi(\shG/\shH)}\Bi(\shG)$
over an object $c\in\shC$. The composition $\shH\arr\shG\arr\Autsh(\shQ)$
has image in $\Autsh^{\shG}(\shQ)$ because $\shH$ is central. Moreover
the map $\Isosh^{\shG}(\shP,\shQ)\arr\Isosh^{\shG/\shH}(\shP/\shH,\shQ/\shH)$
is $\shH$-equivariant. It is also an $\shH$-torsor: locally when
$\shP$ and $\shQ$ are isomorphic to $\shG$, the previous map become
$\shG\arr\shG/\shH$. Thus $\shI_{\lambda}$ is an $\shH$-torsor
over $c\in\shC$. Thus we have two well defined functors
\[
\Lambda\colon\Bi(\shG)\times_{\Bi(\shG/\shH)}\Bi(\shG)\arr\Bi(\shG)\times\Bi(\shH)\text{ and }\Delta\colon\Bi(\shG)\times\Bi(\shH)\arr\Bi(\shG)\times_{\Bi(\shG/\shH)}\Bi(\shG)
\]
and we must show they are quasi-inverses of each other. Let's consider
the composition $\Lambda\circ\Delta$ and $(\shP,\shE)\in\Bi(\shG)\times\Bi(\shH)$
over $c\in\shC$. We have $\Delta(\shP,\shE)=(\shP,(\shP\times\shE\times\shG)/\shG\times\shH,\lambda)$
where $\lambda$ is the inverse of
\[
[(\shP\times\shE\times\shG)/\shG\times\shH]/\shH\arr\shP/\shH\comma(p,e,1)\arr p
\]
We have to give an $\shH$-equivariant map $\shE\arr\shI_{\lambda}$.
Given a global section $e\in\shE$, that is a map $\shH\arr\shE$,
we get a $\shG\times\shH$ equivariant morphism $\shP\times\shH\arr\shP\times\shE$
and thus a $\shG$-equivariant morphism
\[
\delta\colon\shP\arr(\shP\times\shH\times\shG)/\shG\times\shH\arr(\shP\times\shE\times\shG)/\shG\times\shH
\]
which is easily seen to induce $\lambda$. Mapping $e$ to $\delta$
gives an $\shH$-equivariant map $\stE\to\stI_{\lambda}$. There are
several conditions that must be checked but they are all elementary
and left to the reader.

Now consider $\Delta\circ\Lambda$ and $(\shP,\shQ,\lambda)\in\Bi(\shG)\times_{\Bi(\shG/\shH)}\Bi(\shG)$
over an object $c\in\shC$. It is easy to see that
\[
(\shP\times\shI_{\lambda}\times\shG)/\shG\times\shH\arr\shQ\comma(p,\phi,g)\mapsto\phi(p)g
\]
is a $\shG$-equivariant morphism and it induces a morphism $\Delta\circ\Lambda(\shP,\shQ,\lambda)\arr(\shP,\shQ,\lambda)$.
\end{proof}
\begin{rem}
\label{rem:BG insensible to universal homeo, finite} If $X\arr Y$
is integral (e.g. finite) and a universal homeomorphism of schemes
and $G$ is an étale group scheme over a field $k$ then $\Bi G(Y)\arr\Bi G(X)$
is an equivalence. Indeed by \cite[Expose VIII, Theorem 1.1]{SGA4-2}
the fiber product induces an equivalence between the category of schemes
étale over $Y$ and the category of schemes étale over $X$.
\end{rem}

\begin{lem}
\label{lem:dividing rank for gerbes} Let $G$ be a finite group scheme
over $k$ of rank $\rk G$, $U\arr\shG$ a finite, flat and finitely
presented map of degree $\rk G$ and $\shG\arr T$ be a map locally
equivalent to $\Bi G$, where $U$, $\shG$ and $T$ are categories
fibered in groupoids. If $U\arr T$ is faithful then it is an equivalence.
\end{lem}

\begin{proof}
By changing the base $T$ we can assume that $T$ is a scheme, $\shG=\Bi G\times T$
and $U$ is an algebraic space. We must prove that if $P\arr U$ is
a $G$-torsor and $P\arr U\arr T$ is a cover of degree $\rk G$ then
$f\colon U\arr T$ is an isomorphism. It follows that $f\colon U\arr T$
is flat, finitely presented and quasi-finite. Moreover $f\colon U\arr\Bi G\times T\arr T$
is proper. We can conclude that $f\colon U\arr T$ is finite and flat.
Looking at the ranks of the involved maps we see that $f$ must have
rank $1$.
\end{proof}

\section{\label{sec:system of DM stacks}Direct system of Deligne-Mumford
stacks}

In this section we discuss some general facts about direct limits
of DM stacks. For the general notion of limit see Appendix \ref{sec:Limit}.
By a direct system in this section we always mean a direct system
indexed by $\N$.
\begin{defn}
\label{def:system of atlases}Let $\stX$ be a category fibered in
groupoid over $\Z$. A coarse ind-algebraic space for $\stX$ is a
map $\stX\arr X$ to an ind-algebraic space $X$ which is universal
among maps from $\stX$ to an ind-algebraic space and such that, for
all algebraically closed field $K$, the map $\stX(K)/\simeq\arr X(K)$
is bijective.
\end{defn}

\begin{lem}
\label{lem:limit of coarse}Let $\stZ_{*}$ be a direct system of
quasi-compact and quasi-separated algebraic stacks admitting coarse
moduli spaces $\stZ_{n}\arr\overline{\stZ_{n}}$. Then the limit of
those maps $\Delta\arr\overline{\Delta}$ is a coarse ind-algebraic
space.

Assume moreover that the transition maps of $\stZ_{*}$ are finite
and universally injective. Then for all $n\in\N$ and all reduced
rings $B$ the functors $\stZ_{n}(B)\arr\stZ_{n+1}(B)\arr\Delta(B)$
are fully faithful. In particular the maps $\overline{\stZ_{n}}\arr\overline{\stZ_{n+1}}$
are universally injective and, if all $\stZ_{m}$ are DM, $\stZ_{n}\arr\Delta$
preserves the geometric stabilizers.
\end{lem}

\begin{proof}
The first claim follows easily taking into account that, since $\stZ_{n}$
is quasi-compact and quasi-separated, a functor from $\stZ_{n}$ to
an ind-algebraic space factors through an algebraic space and therefore
uniquely through $\overline{\stZ_{n}}$. It is also easy to reduce
the second claim to the case of some $\stZ_{n}$.

Denote by $\psi\colon\stZ_{n}\arr\stZ_{n+1}$ the transition map.
Let $\xi,\eta\in\stZ_{n}(B)$ and $a\colon\psi(\xi)\arr\psi(\eta)$.
Set $\psi(\eta)=\zeta\in\stZ_{n+1}(B)$. If $W$ is the base change
of $\stZ_{n}\arr\stZ_{n+1}$ along $\Spec B\arrdi{\zeta}\stZ_{n+1}$
then $\overline{\xi}=(\xi,\zeta,a),\overline{\eta}=(\eta,\zeta,\id)\in W(B)$.
A lifting of $a$ to an isomorphism $\xi\arr\eta$ is exactly an isomorphism
$\overline{\xi}\arr\overline{\eta}$. Such an isomorphism exists and
it is unique because, since $W\arr\Spec B$ is an homeomorphism and
$B$ is reduced it has at most one section.

Applying the above property when $B$ is an algebraically closed field
we conclude that $\overline{\stZ_{n}}\arr\overline{\stZ_{n+1}}$ is
universally injective. If all $\stZ_{m}$ are DM then the geometric
stabilizers are constant. Since for all algebraically closed field
$K$ the functor $\stZ_{n}(K)\arr\stZ_{n+1}(K)$ is fully faithful
we see that $\stZ_{n}\arr\stZ_{n+1}$ is an isomorphism on geometric
stabilizers.
\end{proof}
\begin{defn}
Given a direct system of stacks $\stY_{*}$, a direct system of smooth
(resp. étale) atlases for $\stY_{*}$ is a direct system of algebraic
spaces $U_{*}$ together with smooth (resp. étale) atlases $U_{i}\arr\stY_{i}$
and $2$-Cartesian diagrams   \[   \begin{tikzpicture}[xscale=2.1,yscale=-1.2]     \node (A0_0) at (0, 0) {$U_i$};     \node (A0_1) at (1, 0) {$U_{i+1}$};     \node (A1_0) at (0, 1) {$\stY_i$};     \node (A1_1) at (1, 1) {$\stY_{i+1}$};     \path (A0_0) edge [->]node [auto] {$\scriptstyle{}$} (A0_1);     \path (A1_0) edge [->]node [auto] {$\scriptstyle{}$} (A1_1);     \path (A0_1) edge [->]node [auto] {$\scriptstyle{}$} (A1_1);     \path (A0_0) edge [->]node [auto] {$\scriptstyle{}$} (A1_0);   \end{tikzpicture}   \] for
all $i\in\N$.
\end{defn}

\begin{lem}
\label{lem:from qcqs stacks to limit} Let $\stY_{*}$ be a direct
system of stacks and $\stX$ be a quasi-compact and quasi-separated
algebraic stack. Then the functor
\[
\varinjlim_{n}\Hom(\stX,\stY_{n})\arr\Hom(\stX,\varinjlim_{n}\stY)
\]
is an equivalence of categories. If the transition maps of $\stY_{*}$
are faithful (resp. fully faithful) so are the transition maps in
the above limit.
\end{lem}

\begin{proof}
Denotes by $\zeta_{\stX}$ the functor in the statement. When $\stX$
is an affine scheme $\zeta_{\stX}$ is an equivalence thanks to \ref{prop:limit of stacks is stack}.
In general there is a smooth atlas $U\arr\stX$ from an affine scheme.
It is easy to see that the functor $\zeta_{\stX}$ is faithful. If
two morphisms become equal in the limit it is enough to pullback to
$U$ and get a finite index for $\zeta_{U}$. By descent this index
will work in general.

The next step is to look at the case when $\stX$ is a quasi-compact
scheme. Using the faithfulness just proved and taking a Zariski covering
of $\stX$ (here one uses that the intersection of two open quasi-compact
subschemes of $\stX$ is again quasi-compact) one proves that $\zeta_{\stX}$
is an equivalence.

Finally using that $\zeta_{U}$, $\zeta_{U\times_{\stX}U}$ and $\zeta_{U\times_{\stX}U\times_{\stX}U}$
are equivalences and using descent one get that $\zeta_{\stX}$ is
an equivalence. The last statement can be proved directly.
\end{proof}
\begin{prop}
\label{lem:descent of ind along torsors}Consider a $2$-Cartesian
diagram   \[   \begin{tikzpicture}[xscale=1.5,yscale=-1.2]     \node (A0_0) at (0, 0) {$\stY$};     \node (A0_1) at (1, 0) {$\Spec k$};     \node (A1_0) at (0, 1) {$\stX$};     \node (A1_1) at (1, 1) {$\Bi G$};     \path (A0_0) edge [->]node [auto] {$\scriptstyle{}$} (A0_1);     \path (A0_0) edge [->]node [auto] {$\scriptstyle{F}$} (A1_0);     \path (A0_1) edge [->]node [auto] {$\scriptstyle{}$} (A1_1);     \path (A1_0) edge [->]node [auto] {$\scriptstyle{}$} (A1_1);   \end{tikzpicture}   \] where
$G$ is a finite group and $\stX$ is a stack over $\Aff/k$. Suppose
that there exists a direct system $\stY_{*}$ of DM stacks of finite
type over $k$ with finite and universally injective transition maps,
affine diagonal, with a direct system of étale atlases $Y_{n}\arr\stY_{n}$
from affine schemes and with an isomorphism $\varinjlim_{n}\stY_{n}\simeq\stY$.
Then there exists a direct system of DM stacks $\stX_{*}$ of finite
type over $k$ with finite and universally injective transition maps,
affine diagonal, with a direct system of étale atlases $X_{n}\arr\stX_{n}$
from affine schemes and with an isomorphism $\varinjlim_{n}\stX_{n}\simeq\stX$.
If all the stacks in $\stY_{*}$ are separated then the stacks $\stX_{*}$
can also be chosen separated. If $Y_{n}\arr\stY_{n}$ is finite and
étale then $X_{n}\arr\stX_{n}$ can also be chosen finite and étale.
\end{prop}

The remaining part of this section is devoted to the proof of the
above Proposition. Its outline is as follows. For some data $\underline{\omega}$,
we define a stack $\shX_{\underline{\omega}}$, and for a suitable
sequence $\underline{\omega_{u}}$, $u\in\N$ of such data, we will
prove that the sequence 
\[
\shX_{\underline{\omega_{0}}}\to\shX_{\underline{\omega_{1}}}\to\cdots
\]
has the desired property. To do this, we reduce the problem to proving
a similar property for the induced sequence $\shZ_{\underline{\omega_{u}}}\cong\shX_{\underline{\omega_{u}}}\times_{\shX}\shY$,
$u\in\N$. Then we describe $\shZ_{\underline{\omega}}$ using fiber
products of simple stacks. Once these are done, it is straightforward
to see the desired properties of $\shZ_{\underline{\omega_{u}}}$,
$u\in\N$. 

We recall that stacks and more generally categories fibered in groupoids
form 2-categories; 1-morphisms are base preserving functors between
them and 2-morphisms are base preserving natural isomorphisms between
functors. In a diagram of stacks, 1-morphisms (functors) are written
as normal thin arrows and 2-morphisms as thick arrows. For instance,\emph{
}in the diagram of categories fibered in groupoids
\[
\xymatrix{A\ar[r]^{f}\ar[d]_{h} & B\ar[d]^{g}\ar@2[dl]^{\lambda}\\
C\ar[r]_{i} & D
}
\]
$A,B,C$ and $D$ are categories fibered in groupoids, $f$, $g$,
$h$ and $i$ are functors and $\lambda$ is a natural isomorphism
$g\circ f\to i\circ h$. We will also say that \emph{$\lambda$ makes
the diagram $2$-commutative}. For a diagram including several 2-morphisms
such as 
\[
\xymatrix{A\ar[r]^{f}\ar[d]_{i} & \ar@2[dl]_{\lambda}B\ar[r]^{g}\ar[d] & \ar@2_{\lambda'}[dl]C\ar[d]^{h}\\
D\ar[r]_{j} & E\ar[r]_{k} & F
}
\]
the \emph{induced natural isomorphism} means that the induced natural
isomorphism of the two outer paths from the upper left to the bottom
right; concretely, in the above diagram, it is the natural isomorphism
$h\circ g\circ f\to k\circ j\circ i$ induced by $\lambda$ and $\lambda'$. 

Let us denote the functor $\stX\arr\Bi G$ in \ref{lem:descent of ind along torsors}
by $Q$. An object of $\shY$ over $T\in\Aff/k$ is identified with
a pair $(\xi,c)$ such that $\xi\in\shX(T)$ and $c$ is a section
of the $G$-torsor $Q(\xi)\arr T$, in other words $c\in Q(\xi)(T)$.
For each $g\in G$, we define an automorphism $\iota_{g}\colon\shY\to\shY$
sending $(x,c)$ to $(x,cg)$. By construction we have $\iota_{1}=\id$
and $\iota_{g}\circ\iota_{h}=\iota_{hg}$ and we interpret those maps
as a map $\iota\colon\stY\times G\arr\stY$. For all $\xi\in\stX$
the map $\iota$ induces the action of $G$ on $Q(\xi)$.
\begin{defn}
We define a category fibered in groupoids $\widetilde{\shX}$ as follows.
An object of $\widetilde{\shX}$ over a scheme $T$ is a tuple $(P,\eta,\mu)$
where $P$ is a $G$-torsor over $T\in\Aff/k$ with a $G$-action
$m_{P}\colon P\times G\to P$, $\eta\colon P\to\shY$ is a morphism
and $\mu$ is a natural isomorphism
\begin{equation}
\xymatrix{P\times G\ar[r]^{m_{P}}\ar[d]_{\eta\times\id_{G}} & \ar@2[dl]_{\mu}P\ar[d]^{\eta}\\
\shY\times G\ar[r]_{\iota} & \shY
}
\label{eq:mu-2}
\end{equation}
such that if $\mu_{g}$ denotes the natural isomorphism induced from
$\mu$ by composing $P\simeq P\times\{g\}\hookrightarrow P\times G$
and $m_{g}\colon P\to P$ denotes the action of $g$, then the diagram
\begin{equation}
\xymatrix{P\ar[r]^{m_{h}}\ar[d]_{\eta} & \ar@2[dl]_{\mu_{h}}P\ar[r]^{m_{g}}\ar[d]_{\eta} & \ar@2[dl]_{\mu_{g}}P\ar[d]_{\eta}\\
\shY\ar[r]_{\iota_{h}} & \shY\ar[r]_{\iota_{g}} & \shY
}
\label{eq:tau-cocycle}
\end{equation}
induces $\mu_{hg}$. A morphism $(P,\eta,\mu)\to(P',\eta',\mu')$
over $T$ is a pair $(\alpha,\beta)$ where $\alpha\colon P\to P'$
is a $G$-equivariant isomorphism over $T$ and $\beta$ is a natural
isomorphism 
\[
\xymatrix{P\ar[rr]^{\alpha}\ar[dr]_{\eta} &  & P'\ar[dl]^{\eta'}\ar@{<=}[dll(0.6)]\sb(0.7){\beta}\\
 & \shY
}
\]
such that  the diagram
\begin{equation}
\xymatrix{P\times G\ar[dd]_{\eta\times\id}\ar[rrr]\ar[dr]^{\alpha\times\id} &  & \ar@2[dl]_{\id} & \ar[dl]^{\alpha}P\ar[dd]^{\eta}\\
 & \ar@2[l]^{\beta^{-1}\times\id}P'\times G\ar[r]\ar[dl]^{\eta'\times\id} & P'\ar[dr]_{\eta'}\ar@2[dl]_{\mu'} & \ar@2[l]\sp(0.4){\beta}\\
\shY\times G\ar[rrr]_{\iota} &  &  & \shY
}
\label{eq:beta-compatible-1}
\end{equation}
induces $\mu$.
\end{defn}

There is a functor $\stX\arr\widetilde{\stX}$: given $\xi\in\stX(T)$
one gets a $G$-torsor $Q(\xi)\to T$ and a morphism $\eta\colon Q(\xi)\to\shY$
and using the Cartesian diagram relating $\stX$ and $\stY$ we also
get a natural transformation as above. The following is a generalization
of \cite[Theorem 4.1]{Romagny2005} for stacks without geometric properties.
\begin{lem}
\label{lem:X iso to tilde X}The functor $\stX\arr\widetilde{\stX}$
is an equivalence. 
\end{lem}

\begin{proof}
The forgetful functor $\widetilde{\stX}\arr\Bi G$ composed with the
functor in the statement is $Q\colon\stX\arr\Bi G$. Since $\stX$
and $\widetilde{\stX}$ are stacks, it is enough to show that the
functor $\Xi\colon\stY\arr\widetilde{\stY}=\widetilde{\stX}\times_{\Bi G}\Spec k$
is an equivalence.

An object of the stack $\widetilde{\shY}$ can be regarded as a pair
$(\eta,\mu)$ such that $\eta\colon T\times G\to\shY$ is a morphism
and $\mu$ is a natural isomorphism as in \eqref{eq:mu-2} with $m_{P}\colon P\times G\to P$
replaced with $\id_{T}\times m_{G}\colon T\times G\times G\to T\times G$,
where $m_{G}$ is the multiplication of $G$. A morphism $(\eta,\mu)\to(\eta',\mu')$
in $\widetilde{\shY}(T)$ is a natural isomorphism $\beta\colon\eta\to\eta'$
satisfying the same compatibility as in \eqref{eq:beta-compatible-1}
where $P$ and $P'$ are replaced with $T\times G$ and $\alpha$
is replaced with $\id_{T\times G}$. The functor $\Xi$ sends an object
$\rho\in\shY(T)$ to $(\widetilde{\rho}\colon T\times G\to\shY,\mu)$
such that $\widetilde{\rho}|_{T\times\{g\}}=\iota_{g}\circ\rho$ and
$\mu$ is the canonical natural isomorphism, and a morphism $\gamma\colon\rho\to\rho'$
to $(\id_{T\times G},\widetilde{\gamma})$ where $\widetilde{\gamma}|_{T\times\{g\}}=\iota_{g}(\gamma)$.
One also gets a functor $\Lambda\colon\widetilde{\stY}\arr\stY$ by
composing with the identity of $G$ and it is easy to see that $\Lambda\circ\Xi=\id$.
The compatibilities defining the objects of $\widetilde{\stY}$ also
allow to define an isomorphism $\Xi\circ\Lambda\simeq\id$. 
\end{proof}
Set $\delta_{u}\colon\stY_{u}\arr\stY$ for the structure maps and
$\delta_{u,v}\colon\stY_{u}\arr\stY_{v}$ for the transition maps
for all $u\leq v\in\N$. Given $w\ge v\geq u\in\N$ we denote by $\shR(u,v,w)$
the collection of tuples $(\omega,\omega',\theta,\theta')$ forming
$2$-commutative diagrams:
\begin{equation}
\xymatrix{\shY_{u}\times G\ar[r]^{\omega}\ar[d]_{\delta_{u}\times\id} & \shY_{v}\ar[d]^{\delta_{v}}\ar@2[dl]_{\theta}\\
\shY\times G\ar[r]_{\iota} & \shY
}
\label{eq:theta}
\end{equation}
\begin{equation}
\xymatrix{\shY_{v}\times G\ar[r]^{\omega'}\ar[d]_{\delta_{v}\times\id} & \shY_{w}\ar[d]^{\delta_{w}}\ar@2[dl]_{\theta'}\\
\shY\times G\ar[r]_{\iota} & \shY
}
\label{eq:theta'}
\end{equation}
We also require the existence of a natural isomorphism 
\begin{equation}
\xymatrix{\shY_{u}\times G\ar[r]^{\omega}\ar[d]_{\delta_{u,v}\times\id} & \shY_{v}\ar[d]^{\delta_{v,w}}\ar@2[dl]_{\zeta}\\
\shY_{v}\times G\ar[r]_{\omega'} & \shY_{w}
}
\label{eq:zeta}
\end{equation}
compatible with $\theta$ and $\theta'$ and, for $g,h\in G$, the
existence of a natural isomorphism $\lambda_{h,g}$ 
\begin{equation}
\xymatrix{\shY_{u}\ar[r]^{\omega_{h}}\ar[d]_{\omega_{hg}} & \shY_{v}\ar[d]^{\omega'_{g}}\ar@2[dl]_{\lambda_{h,g}}\\
\shY_{v}\ar[r]_{\delta_{v,w}} & \shY_{w}
}
\label{eq:lambda}
\end{equation}
such that the natural isomorphism induced by
\begin{equation}
\xymatrix{\shY_{u}\ar[drr]^{\omega_{h}}\ar[ddd]_{\delta_{u}}\ar[rrrr]^{\delta_{v,w}(\omega_{hg})} &  & {}\ar@2[d]^{\lambda_{h,g}^{-1}} &  & \shY_{w}\ar[ddd]^{\delta_{w}}\\
 &  & \ar@2[dll]_{\theta_{h}}\shY_{v}\ar[urr]^{\omega'_{g}}\ar[d]_{\delta_{v}} &  & \ar@2[dll]_{\theta'_{g}}\\
{} &  & \shY\ar[drr]_{\iota_{g}}\ar@2[d]^{\text{id}}\\
\shY\ar[urr]_{\iota_{h}}\ar[rrrr]_{\iota_{hg}} &  & {} &  & \shY
}
\label{eq:theta-cocycle}
\end{equation}
coincides with the one induced by $\theta_{hg}$ and $\delta_{v,w}$.
Since the transition maps of $\stY_{*}$ are faithful, the functor
$\Hom(\shY_{u},\shY_{w})\to\Hom(\shY_{u},\shY)$ is faithful as well,
which means that natural transformations $\zeta$ and $\lambda_{g,h}$
are uniquely determined.
\begin{defn}
For $\underline{\omega}=(\omega,\omega',\theta,\theta')\in\shR(u,v,w)$,
we define the following category fibered in groupoids $\stX_{\underline{\omega}}$
as follows. 

An object of $\stX_{\underline{\omega}}$ over $T$ is a triple $(P,\eta,\mu)$
of a $G$-torsor $P$ over $T$, $\eta\colon P\arr\stY_{u}$ and a
natural isomorphism $\mu$ making the diagram
\[
\xymatrix{P\times G\ar[r]^{m_{P}}\ar[d]_{\eta\times\id} & P\ar[d]^{\delta_{u,v}(\eta)}\ar@2[dl]_{\mu}\\
\shY_{u}\times G\ar[r]_{\omega} & \shY_{v}
}
\]
$2$-commutative such that the diagram
\begin{equation}
\xymatrix{P\ar[rr]^{m_{h}}\ar[dd]_{\eta} &  & \ar@2[dll]_{\mu_{h}}P\ar[rr]^{m_{g}}\ar[d]_{\delta_{u,v}(\eta)} &  & \ar@2[dll]_{\delta_{v,w}(\mu_{g})}P\ar[dd]^{\delta_{u,w}(\eta)}\\
 &  & \shY_{v}\ar[drr]_{\omega'_{g}}\ar@2[d]^{\lambda_{h,g}}\\
\shY_{u}\ar[urr]_{\omega_{h}}\ar[rrrr]_{\delta_{v,w}(\omega_{hg})} &  & {} &  & \shY_{w}
}
\label{eq:mu-cocycle-1}
\end{equation}
induces $\delta_{v,w}(\mu_{hg})$.

A morphism $(P,\eta,\mu)\arr(P',\eta',\mu')$ in $\stX_{\underline{\omega}}(T)$
is a pair $(\alpha,\beta)$ where $\alpha\colon P\arr P'$ is a $G$-equivariant
isomorphism over $T$ and $\beta$ is a natural isomorphism making
the following diagram $2$-commutative 
\[
\xymatrix{P\ar[rr]^{\alpha}\ar[dr]_{\eta} &  & P'\ar[dl]^{\eta'}\ar@{<=}[dll(0.6)]\sb(0.7){\beta}\\
 & \shY_{u}
}
\]
such that the diagram
\begin{equation}
\xymatrix{P\times G\ar[dd]_{\eta\times\id}\ar[rrr]\ar[dr]^{\alpha\times\id} &  & \ar@2[dl]_{\mathrm{id}} & \ar[dl]^{\alpha}P\ar[dd]^{\delta_{u,v}(\eta)}\\
 & \ar@2[l]^{\beta^{-1}\times\id}P'\times G\ar[r]\ar[dl]^{\eta'\times\id} & P'\ar[dr]_{\delta_{u,v}(\eta')}\ar@2[dl]_{\mu'} & \ar@2[l]\sp(0.4){\delta_{u,v}(\beta)}\\
\shY_{u}\times G\ar[rrr]_{\omega} &  &  & \shY_{v}
}
\label{eq:beta-compatible}
\end{equation}
induces the natural isomorphism $\mu$. 
\end{defn}

By \ref{lem:from qcqs stacks to limit} for all algebraic stacks $\stX$
the functor $\varinjlim_{w}\Hom(\stX,\shY_{w})\to\Hom(\stX,\shY)$
is an equivalence. This allows us to choose increasing functions $v,w\colon\N\arr\N$
such that $u\leq v(u)\le w(u)$ and $\underline{\omega_{u}}=(\omega_{u},\omega_{u}',\theta_{u},\theta_{u}')\in\shR(u,v(u),w(u))$,
so that $\stX_{\underline{\omega_{u}}}$ is defined, for all $u\in\N$.
Moreover we can assume there exist natural isomorphisms $\kappa$
and $\kappa'$ 
\[
\xymatrix{\shY_{u}\times G\ar[r]^{\omega_{u}}\ar[d]_{\delta_{u,u+1}\times\id} & \shY_{v(u)}\ar[d]^{\delta_{v(u),v(u+1)}}\ar@2[dl]_{\kappa}\\
\shY_{u+1}\times G\ar[r]_{\omega_{u+1}} & \shY_{v(u+1)}
}
\]
\[
\xymatrix{\shY_{v(u)}\times G\ar[r]^{\omega'_{u}}\ar[d]_{\delta_{v(u),v(u+1)}\times\id} & \shY_{w(u)}\ar[d]^{\delta_{w(u),w(u+1)}}\ar@2[dl]_{\kappa'}\\
\shY_{v(u+1)}\times G\ar[r]_{\omega'_{u+1}} & \shY_{w(u+1)}
}
\]
such that $\delta_{v(u+1)}(\kappa)$ is induced from $\theta_{u}$
and $\theta_{u+1}$ and $\delta_{w(u+1)}(\kappa')$ is induced from
$\theta'_{u}$ and $\theta'_{u+1}$. Again, since the transition maps
of $\stY_{*}$ are faithful, natural transformations $\kappa$ and
$\kappa'$ are uniquely determined.

For each $u\in\N$, there exist canonical functors $\shX_{\underline{\omega_{u}}}\to\shX_{\underline{\omega_{u+1}}}$
and $\shX_{\underline{\omega_{u}}}\to\widetilde{\shX}$, which lead
to a functor
\[
\varinjlim_{u}\stX_{\underline{\omega_{u}}}\arr\widetilde{\stX}.
\]

\begin{prop}
\label{prop:equivalence X omega to tilde X }The functor $\varinjlim_{u}\stX_{\underline{\omega_{u}}}\arr\widetilde{\stX}$
is an equivalence.
\end{prop}

\begin{proof}
By definition, every object and every morphism of $\widetilde{\shX}$
come from ones of $\shX_{\underline{\omega_{u}}}$ for some $u$.
Namely the above functor is essentially surjective and full. To see
the faithfulness, we take objects $(P,\eta,\mu)$, $(P',\eta',\mu')$
of $\shX_{\underline{\omega_{u}}}(T)$ and their images $(P,\eta_{\infty},\mu_{\infty})$,
$(P,\eta'_{\infty},\mu_{\infty}')$ in $\widetilde{\shX}(T)$. The
map
\[
\Hom_{\shX_{\underline{\omega_{u}}}(T)}((P,\eta,\mu),(P',\eta',\mu'))\to\Hom_{\widetilde{\shX}(T)}((P,\eta_{\infty},\mu_{\infty}),(P',\eta'_{\infty},\mu'_{\infty}))
\]
is compatible to projections to the set $\Iso_{T}^{G}(P,P')$ of $G$-equivariant
isomorphisms over $T$. The fibers over $\alpha\in\Iso_{T}^{G}(P,P')$
are respectively identified with subsets of $\Hom_{\shY_{u}(P)}(\eta,\eta'\circ\alpha)$
and of $\Hom_{\shY(P)}(\eta_{\infty},\eta'_{\infty}\circ\alpha)$.
Since $\stY(P)$ is the limit of the categories $\stY_{u}(P)$ by
\ref{lem:from qcqs stacks to limit} we get the faithfulness.
\end{proof}
\begin{defn}
For $\underline{\omega}=(\omega,\omega',\theta,\theta')\in\shR(u,v,w)$,
we define $\stZ_{\underline{\omega}}$ as the stack of pairs $(\eta,\mu)$
where $\eta\colon T\times G\arr\stY_{u}$ is a morphism and $\mu$
is a natural isomorphism making the diagram 
\begin{equation}
\xymatrix{T\times G\times G\ar[d]_{\eta\times\id}\ar[r]\sp(0.6){\id\times m_{G}} & T\times G\ar[d]^{\delta_{u,v}(\eta)}\ar@2[dl]_{\mu}\\
\shY_{u}\times G\ar[r]_{\omega} & \shY_{v}
}
\label{eq:mu}
\end{equation}
$2$-commutative and such that
\begin{equation}
\xymatrix{T\times G\ar[rr]^{\id\times h}\ar[dd]_{\eta} &  & \ar@2[dll]_{\mu_{h}}T\times G\ar[rr]^{\id\times g}\ar[d]_{\delta_{u,v}(\eta)} &  & \ar@2[dll]_{\delta_{v,w}(\mu_{g})}T\times G\ar[dd]^{\delta_{u,w}(\eta)}\\
{} &  & \shY_{v}\ar[drr]^{\omega'_{g}}\ar@2[d]^{\lambda_{h,g}}\\
\shY_{u}\ar[urr]^{\omega_{h}}\ar[rrrr]_{\delta_{v,w}(\omega_{hg})} &  & {} &  & \shY_{w}
}
\label{eq:mu-cocycle}
\end{equation}
induces $\delta_{v,w}(\tau_{hg})$. 

A morphism $(\eta,\mu)\to(\eta',\mu')$ in $\shZ_{\underline{\omega}}$
is a natural isomorphism $\beta\colon\eta\to\eta'$,
\[
\xymatrix{\ar@/_{25pt}/[d]_{\eta}="a"T\times G\ar@/^{25pt}/[d]^{\eta'}="b"\\
\shY_{u}\ar@2"a";"b"^{\beta}
}
\]
such that the diagram
\begin{equation}
\xymatrix{\ar@/_{25pt}/[d]_{\eta}="a"T\times G\times G\ar@/^{25pt}/[d]^{\eta'}="b"\ar[rrr]^{\id\times m_{G}} &  & \ar@2[dl]_{\mu'} & \ar@/_{25pt}/[d]_{\delta_{u,v}(\eta')}="c"T\times G\ar@/^{25pt}/[d]^{\delta_{u,v}(\eta)}="d"\\
\shY_{u}\times G\ar[rrr]_{\omega}\ar@2"b";"a"_{\beta^{-1}\times\id} & {} &  & \shY_{v}\ar@2"d";"c"_{\delta_{u,v}(\beta)}
}
\label{eq:beta-condition}
\end{equation}
induces the natural isomorphism $\mu$. 
\end{defn}

\begin{lem}
\label{lem:Z iso to a fiber prod}There exists a natural equivalence
$\shZ_{\underline{\omega}}\simeq\shX_{\underline{\omega}}\times_{\shX}\shY$.
Here the morphism $\shX_{\underline{\omega}}\longrightarrow\shX$
implicit in the fiber product is the composite of the morphism $\shX_{\underline{\omega}}\longrightarrow\widetilde{\shX}$
and a quasi-inverse $\widetilde{\shX}\stackrel{\sim}{\longrightarrow}\shX$
of the equivalence in \ref{lem:X iso to tilde X}.
\end{lem}

\begin{proof}
Set $\widetilde{\shZ}_{\underline{\omega}}=\stX_{\underline{\omega}}\times_{\stX}\stY$.
We may identify an object of $\widetilde{\shZ}_{\underline{\omega}}(T)$
with a tuple $(P,\eta,\mu,s)$ such that $(P,\eta,\mu)$ is an object
of $\shX_{\underline{\omega}}(T)$ and $s$ is a section of $P\to T$.
A morphism $(P,\eta,\mu,s)\to(P',\eta',\mu',s')$ is identified with
a morphism $(\alpha,\beta)\colon(P,\eta,\mu)\to(P',\eta',\mu')$ in
$\shX_{\underline{\omega}}$ satisfying $\alpha\circ s=s'$. The section
$s$ induces a $G$-equivariant isomorphism $T\times G\arr P$. Identifying
$P$ with $T\times G$ through this isomorphism, we see that $\widetilde{\shZ}_{\underline{\omega}}$
is equivalent to $\shZ_{\underline{\omega}}$.
\end{proof}
Notice that if $\stU\arr\stV$ is a $G$-torsor then $\stV$ has affine
diagonal (resp. is separated) if and only if $\stU$ has the same
property. The ``only if'' part is clear. The ``if'' part follows
because $\Bi G$ is separated, descent and the $2$-Cartesian diagrams
  \[   \begin{tikzpicture}[xscale=1.8,yscale=-1.2]     \node (A0_0) at (0, 0) {$\stU$};     \node (A0_1) at (1, 0) {$\stU\times\stU$};     \node (A0_2) at (2, 0) {$\Spec k$};     \node (A1_0) at (0, 1) {$\stV$};     \node (A1_1) at (1, 1) {$\stV\times_{\Bi G}\stV$};     \node (A1_2) at (2, 1) {$\Bi G$};     \path (A0_0) edge [->]node [auto] {$\scriptstyle{}$} (A0_1);     \path (A0_1) edge [->]node [auto] {$\scriptstyle{}$} (A1_1);     \path (A1_0) edge [->]node [auto] {$\scriptstyle{}$} (A1_1);     \path (A1_1) edge [->]node [auto] {$\scriptstyle{}$} (A1_2);     \path (A0_2) edge [->]node [auto] {$\scriptstyle{}$} (A1_2);     \path (A0_0) edge [->]node [auto] {$\scriptstyle{}$} (A1_0);     \path (A0_1) edge [->]node [auto] {$\scriptstyle{}$} (A0_2);   \end{tikzpicture}   \]  

From this remark and from \ref{lem:X iso to tilde X}, \ref{prop:equivalence X omega to tilde X }
and \ref{lem:Z iso to a fiber prod}, the proof of \ref{lem:descent of ind along torsors}
reduces to:
\begin{lem}
\label{prop:reduction to Z}The stacks $\stZ_{\underline{\omega_{*}}}$
form a direct system of DM stacks of finite type over $k$ with affine
diagonal, with finite and universally injective transition maps, and
with a direct system of étale atlases $Z_{*}\arr\stZ_{\underline{\omega_{*}}}$
from affine schemes. Moreover if all $\stY_{*}$ are separated so
are the $\stZ_{\underline{\omega_{*}}}$ and if $Y_{n}\arr\stY_{n}$
is finite and étale then $Z_{*}\arr\stZ_{\underline{\omega_{*}}}$
can be chosen to be finite and étale.
\end{lem}

To prove this lemma, we will describe $\stZ_{\underline{\omega}}$
by using fiber products of simpler stacks. In what follows, for a
stack $\shK$ and a finite set $I$, we denote by $\shK^{I}$ the
product $\prod_{i\in I}\shK$ and identify its objects over a scheme
$T$ with the morphisms $T\times I=\sqcup_{i}T\to\shK$.

Let $\shW_{\underline{\omega}}$ be the stack of pairs $(\eta,\mu)$
where $\eta\colon T\times G\arr\stY_{u}$ and $\mu$ is a natural
isomorphism as in \eqref{eq:mu}, but not necessarily satisfying the
compatibility imposed on objects of $\stZ_{\underline{\omega}}$.
\begin{rem}
\label{rem:objs plus iso}  Let $F,G\colon\stW_{1}\arr\stW_{0}$ be
two maps of stacks and denote by $\stW_{2}$ the stack of pairs $(w,\zeta)$
were $w\in\stW_{1}(T)$ and $\zeta\colon G(w)\arr F(w)$ is an isomorphism
in $\stW_{0}(T)$. Then there is a $2$-Cartesian diagram   \[   \begin{tikzpicture}[xscale=2,yscale=-1.2]     \node (A0_0) at (0, 0) {$\stW_2$};     \node (A0_1) at (1, 0) {$\stW_1$};     \node (A1_0) at (0, 1) {$\stW_1$};     \node (A1_1) at (1, 1) {$\stW_1\times \stW_0$};     \path (A0_0) edge [->]node [auto] {$\scriptstyle{p}$} (A0_1);     \path (A1_0) edge [->]node [auto] {$\scriptstyle{\Gamma_F}$} (A1_1);     \path (A0_1) edge [->]node [auto] {$\scriptstyle{\Gamma_G}$} (A1_1);     \path (A0_0) edge [->]node [auto] {$\scriptstyle{p}$} (A1_0);   \end{tikzpicture}   \] where
$\Gamma_{*}$ denotes the graph and $p$ the projection. Notice that
the sheaf of isomorphisms of an object of a fiber product can be expressed
as fiber products of the sheaves of isomorphisms of its factors. In
particular, if $\stW_{0},\stW_{1}$ have affine diagonals, then $\stW_{2}$
has affine diagonal. If $\stW_{1}$ has affine diagonal and $F$ is
affine then $p$ is also affine. This is because the graph $\Gamma_{F}$
can be factors as the diagonal $\stW_{1}\arr\stW_{1}\times\stW_{1}$
followed by $\id\times F\colon\stW_{1}\times\stW_{1}\arr\stW_{1}\times\stW_{0}$.
\end{rem}

This remark particularly gives:
\begin{lem}
\label{lem:W as a fiber product}Let $\Phi\colon\shY_{u}^{G}\to\shY_{v}^{G\times G}$
be the morphism sending $\eta\colon T\times G\to\shY_{u}$ to the
composition
\[
\Phi(\eta)\colon T\times G\times G\xrightarrow{\id\times m_{G}}T\times G\xrightarrow{\eta}\shY_{u}\xrightarrow{\delta_{u,v}}\shY_{v}
\]
and let $\Psi\colon\shY_{u}^{G}\to\shY_{v}^{G\times G}$ be the morphism
sending $\eta\colon T\times G\to\shY_{u}$ to the composition
\[
\Psi(\eta)\colon T\times G\times G\xrightarrow{\eta\times\id}\shY_{u}\times G\xrightarrow{\omega}\shY_{v}.
\]
Let $\Gamma_{\Phi},\Gamma_{\Psi}\colon\shY_{u}^{G}\to\shY_{u}^{G}\times\shY_{v}^{G\times G}$
be their respective graph morphisms. Then
\[
\shW_{\underline{\omega}}\simeq\shY_{u}^{G}\times_{\Gamma_{\Phi},\shY_{u}^{G}\times\shY_{v}^{G\times G},\Gamma_{\Psi}}\shY_{u}^{G}.
\]
\end{lem}

Let $(\eta,\mu)\in\stW_{\underline{\omega}}(T)$. In the two diagrams,
\begin{equation}
\xymatrix{T\times G\times G\times G\ar[rr]^{\id_{T}\times m_{G}\times\id_{G}}\ar[d]_{\eta} &  & \ar@2[dll]_{\mu\times\id}T\times G\times G\ar[rr]^{\id_{T}\times m_{G}}\ar[d]_{\delta_{u,v}(\eta)} &  & \ar@2[dll]_{\delta_{v,w}(\mu)}T\times G\ar[d]^{\delta_{u,w}(\eta)}\\
\shY_{u}\times G\times G\ar[rr]_{\omega\times\id_{G}} &  & \shY_{v}\times G\ar[rr]_{\omega'} &  & \shY_{w}
}
\label{eq:mu-cocycle-1-2}
\end{equation}
and
\begin{equation}
\xymatrix{T\times G\times G\times G\ar[rr]^{\id_{T\times G}\times m_{G}}\ar[d]_{\eta} &  & \ar@2[dll]_{\textrm{canonical}}T\times G\times G\ar[rr]^{\id_{T}\times m_{G}}\ar[d]_{\delta_{u,v}(\eta)} &  & \ar@2[dll]_{\delta_{v,w}(\mu)}T\times G\ar[d]^{\delta_{u,w}(\eta)}\\
\shY_{u}\times G\times G\ar[rr]_{\delta_{u,v}\times m_{G}} &  & \shY_{v}\times G\ar[rr]_{\omega'} &  & \shY_{w}
}
\label{eq:mu-cocycle-1-1-1}
\end{equation}
the paths from $T\times G\times G\times G$ to $\shY_{w}$ through
the upper right corner are identical; we denote this morphism $T\times G\times G\times G\to\shY_{w}$
by $r(\eta)$. As for the paths through the left bottom corner, there
is a natural isomorphism between them given by $\zeta$ \eqref{eq:zeta}
and $\lambda_{h,g}$ \eqref{eq:lambda}. We identify the two lower
paths through this natural isomorphism and denote it by $s(\eta)$.
We denote the natural isomorphism $r(\eta)\to s(\eta)$ induced from
the former diagram by $\alpha(\mu)$ and the one from the latter diagram
by $\beta(\eta)$. The compatibility \eqref{eq:mu-cocycle} is nothing
but $\alpha(\mu)=\beta(\mu)$. 

For a stack $\shK$, we denote by $I(\shK)$ its inertia stack. An
object of $I(\shK)$ is a pair $(x,\alpha)$ where $x$ is an object
of $\shK$ and $\alpha$ is an automorphism of $x$. There is an equivalence
$I(\shK)\simeq\shK\times_{\Delta,\shK\times\shK,\Delta}\shK$. We
have the forgetting morphism $I(\shK)\to\shK$, which has the section
$\shK\to I(\shK)$, $x\mapsto(x,\id)$. If $\shK$ is a DM stack of
finite type with finite diagonal, then $I(\shK)\to\shK$ is finite
and unramified and $\shK\to I(\shK)$ is a closed immersion.
\begin{lem}
\label{lem:Z as a fiber product}Let $\Theta,\Lambda\colon\shW_{\underline{\omega}}\to I(\shY_{w}^{G\times G\times G})$
be the functors sending an object $(\eta,\mu)$ of $\shW_{\underline{\omega}}$
to $(r(\eta),\beta(\mu)^{-1}\circ\alpha(\mu))$ and to $(r(\eta),\id)$
respectively. Consider also the functor $\shZ_{\underline{\omega}}\to\shY_{w}^{G\times G\times G}$
sending $(\eta,\mu)$ to $r(\eta)$. Then 
\[
\shZ_{\underline{\omega}}\times_{\shY_{w}^{G\times G\times G}}I(\shY_{w}^{G\times G\times G})\simeq\shW_{\underline{\omega}}\times_{\Gamma_{\Theta},\shW_{\underline{\omega}}\times I(\shY_{w}^{G\times G\times G}),\Gamma_{\Lambda}}\shW_{\underline{\omega}}.
\]
\end{lem}

\begin{proof}
From \ref{rem:objs plus iso}, the right hand side is regarded as
the stack of pairs $((\eta,\mu),\epsilon)$ such that $(\eta,\mu)$
is an object of $\shW_{\underline{\omega}}$ and $\epsilon$ is a
natural isomorphism $\Theta((\eta,\mu))\to\Lambda((\eta,\mu))$. Thus
$\epsilon$ is an isomorphism $r(\eta)\to r(\eta)$ making the diagram
\[
\xymatrix{r(\eta)\ar[d]_{\beta(\mu)^{-1}\circ\alpha(\mu)}\ar[r]^{\epsilon} & r(\eta)\ar[d]^{\id}\\
r(\eta)\ar[r]_{\epsilon} & r(\eta)
}
\]
commutative. Therefore $\beta(\mu)^{-1}\circ\alpha(\mu)=\id$, equivalently,
the compatibility \eqref{eq:mu-cocycle} holds, and $\epsilon$ can
be an arbitrary automorphism of $r(\eta)$. This shows the equivalence
of the lemma.
\end{proof}
\begin{lem}
\label{lem:transition Z}The stack $\shZ_{\underline{\omega}}$ is
a DM stack of finite type with affine diagonal and it is separated
if all the $\stY_{*}$ are separated. Moreover, for functions $v,w\colon\N\to\N$
and $\underline{\omega_{u}}\in\shR(u,v(u),w(u))$, $u\in\N$, the
morphism $\shZ_{\underline{\omega_{u}}}\to\shZ_{\underline{\omega_{u+1}}}$
is finite and universally injective.
\end{lem}

\begin{proof}
If $\shU$, $\shV$ and $\shW$ are DM stacks of finite type with
affine (resp. finite) diagonals, then so is $\shU\times_{\shW}\shV$.
Indeed, $\shU\times\shV$ is a DM stack of finite type with affine
(resp. finite) diagonal and $\shU\times_{\shW}\shV\to\shU\times\shV$
is an affine (resp. finite) morphism since it is a base change of
the diagonal $\shW\to\shW\times\shW$. Hence $\shU\times_{\shW}\shV$
is a DM stack of finite type with affine (resp. finite) diagonal.

From \ref{lem:W as a fiber product} and \ref{lem:Z as a fiber product},
$\shZ_{\underline{\omega}}\times_{\shY_{w}^{G\times G\times G}}I(\shY_{w}^{G\times G\times G})$
is a DM stack of finite type with affine (resp. finite, provided that
all $\stY_{*}$ are separated) diagonal. Since the section $\shY_{w}^{G\times G\times G}\to I(\shY_{w}^{G\times G\times G})$
is a closed immersion, the same conclusion holds for $\shZ_{\underline{\omega}}$.

From \ref{rem:product of universally injective finite}, we can conclude
that the morphism $I(\shY_{w(u)}^{G\times G\times G})\to I(\shY_{w(u+1)}^{G\times G\times G})$
is finite and universally injective. From \ref{rem:product of universally injective finite},
\ref{lem:W as a fiber product} and \ref{lem:Z as a fiber product},
\[
\shZ_{\underline{\omega_{u}}}\times_{\shY_{w(u)}^{G\times G\times G}}I(\shY_{w(u)}^{G\times G\times G})\to\shZ_{\underline{\omega_{u+1}}}\times_{\shY_{w(u+1)}^{G\times G\times G}}I(\shY_{w(u+1)}^{G\times G\times G})
\]
is finite and universally injective, and so is the composition
\[
\shZ_{\underline{\omega_{u}}}\to\shZ_{\underline{\omega_{u}}}\times_{\shY_{w(u)}^{G\times G\times G}}I(\shY_{w(u)}^{G\times G\times G})\to\shZ_{\underline{\omega_{u+1}}}\times_{\shY_{w(u+1)}^{G\times G\times G}}I(\shY_{w(u+1)}^{G\times G\times G}).
\]
The morphism $\shZ_{\underline{\omega_{u}}}\to\shZ_{\underline{\omega_{u+1}}}$
factorizes this and hence is finite and universally injective.
\end{proof}
If $\shU\to\shV$ is a map of stacks and $\stV$ has affine diagonal
then $\shU\times_{\shV}\shU\to\shU\times\shU$ is affine, because
it is the base change of $\Delta\colon\shV\to\shV\times\shV$ along
$\shU\times\shU\to\shV\times\shV$. Therefore the morphism 
\[
\shZ_{\underline{\omega_{u}}}\to\shW_{\underline{\omega_{u}}}\times\shW_{\underline{\omega_{u}}}\to(\shY_{u}^{G}\times\shY_{u}^{G})\times(\shY_{u}^{G}\times\shY_{u}^{G})
\]
induced from equivalences in \ref{lem:W as a fiber product} and \ref{lem:Z as a fiber product}
is affine. Pulling back a direct system of étale atlases for $(\shY_{u}^{G}\times\shY_{u}^{G})\times(\shY_{u}^{G}\times\shY_{u}^{G})$
to $\shZ_{\underline{\omega_{u}}}$, we obtain a system of atlases
as in \ref{prop:reduction to Z}, which completes the proof of \ref{prop:reduction to Z}
and the one of \ref{lem:descent of ind along torsors}.

\section{\label{sec:formal torsors}The stack of formal $G$-torsors}

We fix a field $k$ and an étale group scheme $G$ over $k$. In this
section we will introduce and study the stack of formal $G$-torsors.
\begin{defn}
We denote by $\Delta_{G}$ the category fibered in groupoids over
$\Aff/k$ whose objects over $B$ are $\Delta_{G}(B)=\Bi G(B((t)))$.
\end{defn}

\begin{rem}
\label{rem:base change for Delta} By construction we have that if
$k'/k$ is a field extension then $\Delta_{G}\times_{k}k'\simeq\Delta_{G\times_{k}k'}$.
\end{rem}

\begin{cor}
\label{cor:Delta G is a prestack} The fiber category $\Delta_{G}$
is a pre-stack in the fpqc topology.
\end{cor}

\begin{proof}
{} Let $D_{1},D_{2}\in\Delta_{G}(B)=\Bi G(B((t)))$. We must show that
\[
I=\Isosh_{\Delta_{G}}(D_{1},D_{2})\colon\Aff/B\arr\sets
\]
 is an fpqc sheaf. By \ref{lem:Hom sheaves}, $\Homsh_{B((t))}(D_{1},D_{2})$
is a sheaf, so that in particular $I$ is separated. Thus we must
show that if $B\arr C$ is an fpqc covering, $\phi\colon D_{1}\arr D_{2}$
is a map and $\phi\otimes C((t))$ is a $G$-equivariant morphisms
of $C((t))$-algebras, then $\phi$ is also $G$-equivariant. But
this is obvious since $D_{i}$ is a subset of $D_{i}\otimes C((t))$. 
\end{proof}
\begin{defn}
\label{def:Frobenius}If $\stX$ is a category fibered in groupoids
over $\F_{p}$ then its Frobenius $F_{\stX}\colon\stX\arr\stX$ is
the functor mapping $\xi\in\stX(B)$ to $F_{B}^{*}(\xi)\in\stX(B)$,
where $F_{B}\colon B\arr B$ is the absolute Frobenius. The Frobenius
is $\F_{p}$-linear, natural in $\stX$ and coincides with the usual
Frobenius if $\stX$ is a scheme.

A category fibered in groupoid $\stX$ over $\F_{p}$ is called perfect
if the Frobenius $F_{\stX}\colon\stX\arr\stX$ is an equivalence.
\end{defn}

\begin{example}
As a consequence of \ref{rem:BG insensible to universal homeo, finite}
DM stacks étale over a perfect field are perfect.
\end{example}

We have the following basic property of $\Delta_{G}$, although we
will not use it later.
\begin{prop}
\label{prop:DeltaG is perfect} If $k$ is perfect the fiber category
$\Delta_{G}$ is perfect.
\end{prop}

\begin{proof}
By \ref{rem:BG insensible to universal homeo, finite} the functor
$\Bi G(B((t)))\arr\Bi G(B((t)))$ induced by the Frobenius $F_{B}\colon B\arr B$
is an equivalence: the $p$-th powers of elements in $B((t))$ are
in the image of $F_{B}\colon B((t))\arr B((t))$ and therefore the
spectrum of this map is integral and a universal homeomorphism. 
\end{proof}
Another example of a perfect object that will be used later is the
following:
\begin{defn}
\label{def:X^infty}If $X$ is a functor $\Aff/\F_{p}\arr\sets$ we
denote by $X^{\infty}$ the direct limit of the direct system of Frobenius
morphisms $X\arrdi FX\arrdi F\cdots$.
\end{defn}

Notice that if $X$ is a $k$-pre-sheaf then $X^{\infty}$ does not
necessarily have a $k$-structure unless $k$ is perfect.
\begin{prop}
\label{prop:fiber product modding out by a central subgroup} Let
$H$ be a central subgroup of $G$. Then the equivalence $\Bi G\times\Bi H\arr\Bi G\times_{\Bi(G/H)}\Bi G$
of \ref{lem:central subgroups and torsors} induces an equivalence
$\Delta_{G}\times\Delta_{H}\arr\Delta_{G}\times_{\Delta_{G/H}}\Delta_{G}$.
If $\stX$ is a fibered category with a map $\stX\arr\Delta_{G}$
then we have a $2$-Cartesian diagram   \[   \begin{tikzpicture}[xscale=2.1,yscale=-1.2]     \node (A0_0) at (0, 0) {$\stX\times \Delta_H$};     \node (A0_1) at (1, 0) {$\Delta_G$};     \node (A1_0) at (0, 1) {$\stX$};     \node (A1_1) at (1, 1) {$\Delta_{G/H}$};     \path (A0_0) edge [->]node [auto] {$\scriptstyle{\alpha}$} (A0_1);     \path (A1_0) edge [->]node [auto] {$\scriptstyle{}$} (A1_1);     \path (A0_1) edge [->]node [auto] {$\scriptstyle{}$} (A1_1);     \path (A0_0) edge [->]node [auto] {$\scriptstyle{\pr_1}$} (A1_0);   \end{tikzpicture}   \] where
$\alpha$ is given by $\stX\times\Delta_{H}\arr\Delta_{G}\times\Delta_{H}=\Delta_{G\times H}\arr\Delta_{G}$
and the last map is induced by the multiplication $G\times H\arr G$.
\end{prop}

\begin{proof}
The first claim is clear. For the other we have the following Cartesian
diagrams   \[   \begin{tikzpicture}[xscale=2.1,yscale=-1.2]     \node (A0_0) at (0, 0) {$\stX\times \Delta_H$};     \node (A0_1) at (1, 0) {$\Delta_G\times \Delta_H$};     \node (A0_2) at (2, 0) {$\Delta_G$};     \node (A1_0) at (0, 1) {$\stX$};     \node (A1_1) at (1, 1) {$\Delta_G$};     \node (A1_2) at (2, 1) {$\Delta_{G/H}$};     \path (A0_0) edge [->]node [auto] {$\scriptstyle{}$} (A0_1);     \path (A0_1) edge [->]node [auto] {$\scriptstyle{}$} (A0_2);     \path (A1_0) edge [->]node [auto] {$\scriptstyle{}$} (A1_1);     \path (A1_1) edge [->]node [auto] {$\scriptstyle{}$} (A1_2);     \path (A0_2) edge [->]node [auto] {$\scriptstyle{}$} (A1_2);     \path (A0_0) edge [->]node [auto] {$\scriptstyle{\pr_1}$} (A1_0);     \path (A0_1) edge [->]node [auto] {$\scriptstyle{\pr_1}$} (A1_1);   \end{tikzpicture}   \] 
\end{proof}

\subsection{The group $G=\protect\Z/p\protect\Z$ in characteristic $p$.}

In this section we consider $G=\Z/p\Z$ over $k=\F_{p}$. 

Let $C$ be an $\F_{p}$-algebra. By Artin-Schreier a $\Z/p\Z$-torsor
over $C$ is of the form $C[X]/(X^{p}-X-c)$, where $c\in C$ and
the action is induced by $X\longmapsto X+f$ for $f\in\F_{p}$.
\begin{lem}
\label{lem:bij Iso Artin-Schreier}Let $c,d\in C$. Then   \[   \begin{tikzpicture}[xscale=5.5,yscale=-0.7]     \node (A0_0) at (0, 0) {$\{u\in C\st u^{p}-u+c=d\}$};     \node (A0_1) at (1, 0) {$\Iso_{C}^{\Z/p\Z}(\frac{C[X]}{(X^{p}-X-c)},\frac{C[X]}{(X^{p}-X-d)})$};     \node (A1_0) at (0, 1) {$u$};     \node (A1_1) at (1, 1) {$(X\longmapsto X-u)$};     \path (A0_0) edge [->]node [auto] {$\scriptstyle{}$} (A0_1);     \path (A1_0) edge [serif cm->]node [auto] {$\scriptstyle{}$} (A1_1);   \end{tikzpicture}   \] is
bijective.
\end{lem}

\begin{proof}
The map in the statement is well defined and it induces a morphism
\[
\Spec(C[X]/(X^{p}-X-(d-c)))\arr\Isosh^{\Z/p\Z}(C[X]/(X^{p}-X-c),C[X]/(X^{p}-X-d))=I
\]
The group $\Z/p\Z$ acts on both sides and the map is equivariant.
Since both sides are $\Z/p\Z$-torsors it follows that the above map
is an isomorphism.
\end{proof}
\begin{notation}
\label{nota:description of ZpZ torsors} If $C$ is an $\F_{p}$-algebra,
according to \ref{lem:bij Iso Artin-Schreier}, we identify $(\Bi\Z/p\Z)(C)$
with the category whose objects are elements of $C$ and a morphism
$c\arrdi ud$ is an element $u\in C$ such that $u^{p}-u+c=d$. Composition
is given by the sum, identities correspond to $0\in C$ and the inverse
of $u\in C$ is $-u$.
\end{notation}

In particular we see that if $c\in(\Bi\Z/p\Z)(C)$ then $c\simeq c^{p}$.
\begin{lem}
\label{lem:removing the positive part} Any element $b\in tB[[t]]$
is of the form $u^{p}-u$ for a unique element $u\in tB[[t]]$.
\end{lem}

\begin{proof}
Let $b,u\in tB[[t]]$ and $b_{s},u_{s}$ for $s\in\N$ their coefficients,
so that $b_{0}=u_{0}=0$. We extend the symbol $b_{s},u_{s}$ for
$s\in\Q$ by setting $b_{s}=u_{s}=0$ if $s\notin\N$. The equation
$u^{p}-u=b$ translates in $b_{s}=u_{s/p}^{p}-u_{s}$ for all $s\in\N$.
A simple computation shows that, given $b$, the only solution of
the system is 
\[
u_{s}=-\sum_{n\in\N}b_{s/p^{n}}^{p^{n}}
\]
\end{proof}
\begin{notation}
In what follows we set
\[
S=\{n\geq1\st p\nmid n\}
\]
and $\A^{(S)}\colon\Aff/\F_{p}\arr\sets$ where $\A^{(S)}(B)$ is
the set of maps $b\colon S\arr B$ such that $\{s\in S\st b_{s}\neq0\}$
is finite. 
\end{notation}

Given $k\in\N$ we set
\[
\phi_{k}\colon\A^{(S)}\arr\Delta_{\Z/p\Z}\comma\phi_{k}(b)=\sum_{s\in S}b_{s}t^{-sp^{k}}\in B((t))=\Delta_{\Z/p\Z}(B)
\]
and $\psi_{k}\colon\A^{(S)}\times\Bi(\Z/p\Z)\arr\Delta_{\Z/p\Z}$,
$\psi_{k}(b,b_{0})=\phi_{k}(b)+b_{0}$. Let $F_{\A^{(S)}}$ be the
Frobenius morphism of $\A^{(S)}$ defined in \ref{def:Frobenius}
and let $(\A^{(S)})^{\infty}$ be the limit defined as in \ref{def:X^infty}.
For all $b\in\A^{(S)}(B)$ and $b_{0}\in B$ there is a natural map
\[
\psi_{k+1}\circ(F_{\A^{(S)}}\times\id_{\Bi(\Z/p\Z)})(b,b_{0})\arrdi{-\phi_{k}(b)}\psi_{k}(b,b_{0})
\]
which therefore induces a functor $(\A^{(S)})^{\infty}\times\Bi(\Z/p\Z)\arr\Delta_{\Z/p\Z}$. 
\begin{thm}
\label{thm:The case of Zp} The functor $(\A^{(S)})^{\infty}\times\Bi(\Z/p\Z)\arr\Delta_{\Z/p\Z}$
is an equivalence of fibered categories.
\end{thm}

\begin{proof}
\emph{Essential surjectivity}. Let $b(t)=\sum_{j}b_{j}t^{j}\in\Delta_{\Z/p\Z}(B)$.
By \ref{lem:removing the positive part} and the definition of the
map in the statement we can assume that $b_{j}=0$ for $j>0$. Let
$k\in\N$ be a sufficiently large positive integer such that every
$j<0$ with $b_{j}\ne0$ is written as $-j=p^{k-m(j)}s(j)$ for some
$m(j)\ge0$ and $s(j)\in S$. Then $b_{j}t^{j}\simeq(b_{j}t^{j})^{p^{m(j)}}=b_{j}^{p^{m(j)}}t^{-p^{k}s(j)}$
if $b_{j}\neq0$. We see therefore that, up to change $b$ with an
isomorphic element, $b$ can be written as $\psi_{k}(c)$ for some
$c\in(\A^{(S)}\times\Bi(\Z/p\Z))(B)$.

\emph{Faithfulness. }Let $([b,k],b_{0}),([c,k],c_{0})\in(\A^{(S)})^{\infty}(B)\times\Bi(\Z/p\Z)(B)$
and $u,v\colon(b,b_{0})\arr(c,c_{0})$ two morphisms in $\A^{(S)}\times\Bi(\Z/p\Z)$,
that is $b=c$ and $u^{p}-u=v^{p}-v=c_{0}-b_{0}$ with $u,v\in B$.
If $\psi_{k}(u)=\psi_{k}(v)$ then $u=v$ by definition of $\psi_{k}$
as desired.

\emph{Fullness. }Let $([b,k],b_{0}),([c,k'],c_{0})\in(\A^{(S)})^{\infty}(B)\times\Bi(\Z/p\Z)(B)$
and let $u\colon\psi_{k}(b,b_{0})\arr\psi_{k'}(c,c_{0})$ be a map
in $\Delta_{\Z/p\Z}$. We want to lift this morphism to $(\A^{(S)})^{\infty}\times\Bi(\Z/p\Z)$.
We can assume $k=k'$. The element $u=\sum_{q}u_{q}t^{q}\in B((t))$
can be written as $u=u_{-}+u_{+}$, where $u_{-}=\sum_{q<0}u_{q}t^{q}$
and $u_{+}=\sum_{q\ge0}u_{q}t^{q}$. In particular we obtain that
$u_{-}^{p}-u_{-}=\phi_{k}(c)-\phi_{k}(b)$ and $u_{+}^{p}-u_{+}=c_{0}-b_{0}$.
By \ref{lem:removing the positive part} it follows that $u_{+}\in B$.
It suffices to show that $u_{-}=0$. To see this, we first show that
$c-b$ is nilpotent. We have $\phi_{k}(c)\simeq\phi_{k}(b)$ and,
applying $F_{B}$ to both side we get $\phi_{0}(b)\simeq\phi_{0}(b)^{p^{k}}=F_{B}^{k}(\phi_{k}(b))\simeq F_{B}^{k}(\phi_{k}(c))=\phi_{0}(c)^{p^{k}}\simeq\phi_{0}(c)$.
Thus there exists $v=\sum_{q<0}v_{q}t^{q}\in B((t))$ such that 
\[
\phi_{0}(c)-\phi_{0}(b)=\sum_{s\in S}(c_{s}-b_{s})t^{-s}=v^{p}-v=\sum_{q<0}(v_{q/p}^{p}-v_{q})t^{q}
\]
where we set $v_{l}=0$ if $l\in\Q-\Z_{<0}$. In particular for $s\in S$
and $l\in\N$ we obtain $b_{s}-c_{s}=v_{-s}$ and $v_{-sp^{l}}=v_{-s}^{p^{l}}$.
Since $v_{-sp^{l}}=0$ for $l\gg0$ we see that $b_{s}-c_{s}$ is
nilpotent. This means that there exists $j>0$ such that $F_{\A^{(S)}}^{j}(b)=F_{\A^{(S)}}^{j}(c)$.
Thus, up to replace $k$ by $k+j$, we can assume $b=c$, so that
$u_{-}^{p}=u_{-}$. If $u_{-}=\sum_{q<0}u_{q}t^{q}$ and we put $u_{q}=0$
for $q\notin\Z$ then we have $u_{q}=(u_{q/p^{l}})^{p^{l}}$ for every
$q\in\Q$ with $q<0$ and $l\in\N$. For each $q$, taking a sufficiently
large $l$ with $q/p^{l}\notin\Z$, we see $u_{q}=0$ as desired. 
\end{proof}
\begin{rem}
The addition $\Z/p\Z\times\Z/p\Z\longrightarrow\Z/p\Z$ induces maps
$\Bi(\Z/p\Z)\times\Bi(\Z/p\Z)\longrightarrow\Bi(\Z/p\Z)$ and $\Delta_{\Z/p\Z}\times\Delta_{\Z/p\Z}\longrightarrow\Delta_{\Z/p\Z}$.
The ind-scheme $(\A^{(S)})^{\infty}$ also has a natural group structure
by addition. Notice that the functor in the last theorem preserves
the induced ``group structure'' on both sides. This is because the
maps $\psi_{k}$ preserve the sum and the Frobenius of $\A^{(S)}$
is a group homomorphism. In particular the induced map from $(\A^{(S)})^{\infty}$
to the coarse ind-algebraic space of $\Delta_{\Z/p\Z}$ is an isomorphism
of sheaves of groups.
\end{rem}

\begin{rem}
\label{rem:extension of torsors}If $B$ is an $\F_{p}$-algebra,
$G$ is any constant $p$-group and $H$ is a central subgroup consisting
of elements of order at most $p$ then any map $\Spec B\arr\Delta_{G/H}$
lifts to a map $\Spec B\arr\Delta_{G}$. More generally any $G/H$-torsor
over $B$ extends to a $G$-torsor. This follows from the fact that
there is an exact sequence of sets $\Hl^{1}(B,G)\arr\Hl^{1}(B,G/H)\arr\Hl^{2}(B,H)=0$.
The last vanishing follows because $H\simeq(\Z/p\Z)^{r}$ for some
$r$ and using the Artin-Schreier sequence.
\end{rem}

\begin{cor}
\label{cor:Delta p-gp stack}If $G$ is an étale $p$-group scheme
over a field $k$ then $\Delta_{G}$ is a stack in the fpqc topology.
\end{cor}

\begin{proof}
If $B$ is a $k$-algebra and $A/k$ is a finite $k$-algebra then
$(B\otimes_{k}A)((t))\simeq B((t))\otimes_{k}A$ by \ref{lem:change of rings for series}.
Therefore $\Delta_{G}$ satisfies descent along coverings of the form
$B\arr B\otimes_{k}A$. This implies that it is enough to show that
$\Delta_{G}\times_{k}L\simeq\Delta_{G\times_{k}L}$ is a stack, where
$L/k$ is a finite field extension such that $G\times_{k}L$ is constant.
Again using base change, we can assume $k=\F_{p}$ and $G$ a constant
$p$-group. If $\sharp G=p^{l}$ we proceed by induction on $l$.
If $l=1$ then $\Delta_{\Z/p\Z}\simeq(\A^{(S)})^{\infty}\times\Bi(\Z/p\Z)$
which is a product of stacks. For a general $G$ let $H$ a non-trivial
central subgroup. By induction $\Delta_{G/H}$ is a stack and it is
enough to show that all base change of $\Delta_{G}\arr\Delta_{G/H}$
along a map $\Spec B\arr\Delta_{G/H}$ is a stack. This fiber product
is $\Spec B\times\Delta_{H}$ thanks to \ref{prop:fiber product modding out by a central subgroup}
and \ref{rem:extension of torsors}, which is a stack by inductive
hypothesis.
\end{proof}

\subsection{Tame cyclic case}

Let $k$ be a field and $n\in\N$ such that $n\in k^{*}$. The aim
of this section is to prove Theorem \ref{cyclic}. 

Set $G=\mu_{n}$, the group of $n$-th roots of unity, which is a
finite and étale group scheme over $k$. In particular $\Delta_{G}(B)$
can be seen as the category of pairs $(\shL,\sigma)$ where $\shL$
is an invertible sheaf over $B((t))$ and $\sigma\colon\shL^{\otimes n}\arr B((t))$
is an isomorphism. When $\shL=B((t))$ the isomorphism $\sigma$ will
often be thought of as an element $\sigma\in B((t))^{*}$.
\begin{lem}
\label{lem:Pic(B((t))) torsions}An invertible sheaf $\shL$ over
$B((t))$ with $\shL^{\otimes n}\simeq B((t))$ is the pullback of
an invertible sheaf over $\Spec B$. More precisely , the $n$-torsion
part of $\Pic(B((t)))$ is naturally isomorphic to the one of $\Pic(B)$. 
\end{lem}

\begin{proof}
Gabber's formula \cite[(1.2.2)]{BouthierCesnavicius} says
\[
\Pic(B((t)))\simeq\Pic(B[t^{-1}])\oplus\mathrm{H}_{\textrm{ét}}^{1}(B,\Z).
\]
This is proved in a slightly more general form in Theorem 3.1.7 of
the cited paper by Bouthier-\v{C}esnavi\v{c}ius. Let $N\Pic(B)$
denote the kernel of $\Pic(B[t])\to\Pic(B)$ so that 
\[
\Pic(B[t^{-1}])\simeq\Pic(B[t])=\Pic(B)\oplus N\Pic(B).
\]
According to \cite[Th.\ 6.1]{Swan}, $N\Pic(B)$ has no $n$-torsion
if and only if $B_{\mathrm{red}}$ is $n$-seminormal. The $n$-seminormality
is defined as follows. For a reduced ring $A$, there exists an extension
$A\subset\overline{A}$ called the seminormalization (see \cite[Section 4]{Swan}).
For our purpose, we only need to know its existence. We say that $A$
is $n$-seminormal if every element $x\in\overline{A}$ with $x^{2},x^{3},nx\in A$
belongs to $A$. In our situation, since $n$ is invertible in $k$,
every $k$-algebra is $n$-seminormal.  Thus $B_{\mathrm{red}}$
is $n$-seminormal and $N\Pic(B)$ has no $n$-torsion. 

It remains to show that $\mathrm{H}_{\textrm{ét}}^{1}(B,\Z)$ has
no $n$-torsion. To do so, we consider the exact sequence
\[
0\longrightarrow\Z\xrightarrow{\times n}\Z\longrightarrow\Z/n\Z\longrightarrow0
\]
of constant étale sheaves on the small étale site of $\Spec B$. Taking
cohomology groups, we get the following exact sequence:
\[
\mathrm{H}^{0}(B,\Z)\longrightarrow\mathrm{H}^{0}(B,\Z/n\Z)\longrightarrow(\textrm{the \ensuremath{n}-torsion part of }\mathrm{H}_{\textrm{ét}}^{1}(B,\Z))\longrightarrow0
\]
The left map is surjective, since every locally constant function
$\Spec B\longrightarrow\Z/n\Z$ lifts to a locally constant function
$\Spec B\longrightarrow\Z$. It follows that $\mathrm{H}_{\textrm{ét}}^{1}(B,\Z)$
has no $n$-torsion. 
\end{proof}
\begin{lem}
\label{lem:roots in power series}For a $k$-algebra $B$, we have
\[
\mu_{n}(B)=\{b\in B^{*}\mid b^{n}=1\}=\{b\in B((t))^{*}\mid b^{n}=1\}=\mu_{n}(B((t))).
\]
\end{lem}

\begin{proof}
Let $L$ and $R$ denote the left and right sides respectively. Obviously
$L\subset R$. It is also easy to see $R\cap B[[t]]=L$. Thus it suffices
to show that $R\subset B[[t]]$. Conversely, we suppose that it was
not the case. We define the naive order $\ord_{naive}(a)$ of $a=\sum_{i\in\Z}a_{i}t^{i}\in B[[t]]$
as $\min\{i\mid a_{i}\ne0\}$. Elements outside $B[[t]]$ have negative
naive orders and choose an element $c=\sum_{i\in\Z}c_{i}t^{i}\in R\setminus B[[t]]$
such that $\ord_{naive}(c)$ attains the maximum, say $i_{0}<0$.
Taking derivatives of $c^{n}=1$, we get $ncc'=0$ with $c'$ the
derivative of $c$. Since $nc$ is invertible, $c'=0$. If $\car k=0$
it immediately follows that $c\in B$. So assume $\car k=p>0$. In
this case $c_{i}=0$ for all $i$ with $p\nmid i$. This means that
$c$ is in the image of the injective $B$-algebra homomorphism $f\colon B[[t]]\to B[[t]]$,
$t\mapsto t^{p}$. Let $d$ be the unique preimage of $c$ under $f$,
which is explicitly given by $d=\sum_{i\in\Z}c_{pi}t^{i}$. In particular,
\[
\ord_{naive}(d)=\ord_{naive}(c)/p>\ord_{naive}(c).
\]
Since $f(d^{n})=c^{n}=1$ and $f$ is injective, we have $d^{n}=1$
and $d\in R\setminus B[[t]]$. This contradicts the way of choosing
$c$. We have proved the lemma.
\end{proof}

\begin{proof}
[Proof of Theorem \ref{cyclic}]\label{pf: B} We first define the
functor $\psi\colon\bigsqcup_{q=0}^{n-1}\Bi(G)\arr\Delta_{G}$. An
object of $\bigsqcup_{q=0}^{n-1}\Bi(G)$ over a $k$-algebra $A$
is a factorization $A=\prod_{q}A_{q}$ plus a tuple $(L_{q},\xi_{q})_{q}$
where $L_{q}$ is an $A_{q}$-module and $\xi_{q}\colon L_{q}^{\otimes n}\arr A_{q}$
is an isomorphism. A morphism $(A=\prod_{q}A_{q},L_{q},\xi_{q})\arr(A=\prod_{q}A'_{q},L'_{q},\xi'_{q})$
exists if and only if $A_{q}\simeq A_{q'}$ as $A$-algebras (so that
such isomorphism is unique) and in this case is a collection of isomorphisms
$L_{q}\arr L'_{q}$ compatible with the maps $\xi_{q}$ and $\xi'_{q}$.
To such an object we associate the invertible $A((t))=\prod_{q}A_{q}((t))$-module
$L=\prod_{q}(L_{q}\otimes_{A_{q}}A_{q}((t)))$ together with the map
\[
L^{\otimes n}\simeq\prod_{q}(L_{q}^{\otimes n}\otimes_{A_{q}}A_{q}((t)))\arr\prod_{q}A_{q}((t))=A((t))\comma(x_{q}\otimes1)_{q}\mapsto(\xi_{q}(x_{q})t^{q})_{q}
\]
It is easy to see that the functor $\psi$ on $\Bi G$ in the index
$q$ is the one in the statement. We are going to show that $\psi$
is an equivalence. Since $\Delta_{G}$ is a prestack by \ref{cor:Delta G is a prestack},
it will be enough to show that $\psi$ is an epimorphism and that
it is fully faithful. Indeed this would imply that $\Delta_{G}$ is
also a stack for the following reason. Given a descent datum for $\Delta_{G}$,
since $\Delta_{G}$ is a prestack, in order to show that it is effective
we can refine this datum, that is refine the covering over which is
defined. If $\psi$ is fully faithful and an epimorphism, it follows
that we can always assume that the descent datum for $\Delta_{G}$
comes from a descent datum for $\bigsqcup_{q=0}^{n-1}\Bi(G)$, which
is therefore effective.

$\psi$ \emph{epimorphism. }Let $\chi\in\Delta_{G}(B)$. From \ref{lem:Pic(B((t))) torsions},
we can assume that the associated invertible sheaf is trivial and
$\chi=(B((t)),b)$. We have $(B((t)),b)\simeq(B((t)),b')$ if and
only if there exists $u\in B((t))^{*}$ such that $u^{n}b=b'$. For
$c\in B((t))^{*}$ we define $\ord c\colon\Spec B\to\Z$ as follows:
if $x\in\Spec B$ is a point with the residue field $\kappa$ and
$c_{x}\in\kappa((t))$ is the induced power series, then $(\ord c)(x):=\ord c_{x}$.
This function is upper semicontinuous. From the additivity of orders,
$\ord b+\ord(b^{-1})$ is constant zero. Since $\ord b$ and $\ord(b^{-1})$
are both upper semicontinuous, they are in fact locally constant.
Thus we may suppose that $\ord b$ is constant, equivalently that
if $b_{j}$ are coefficients of $b$, then for some $i$, $b_{i}$
is a unit and $b_{j}$ are nilpotents for $j<i$.

Thus we can write $b=b_{-}+t^{i}b_{+}$ with $b_{t}\in B[[t]]^{*}$
and $b_{-}\in B((t))$ nilpotent. Set $\omega=b/(t^{i}b_{+})\in B((t))$,
$A=B((t))[Y]/(Y^{n}-\omega)$ and $C=B((t))/(b_{-})$. We have that
$\omega=1$ in $C$ and therefore that $A\otimes_{B((t))}C$ has a
section. Since $A/B((t))$ is étale and $B((t))\arr C$ is surjective
with nilpotent kernel the section extends, that is $\omega$ is an
$n$-th power. Thus we can assume $b_{-}=0$. Since $B[[t]][Y]/(Y^{n}-b_{+})$
is étale over $B[[t]]$, by \ref{cor:shrinking to etale neighborhood}
we can assume there exists $\hat{b}\in B[[t]]^{*}$ such that $\hat{b}^{n}=b_{+}$.
In conclusion we reduce to the case $b=t^{i}$ and, multiplying by
a power of $t^{n}$, we can finally assume $0\leq i<n$. %

$\psi$ \emph{fully faithful. }If $(\shL,\sigma)\in\Delta_{G}(B)$
then, by Lemma \ref{lem:roots in power series}, its automorphisms
are canonically isomorphic to $\mu_{n}(B((t)))=\mu_{n}(B)$. This
easily implies that the restriction of $\psi$ on each component is
fully faithful. Given two objects $\alpha,\beta\in\bigsqcup_{q=0}^{n-1}\Bi(G)$
and an isomorphism $\psi(\alpha)\longrightarrow\psi(\beta)$ of their
images the problem of finding an isomorphism $\alpha\longrightarrow\beta$
inducing the given one is local and easily reducible to the following
claim: if $(B((t)),t^{q})\simeq(B((t)),t^{q'})$ then $q\equiv q'\mod n$.
But the first condition means that there exists $u\in L((t))^{*}$
such that $u^{n}t^{q}=t^{q'}$ and, applying $\ord$, we get the result.
\end{proof}

\subsection{General $p$-groups}

In this section we consider the case of a constant $p$-group $G$
over a field $k$ of characteristic $p$ and the aim is to prove Theorem
\ref{A} in this case. We setup the following notation for this section.
All groups considered in this section are constant.
\begin{defn}
\label{nota:direct system of Fp vector spaces} We set 
\[
S=\{n\geq1\st p\nmid n\}
\]
and, given a finite dimensional $\F_{p}$-vector space $H$ (regarded
as an abelian $p$-group), we define a sheaf of abelian groups 
\[
X_{H}=(\A^{(S)})^{\infty}\otimes_{\F_{p}}H\colon\Aff/\F_{p}\arr(\textrm{Abelian groups})
\]
(that is $X_{H}=[(\A^{(S)})^{\infty}]^{n}$ if $\dim_{\F_{p}}H=n$
after the choice of a basis of $H$). We also define sheaves of abelian
groups $X_{H,m}=\A^{(S_{m})}\otimes_{\F_{p}}H$ for $m\in\N$ with
$S_{m}=\{n\in S\st n\leq m\}$. We finally define $\Delta_{H,m}=X_{H,m}\times\Bi(H)$.
\end{defn}

\begin{lem}
\label{lem:direct system of Fp vector spaces}
\begin{enumerate}
\item We have an isomorphism $X_{H}\times\Bi H\arr\Delta_{H}$. 
\item We have $X_{H}=\varinjlim_{m}X_{H,m}$ as sheaves of abelian groups,
where the transition map $X_{H,m}\arr X_{H,m+1}$ is the composition
of the inclusion $\A^{(S_{m})}\otimes_{\F_{p}}H\arr\A^{(S_{m+1})}\otimes_{\F_{p}}H$
and the Frobenius of $\A^{(S_{m+1})}\otimes_{\F_{p}}H$. 
\item We have an equivalence $\varinjlim_{m}(\Delta_{H,m})\simeq\Delta_{H}$.
\end{enumerate}
\end{lem}

\begin{proof}
The last two assertions are obvious. We prove the first one. If $H$
is the cyclic group of order $p$, then this is just \ref{thm:The case of Zp}.
Otherwise we take a subgroup $1\ne I\subsetneq H$. Since the quotient
map $H\to H/I$ has a section, the morphism $\Delta_{H}\to\Delta_{H/I}$
also has a section. From \ref{prop:fiber product modding out by a central subgroup},
we have $\Delta_{H}\simeq\Delta_{H/I}\times\Delta_{I}$. The assertion
follows from induction on the order of $H$. 
\end{proof}
In the following proposition, we use rigidification, an operation
introduced in \cite{Abramovich2007} for algebraic stacks. Roughly
speaking, it kills some subgroups of stabilizers. Generalization to
non-algebraic stacks will be treated in Appendix \ref{sec:Rigidification}.
Note that from \ref{cor:Delta p-gp stack}, $\Delta_{G}$ is a stack
for a $p$-group $G$. 
\begin{lem}
\label{lem:rigidify Delta}Let $H$ be a finite dimensional $\F_{p}$-vector
space. We have $\Delta_{H}\rig H\simeq X_{H}$.
\end{lem}

\begin{proof}
By \ref{lem:direct system of Fp vector spaces} we have $\Delta_{H}\simeq X_{H}\times\Bi H$.
Since $H$ is abelian the result follows from \ref{prop:trivial rigidification}.
\end{proof}
\begin{prop}
\label{prop:factorizing through the rigidification }Let $G$ be a
$p$-group and $H$ be a central subgroup which is an $\F_{p}$-vector
space. Then $H$ is naturally a subgroup sheaf of the inertia stack
of $\Delta_{G}$ (see Appendix \ref{sec:Rigidification} for the inertia
stack as a group sheaf) and the quotient map $\Delta_{G}\arr\Delta_{G/H}$
is the composition of the rigidification $\Delta_{G}\arr\Delta_{G}\rig H$
and an $X_{H}$-torsor $\Delta_{G}\rig H\arr\Delta_{G/H}$, where
the action of $X_{H}$ on $\Delta_{G}\rig H$ is induced by $\Delta_{G}\times\Delta_{H}\arr\Delta_{G}$
and rigidification.
\end{prop}

\begin{proof}
The subgroup $H$ acts on any $G$-torsor because its central. Moreover
the functor $\Delta_{G}\arr\Delta_{G/H}$ sends isomorphisms coming
from $H$ to the identity and therefore factors through the rigidification
$\Delta_{G}\rig H$ by \ref{prop:properties of rigidification}, 2).
Rigidifying both sides of $\Delta_{G}\times\Delta_{H}\arr\Delta_{G}$
we get a map $(\Delta_{G}\rig H)\times X_{H}\arr\Delta_{G}\rig H$
over $\Delta_{G/H}$. Using \ref{prop:fiber product modding out by a central subgroup}
and \ref{prop:properties of rigidification}, 3) we can deduce that
$\Delta_{G}\rig H\arr\Delta_{G/H}$ is an $X_{H}$-torsor.
\end{proof}
\begin{lem}
\label{lem:inductive step for p-groups} Let $G$ be a $p$-group
and $H$ be a central subgroup which is an $\F_{p}$-vector space.
Let also $\stY_{*}$ be a direct system of quasi-separated stacks
over $\N$ with a direct system of smooth (étale) atlases $U_{*}$
made of quasi-compact schemes, $\varinjlim_{n}\stY_{n}\arr\Delta_{G/H}$
a map and $\varinjlim_{n}U_{n}\arr\Delta_{G}$ a lifting. Then there
exists a strictly increasing map $q\colon\N\arr\N$, a direct system
of quasi-separated stacks $\stZ_{*}$ with a direct system of smooth
(étale) atlases $U_{*}\times X_{H,q_{*}}$ (where the transition morphisms
$U_{i}\times X_{H,q_{i}}\arr U_{i+1}\times X_{H,q_{i+1}}$ is the
product of the given map $U_{i}\arr U_{i+1}$ and the map $X_{H,q_{i}}\arr X_{H,q_{i+1}}$
of \ref{lem:direct system of Fp vector spaces}), compatible maps
$\stZ_{i}\arr\stY_{i}$ induced by the projection $U_{i}\times X_{H,q_{i}}\arr U_{i}$
and which are a composition of an $H$-gerbe $\stZ_{i}\arr\stZ_{i}\rig H$
and a $X_{H,q_{i}}$-torsor $\stZ_{i}\rig H\arr\stY_{i}$ and an equivalence
\[
\varinjlim_{n}\stZ_{n}\simeq(\varinjlim_{n}\stY_{n})\times_{\Delta_{G/H}}\Delta_{G}
\]
Moreover there is a factorization $U_{i}\times X_{H,q_{i}}\arr U_{i}\times_{\stY_{i}}\stZ_{i}\arr\stZ_{i}$
where the first arrow is an $H$-torsor. 
\end{lem}

\begin{proof}
Consider one index $i\in\N$ and the Cartesian diagrams   \[   \begin{tikzpicture}[xscale=2.1,yscale=-1.2]     \node (A0_0) at (0, 0) {$P'_{U,i}$};     \node (A0_1) at (1, 0) {$P'_i$};     \node (A0_2) at (2, 0) {$\Delta_G$};     \node (A1_0) at (0, 1) {$P_{U,i}$};     \node (A1_1) at (1, 1) {$P_i$};     \node (A1_2) at (2, 1) {$\Delta_{G}\rig H$};     \node (A2_0) at (0, 2) {$U_i$};     \node (A2_1) at (1, 2) {$\stY_i$};     \node (A2_2) at (2, 2) {$\Delta_{G/H}$};     \path (A0_1) edge [->]node [auto] {$\scriptstyle{}$} (A1_1);     \path (A0_0) edge [->]node [auto] {$\scriptstyle{}$} (A0_1);     \path (A2_1) edge [->]node [auto] {$\scriptstyle{}$} (A2_2);     \path (A1_0) edge [->]node [auto] {$\scriptstyle{}$} (A1_1);     \path (A0_1) edge [->]node [auto] {$\scriptstyle{}$} (A0_2);     \path (A0_2) edge [->]node [auto] {$\scriptstyle{}$} (A1_2);     \path (A1_1) edge [->]node [auto] {$\scriptstyle{}$} (A1_2);     \path (A1_0) edge [->]node [auto] {$\scriptstyle{}$} (A2_0);     \path (A1_1) edge [->]node [auto] {$\scriptstyle{}$} (A2_1);     \path (A0_0) edge [->]node [auto] {$\scriptstyle{}$} (A1_0);     \path (A2_0) edge [->]node [auto] {$\scriptstyle{}$} (A2_1);     \path (A1_2) edge [->]node [auto] {$\scriptstyle{}$} (A2_2);   \end{tikzpicture}   \] Set
also $R_{i}=U_{i}\times_{\stY_{i}}U_{i}$, which is a quasi-compact
algebraic space. By \ref{prop:factorizing through the rigidification }
$P_{i}\arr\stY_{i}$ is a $X_{H}$-torsor, $P'_{i}\arr P_{i}$ an
$H$-gerbe. Moreover the lifting $U_{i}\arr\Delta_{G}$ gives an isomorphism
$P'_{U,i}\simeq U_{i}\times\Delta_{H}$ and $P_{U,i}\simeq U_{i}\times X_{H}$
by \ref{prop:fiber product modding out by a central subgroup}. Thus
the $X_{H}$-torsor $P_{i}\arr\stY_{i}$, by descent along $U_{i}\arr\stY_{i}$
is completely determined by the identification $R_{i}\times X_{H}\simeq_{R_{i}}R_{i}\times X_{H}$,
which consists of an element $\omega_{i}\in X_{H}(R_{i})$ satisfying
the cocycle condition on $U_{i}\times_{\stY_{i}}U_{i}\times_{\stY_{i}}U_{i}$.
The given equivalence $P_{i}\simeq(P_{i+1})_{|\stY_{i}}$ of $X_{H}$-torsors
over $\stY_{i}$ is completely determined by its pullback on $U_{i}$,
which is given by multiplication by $\gamma_{i}\in X_{H}(U_{i})$.
The compatibility this element has to satisfy is expressed by
\[
(\omega_{i+1})_{|R_{i}}(s_{i}^{*}\gamma_{i})=(t_{i}^{*}\gamma_{i})\omega_{i}\text{ in }X_{H}(R_{i})
\]
where $s_{i},t_{i}\colon R_{i}\arr U_{i}$ are the two projections.
Since all $R_{i}$ are quasi-compact and $X_{H}\simeq\varinjlim_{j}X_{H,j}$
we can find an increasing sequence of natural numbers $q\colon\N\arr\N$
and elements $e_{q_{i}}\in X_{H,q_{i}}(R_{i})$, $f_{q_{i}}\in X_{H,q_{i+1}}(U_{i})$
such that:

\begin{enumerate}
\item the element $e_{q_{i}}$ is mapped to $\omega_{i}$ under the map
$X_{H,q_{i}}(R_{i})\arr X_{H}(R_{i})$ and it satisfies the cocycle
condition in $X_{H,q_{i}}(U_{i}\times_{\stY_{i}}U_{i}\times_{\stY_{i}}U_{i})$;
\item the element $f_{q_{i}}$ is mapped to $\gamma_{i}$ under the map
$X_{H,q_{i+1}}(U_{i})\arr X_{H}(U_{i})$ and, if $\overline{e_{q_{i}}}$
is the image of $e_{q_{i}}$ under the map $X_{H,q_{i}}(R_{i})\arr X_{H,q_{i+1}}(R_{i})$,
it satisfies 
\[
(e_{q_{i+1}})_{|R_{i}}(s_{i}^{*}f_{q_{i}})=(t_{i}^{*}f_{q_{i}})\overline{e_{q_{i}}}\text{ in }X_{H,q_{i+1}}(R_{i})
\]
\end{enumerate}
The data of $1)$ determine $X_{H,q_{i}}$-torsors $Q_{i}\arr\stY_{i}$
with a map $U_{i}\arr Q_{i}$ over $\stY_{i}$ and together with an
$X_{H,q_{i}}$-equivariant map $Q_{i}\arr P_{i}$ such that $U_{i}\arr Q_{i}\arr P_{i}$
is the given map. The data of $2)$ determine an $X_{H,q_{i}}$-equivariant
map $Q_{i}\arr(Q_{i+1})_{|\stY_{i}}$inducing the equivalence $P_{i}\simeq(P_{i+1})_{|\stY_{i}}$.

Consider also the $H$-gerbe $Q_{i}'\arr Q_{i}$ pullback of $P_{i}'\arr P_{i}$
along $Q_{i}\arr P_{i}$. We set $\stZ_{i}=Q_{i}'$. We have Cartesian
diagrams   \[   \begin{tikzpicture}[xscale=1.5,yscale=-1.2]     \node (A0_0) at (0, 0) {$Q_i'$};     \node (A0_1) at (1, 0) {$P_i'$};     \node (A0_2) at (2, 0) {$P_{i+1}'$};     \node (A0_3) at (3, 0) {$Q_i'$};     \node (A0_4) at (4, 0) {$Q_{i+1}'$};     \node (A0_5) at (5, 0) {$P_{i+1}'$};     \node (A1_0) at (0, 1) {$Q_i$};     \node (A1_1) at (1, 1) {$P_i$};     \node (A1_2) at (2, 1) {$P_{i+1}$};     \node (A1_3) at (3, 1) {$Q_i$};     \node (A1_4) at (4, 1) {$Q_{i+1}$};     \node (A1_5) at (5, 1) {$P_{i+1}$};     \path (A1_4) edge [->]node [auto] {$\scriptstyle{}$} (A1_5);     \path (A0_0) edge [->]node [auto] {$\scriptstyle{}$} (A0_1);     \path (A0_1) edge [->]node [auto] {$\scriptstyle{}$} (A0_2);     \path (A1_0) edge [->]node [auto] {$\scriptstyle{}$} (A1_1);     \path (A0_3) edge [->]node [auto] {$\scriptstyle{}$} (A1_3);     \path (A0_2) edge [->]node [auto] {$\scriptstyle{}$} (A1_2);     \path (A0_4) edge [->]node [auto] {$\scriptstyle{}$} (A0_5);     \path (A0_3) edge [->]node [auto] {$\scriptstyle{}$} (A0_4);     \path (A0_4) edge [->]node [auto] {$\scriptstyle{}$} (A1_4);     \path (A1_1) edge [->]node [auto] {$\scriptstyle{}$} (A1_2);     \path (A0_5) edge [->]node [auto] {$\scriptstyle{}$} (A1_5);     \path (A0_0) edge [->]node [auto] {$\scriptstyle{}$} (A1_0);     \path (A0_1) edge [->]node [auto] {$\scriptstyle{}$} (A1_1);     \path (A1_3) edge [->]node [auto] {$\scriptstyle{}$} (A1_4);   \end{tikzpicture}   \] Notice
that, if $\stM$ is a stack over $\stY_{i}$, then $\stM\times_{\stY_{i+1}}U_{i+1}\simeq\stM\times_{\stY_{i}}U_{i}$
because $U_{*}$ is a direct system of atlases. Pulling back along
$U_{i+1}\arr\stY_{i+1}$ the above diagrams, we obtain the bottom
rows of the following diagrams.   \[   \begin{tikzpicture}[xscale=3.2,yscale=-1.2]     \node (A0_0) at (0, 0) {$ X_{H,q_i}\times U_{i}$};     \node (A0_1) at (1, 0) {$ X_{H}\times U_{i}$};     \node (A0_2) at (2, 0) {$ X_{H}\times U_{i+1}$};     \node (A1_0) at (0, 1) {$Q_i'\times_{\stY_{i+1}}U_{i+1}$};     \node (A1_1) at (1, 1) {$ \Delta_{H}\times U_{i}$};     \node (A1_2) at (2, 1) {$ \Delta_{H}\times U_{i+1}$};     \node (A2_0) at (0, 2) {$ X_{H,q_i}\times U_{i}$};     \node (A2_1) at (1, 2) {$ X_{H}\times U_{i}$};     \node (A2_2) at (2, 2) {$ X_{H}\times U_{i+1}$};     \node (A3_0) at (0, 3) {$ X_{H,q_i}\times U_{i}$};     \node (A3_1) at (1, 3) {$ X_{H,q_{i+1}}\times U_{i+1}$};     \node (A3_2) at (2, 3) {$ X_{H}\times U_{i+1}$};     \node (A4_0) at (0, 4) {$Q_i'\times_{\stY_{i+1}}U_{i+1}$};     \node (A4_1) at (1, 4) {$Q_{i+1}'\times_{\stY_{i+1}}U_{i+1}$};     \node (A4_2) at (2, 4) {$ \Delta_{H}\times U_{i+1}$};     \node (A5_0) at (0, 5) {$X_{H,q_i}\times U_{i}$};     \node (A5_1) at (1, 5) {$ X_{H,q_{i+1}}\times U_{i+1}$};     \node (A5_2) at (2, 5) {$X_{H}\times U_{i+1}$};     \path (A5_0) edge [->]node [auto] {$\scriptstyle{}$} (A5_1);     \path (A2_1) edge [->]node [auto] {$\scriptstyle{}$} (A2_2);     \path (A0_2) edge [->]node [auto] {$\scriptstyle{\alpha_{i+1}}$} (A1_2);     \path (A4_0) edge [->]node [auto] {$\scriptstyle{}$} (A5_0);     \path (A1_0) edge [->]node [auto] {$\scriptstyle{}$} (A2_0);     \path (A1_1) edge [->]node [auto] {$\scriptstyle{}$} (A2_1);     \path (A1_2) edge [->]node [auto] {$\scriptstyle{}$} (A2_2);     \path (A2_0) edge [->]node [auto] {$\scriptstyle{}$} (A2_1);     \path (A3_0) edge [->]node [auto] {$\scriptstyle{}$} (A3_1);     \path (A0_1) edge [->]node [auto] {$\scriptstyle{\alpha_{i}}$} (A1_1);     \path (A4_2) edge [->]node [auto] {$\scriptstyle{}$} (A5_2);     \path (A4_0) edge [->]node [auto] {$\scriptstyle{}$} (A4_1);     \path (A1_0) edge [->]node [auto] {$\scriptstyle{}$} (A1_1);     \path (A3_1) edge [->]node [auto] {$\scriptstyle{\beta_{i+1}}$} (A4_1);     \path (A1_1) edge [->]node [auto] {$\scriptstyle{}$} (A1_2);     \path (A4_1) edge [->]node [auto] {$\scriptstyle{}$} (A5_1);     \path (A0_0) edge [->]node [auto] {$\scriptstyle{\beta_i}$} (A1_0);     \path (A4_1) edge [->]node [auto] {$\scriptstyle{}$} (A4_2);     \path (A0_0) edge [->]node [auto] {$\scriptstyle{}$} (A0_1);     \path (A3_0) edge [->]node [auto] {$\scriptstyle{\beta_i}$} (A4_0);     \path (A3_2) edge [->]node [auto] {$\scriptstyle{\alpha_{i+1}}$} (A4_2);     \path (A0_1) edge [->]node [auto] {$\scriptstyle{}$} (A0_2);     \path (A3_1) edge [->]node [auto] {$\scriptstyle{}$} (A3_2);     \path (A5_1) edge [->]node [auto] {$\scriptstyle{}$} (A5_2);   \end{tikzpicture}   \] 

The top rows of the above diagrams is instead obtained using \ref{prop:fiber product modding out by a central subgroup},
where the map $\alpha_{i}$ are induced by the map $X_{H}\arr X_{H}\times\Bi H=\Delta_{H}$
and the $\beta_{i}$ are induced by the $\alpha_{i}$. Since $Q_{i}'\times_{\stY_{i+1}}U_{i+1}\simeq Q_{i}'\times_{\stY_{i}}U_{i}$
we see that the atlases $U_{i}\times X_{H,q_{i}}\arrdi{\beta_{i}}Q_{i}'\times_{\stY_{i}}U_{i}\arr Q_{i}'=\stZ_{i}$
define a direct system of smooth (resp. étale) atlases satisfying
the requests of the statement. 

Let us show the last equivalence in the statement. By \ref{cor:about limits and atlas}
the map 
\[
\varinjlim_{n}(U_{n}\times_{\stY_{n}}P_{n})=\varinjlim_{n}(U_{n}\times X_{H})\arr\varinjlim_{n}P_{n}
\]
is a smooth atlas. The map $\varinjlim_{n}Q_{n}\arr\varinjlim_{n}P_{n}$
is therefore an equivalence because its base change along the above
atlas is $\varinjlim_{n}(U_{n}\times X_{H,q_{n}})\arr\varinjlim_{n}(U_{n}\times X_{H})$,
which is an isomorphism. Here we have used \ref{prop:fiber product of direct systems}.
Using again this we see that the map
\[
\varinjlim_{n}\stZ_{n}=\varinjlim_{n}(Q_{n}\times_{P_{n}}P_{n}')\arr\varinjlim_{n}P_{n}'=\varinjlim_{n}(\stY_{n}\times_{\Delta_{G/H}}\Delta_{G})
\]
is an equivalence as well. 
\end{proof}
\begin{lem}
\label{lem:extension of torsors along closed embedding} Let $G$
be a $p$-group, $H$ be a central subgroup which is an $\F_{p}$-vector
space and $X\arr Y$ be a finite, finitely presented and universally
injective map of affine schemes. Then a $2$-commutative diagram   \[   \begin{tikzpicture}[xscale=2.1,yscale=-1.2]     \node (A0_0) at (0, 0) {$X$};     \node (A0_1) at (1, 0) {$\Delta_G$};     \node (A1_0) at (0, 1) {$Y$};     \node (A1_1) at (1, 1) {$\Delta_{G/H}$};     \path (A0_0) edge [->]node [auto] {$\scriptstyle{}$} (A0_1);     \path (A1_0) edge [->,dashed]node [auto] {$\scriptstyle{}$} (A0_1);     \path (A1_0) edge [->]node [auto] {$\scriptstyle{}$} (A1_1);     \path (A0_1) edge [->]node [auto] {$\scriptstyle{}$} (A1_1);     \path (A0_0) edge [->]node [auto] {$\scriptstyle{}$} (A1_0);   \end{tikzpicture}   \] always
admits a dashed map.
\end{lem}

\begin{proof}
Set $X=\Spec B$ and $Y=\Spec C$ and consider the induced map $C\arr B$.
Since $\Hl^{2}(B((t)),H)=\Hl^{2}(C((t)),H)=0$ by the Artin-Schreier
sequence, we have a commutative diagram   \[   \begin{tikzpicture}[xscale=3.3,yscale=-1.2]     \node (A0_0) at (0, 0) {$\Hl^1(C((t)),H)$};     \node (A0_1) at (1, 0) {$\Hl^1(C((t)),G)$};     \node (A0_2) at (2, 0) {$\Hl^1(C((t)),G/H)$};     \node (A0_3) at (3, 0) {$0$};     \node (A1_0) at (0, 1) {$\Hl^1(B((t)),H)$};     \node (A1_1) at (1, 1) {$\Hl^1(B((t)),G)$};     \node (A1_2) at (2, 1) {$\Hl^1(B((t)),G/H)$};     \node (A1_3) at (3, 1) {$0$};     \path (A0_1) edge [->]node [auto] {$\scriptstyle{}$} (A1_1);     \path (A0_0) edge [->]node [auto] {$\scriptstyle{}$} (A0_1);     \path (A0_1) edge [->]node [auto] {$\scriptstyle{}$} (A0_2);     \path (A1_0) edge [->]node [auto] {$\scriptstyle{}$} (A1_1);     \path (A1_1) edge [->]node [auto] {$\scriptstyle{}$} (A1_2);     \path (A0_2) edge [->]node [auto] {$\scriptstyle{}$} (A1_2);     \path (A1_2) edge [->]node [auto] {$\scriptstyle{}$} (A1_3);     \path (A0_2) edge [->]node [auto] {$\scriptstyle{}$} (A0_3);     \path (A0_0) edge [->]node [auto] {$\scriptstyle{\alpha}$} (A1_0);   \end{tikzpicture}   \] with
exact rows. By hypothesis there are $u\in\Hl^{1}(B((t)),G)$ and $v\in\Hl^{1}(C((t)),G/H)$
which agree in $\Hl^{1}(B((t)),G/H)$. We can find a common lifting
in $\Hl^{1}(C((t)),G)$ by proving that the map $\alpha$ is surjective.
By \ref{lem:change of rings for series} we have that $\Spec B((t))\arr\Spec C((t))$
is the base change of $\Spec B\arr\Spec C$ and therefore is finite
and universally injective. Let $D$ be the image of $C((t))\arr B((t))$.
The map $\Hl^{1}(C((t)),H)\arr\Hl^{1}(D,H)$ is surjective because
$H\simeq(\Z/p\Z)^{l}$ for some $l$ and using the description of
$\Z/p\Z$-torsors in \ref{nota:description of ZpZ torsors}. By \ref{prop:finite and universally injective maps}
the map $\Spec B((t))\arr\Spec D$ is a finite universal homeomorphism.
Thus $\Hl^{1}(D,H)\arr\Hl^{1}(B((t)),H)$ is bijective by \ref{rem:BG insensible to universal homeo, finite}.
\end{proof}

\begin{proof}
[Proof of Theorem \ref{A}, \ref{thma-p-gp} and \ref{thma-abelian-p}.]\label{pf: A p}
Since $p$-groups have non trivial center we can find a sequence of
quotients
\[
G=G_{l}\arr G_{l-1}\arr G_{l-2}\arr\cdots\arr G_{0}\arr G_{-1}=0
\]
where $\Ker(G_{u}\arr G_{u-1})$ is central in $G_{u}$ and an $\F_{p}$-vector
space. We proceed by induction on $l$. In the base case $l=0$, so
that $G=G_{0}$ is an $\F_{p}$-vector space, following \ref{lem:direct system of Fp vector spaces}
it is enough to set $\stX_{n}=X_{G,n}\times\Bi G$. Consider now the
inductive step and set $H=\Ker(G=G_{l}\arr G_{l-1})$. Of course we
can assume $H\neq0$, so that we can use the inductive hypothesis
on $G/H$, obtaining a direct system $\stY_{*}$ and a map $\overline{v}\colon\N\arr\N$
with a direct system of atlases $\A^{\overline{v}_{*}}$. The first
result follows applying \ref{lem:inductive step for p-groups} with
$U_{n}=\A^{\overline{v}_{n}}$. We just have to prove the existence
of a lifting $\varinjlim_{n}U_{n}\arr\Delta_{G}$. Since the schemes
$U_{n}$ are affine one always find a lifting $U_{n}\arr\Delta_{G}$
thanks to \ref{rem:extension of torsors}. Thanks to \ref{lem:extension of torsors along closed embedding}
any lifting $U_{n}\arr\Delta_{G}$ always extends to a lifting $U_{n+1}\arr\Delta_{G}$.

Assume now $G$ abelian and set $X_{G}=X_{G/H}\times X_{H}$. By induction
we can assume $\Delta_{G/H}=X_{G/H}\times\Bi(G/H)$. By \ref{rem:extension of torsors}
and \ref{lem:extension of torsors along closed embedding} there is
a lifting $X_{G/H}\arr\Delta_{G}$ of the given map $X_{G/H}\arr\Delta_{G/H}$.
In particular, using \ref{prop:fiber product modding out by a central subgroup},
we obtain a map $X_{G}=X_{G/H}\times X_{H}\arr X_{G/H}\times\Delta_{H}\arr\Delta_{G}$,
which is finite and étale of degree $\sharp G$. Since $G$ is abelian
and using \ref{lem: constant sheaves and t} we have $(\Autsh_{\Delta_{G}}P)(B)=G(B((t)))=G(B)$
for all $P\in\Delta_{G}(B)$. By \ref{prop:rigidification as gerbe},
it follows that the rigidification $F=\Delta_{G}\rig G$ is the sheaf
of isomorphism classes of $\Delta_{G}$ and that $\Delta_{G}\arr F$
is a gerbe locally $\Bi G$. Since $X_{G}\arr\Delta_{G}$ and, thanks
to \ref{cor:about limits and atlas}, $\A^{v}=\varinjlim_{n}\A^{v_{n}}\arr\Delta_{G}$
are finite and étale of degree $\sharp G$, by \ref{lem:dividing rank for gerbes}
we can conclude that $X_{G}\arr F$ and $\A^{v}\arr F$ are isomorphisms.
Since a gerbe having a section is trivial we get our result.
\end{proof}
With notation and hypothesis from Theorem \ref{A} set $\A^{v}=\varinjlim\A^{v_{*}}$,
$\overline{\Delta_{G}}$ for the coarse ind-algebraic space of $\Delta_{G}$
and consider the induced map $\A^{v}\arr\overline{\Delta_{G}}$. We
want to show that when $G$ is non-abelian this map is not an isomorphism
in general. The key point is the following Lemma.
\begin{lem}
If $K$ is an algebraically closed field and $P\in\Delta_{G}(K)$
then $H=\Autsh_{\Delta_{G}}(P)$ is (non canonically) a subgroup of
$G$ and the fiber of $\A^{v}(K)\arr\overline{\Delta_{G}}(K)\simeq\Delta_{G}(K)/\simeq$
over $P$ has cardinality $\sharp G/\sharp H$.
\end{lem}

\begin{proof}
There exist $n\in\N$ and $P_{n}\in\stX_{n}(K)$ inducing $P\in\Delta_{G}(K)$.
By \ref{lem:limit of coarse} we have $\Autsh_{\stX_{n}}(P_{n})=H$
and, since $\stX_{n}$ is a quasi-separated DM stack, it follows that
$H$ is a finite and constant group scheme. Moreover the map
\[
H(K)=\Aut_{K((t))}^{G}(P)\arr\Aut_{\overline{K((t))}}^{G}(P\times\overline{K((t))})\simeq G
\]
is injective (the last isomorphism depends on the choice of a section
in $P(\overline{K((t))})$. Thanks to \ref{cor:about limits and atlas}
we have $2$-Cartesian diagrams   \[   \begin{tikzpicture}[xscale=1.5,yscale=-1.2]     \node (A0_0) at (0, 0) {$Y$};     \node (A0_1) at (1, 0) {$W$};     \node (A0_2) at (2, 0) {$V$};     \node (A0_3) at (3, 0) {$\A^{v_n}$};     \node (A0_4) at (4, 0) {$\A^v$};     \node (A1_0) at (0, 1) {$\Spec K$};     \node (A1_1) at (1, 1) {$\Bi H$};     \node (A1_2) at (2, 1) {$U$};     \node (A1_3) at (3, 1) {$\stX_n$};     \node (A1_4) at (4, 1) {$\Delta_G$};     \node (A2_2) at (2, 2) {$\Spec K$};     \node (A2_3) at (3, 2) {$\overline\stX_n$};     \path (A1_3) edge [->]node [auto] {$\scriptstyle{}$} (A1_4);     \path (A0_0) edge [->]node [auto] {$\scriptstyle{}$} (A0_1);     \path (A0_1) edge [->]node [auto] {$\scriptstyle{}$} (A1_1);     \path (A1_0) edge [->]node [auto] {$\scriptstyle{}$} (A1_1);     \path (A0_3) edge [->]node [auto] {$\scriptstyle{}$} (A1_3);     \path (A1_1) edge [->]node [auto] {$\scriptstyle{}$} (A1_2);     \path (A2_2) edge [->]node [auto] {$\scriptstyle{}$} (A2_3);     \path (A0_3) edge [->]node [auto] {$\scriptstyle{}$} (A0_4);     \path (A0_4) edge [->]node [auto] {$\scriptstyle{}$} (A1_4);     \path (A0_2) edge [->]node [auto] {$\scriptstyle{}$} (A1_2);     \path (A0_0) edge [->]node [auto] {$\scriptstyle{}$} (A1_0);     \path (A1_3) edge [->]node [auto] {$\scriptstyle{}$} (A2_3);     \path (A0_1) edge [->]node [auto] {$\scriptstyle{}$} (A0_2);     \path (A1_2) edge [->]node [auto] {$\scriptstyle{}$} (A1_3);     \path (A1_2) edge [->]node [auto] {$\scriptstyle{}$} (A2_2);     \path (A0_2) edge [->]node [auto] {$\scriptstyle{}$} (A0_3);   \end{tikzpicture}   \] If
$F\subseteq\A^{v}(K)$ is the fiber we are looking for then we get
an induced map $V(K)\arr F$ which is easily seen to be surjective.
From \ref{lem:limit of coarse} it follows that $\A^{v_{n}}(K)\arr\A^{v}(K)$
is injective, which implies that $V(K)\arr F$ is bijective. From
\cite[\href{http://stacks.math.columbia.edu/tag/06ML}{06ML}]{SP014}
one get that $\Bi H$ is the reduction of $U$. Thus $W(K)=V(K)$.
Notice that, since $\overline{\stX_{n}}$ has schematically representable
diagonal, $V$, $W$ and $Y$ are all schemes. Since the vertical
maps in the top row are finite and étale of degree $\sharp G$ and
$Y\arr W$ is an $H$-torsor we conclude that $\sharp Y=G$ and $\sharp W=\sharp Y/\sharp H$
as required. 
\end{proof}
We see that if $G$ is not abelian and $P$ is a Galois extension
of $K((t))$ with group $G$, where $K$ is an algebraically closed
field, then the map $\A^{v}\arr\overline{\Delta_{G}}$ is not injective.
Indeed the fiber of $\A^{v}(K)\arr\overline{\Delta_{G}}(K)$ over
$P$ has cardinality $\sharp G/\sharp Z(G)$ because $\Aut_{K((t))}^{G}(P)=Z(G)$.
\begin{rem}
\label{rem:Harbater vs us}The moduli functor $F'$ described in \cite[Proof of 2.1]{Harbater1980}
is very similar to the sheaf of isomorphism classes of $\Delta_{G}$
but with some differences. Firstly $F'$ maps pointed connected affine
schemes to sets, while we look at the category of all (non-pointed)
affine schemes, which is standard in modern moduli theory. Secondly,
for a connected and pointed affine scheme $\Spec B$, he defines $F'(\Spec B)$
as the set of equivalence classes of pointed $G$-torsors on $B\otimes_{k}k((t))$
rather than $B((t))$ as in our case. Two covers are defined to be
equivalent if they agree after a finite étale pullback of $B\otimes_{k}k[[t]]$.
This equivalence relation plays the role of ``killing terms of positive
degrees'', while the same role is played by \ref{lem:removing the positive part}
in our setting. This is better understood in the case $G=\Z/p\Z$
where one can show that $F'$ is exactly the sheaf $\overline{\Delta}$
of isomorphism classes of $\Delta_{\Z/p\Z}$ (one can ignore base
points here because $\Z/p\Z$ is abelian). 

A map $\alpha\colon F'\arr\overline{\Delta}$ is well defined because
if two torsors $P,Q$ over $B\otimes k((t))$ become isomorphic after
an étale cover of $B\otimes k[[t]],$ by \ref{cor:shrinking to etale neighborhood}
$P\times B((t)),Q\times B((t))$ become isomorphic over $C((t))$,
where $C/B$ is an étale covering. The surjectivity of $\alpha$ is
easy: from the description of $\Delta_{\Z/p\Z}$ a torsor in $\overline{\Delta}(B)$
is given by an element $b\in B((t))$ with zero positive part and,
therefore, belonging to $B\otimes k((t))\subseteq B((t))$ (see also
\ref{lem:change of rings for series}). For the injectivity take $e\in B\otimes k((t))$
defining a torsor over $B\otimes k((t))$ which become trivial in
$\overline{\Delta}(B)$. Write $e=e_{-}+e_{+}$ as usual. Since $e_{+}\in B\otimes k[[t]]$
(which means that its associated torsor extends to $B\otimes k[[t]]$)
one has $e=e_{-}$ in $F'(B)$. Using the same notation and strategy
of \ref{thm:The case of Zp}, in particular of the essential surjectivity,
one can assume $e=\phi_{k}(\underline{b})$ for $\underline{b}\in\A^{(S)}$.
Since $e=0$ in $\overline{\Delta}=(\A^{(S)})^{\infty}$ it follows
that the coefficients of $b$ and $e$ are nilpotent. Since $e^{p}=e$
in $F'(B)$ it follows that $e=0$ in $F'(B)$. 

\end{rem}

\subsection{Semidirect products\label{subsec:Semidirect-products}}

The aim of this section is to complete the proof of Theorem \ref{A}.
So let $k$ be a field of positive characteristic $p$ and $G$ be
a finite and étale group scheme over $k$ such that $G\times_{k}\overline{k}$
is a semidirect product of a $p$-group and cyclic group of rank coprime
with $p$. 

Extending the base field by a Galois extension and using \ref{rem:base change for Delta}
and \ref{lem:descent of ind along torsors} we can assume that $G$
is constant, say $G=H\rtimes C_{n}$ where $H$ is a $p$-group and
$C_{n}$ is a cyclic group of order $n$ coprime with $p$. We can
moreover assume that the base field $k$ has all $n$-roots of unity,
so that $C_{n}\simeq\mu_{n}$ as group schemes. More precisely we
assume that $C_{n}=\mu_{n}(k)\subseteq k^{*}$ is the group of $n$-th
roots of unity of $k$.

Consider $\Spec k\arr\Delta_{C_{n}}$ given by $(k((t)),t^{q})$ as
in Theorem \ref{cyclic} and denote by $\stZ_{G,q}$ the fiber product
$\Spec k\times_{\Delta_{C_{n}}}\Delta_{G}$, which is the fibered
category of pairs $(P,\delta)$ where $P\in\Delta_{G}(B)$ and $\delta\colon P\arr\Spec(B((t))[X]/(X^{n}-t^{q}))$
is a $G$-equivariant map.
\begin{prop}
\label{prop:reducing to coprime numbers} Let $d=(n,q)$ and $G_{d}=H\rtimes C_{n/d}<G$.
Then the functor $\stZ_{G_{d},q/d}\arr\stZ_{G,q}$ induced by $\Delta_{G_{d}}\arr\Delta_{G}$
and $\Delta_{C_{n/d}}\arr\Delta_{C_{n}}$ is an equivalence.
\end{prop}

\begin{proof}
Set $Q_{n,q}=\Spec(B((t))[X]/(X^{n}-t^{q}))$. The map $X\longmapsto X$
induces a map $Q_{n/d,q/d}\arr Q_{n,q}$ which is $C_{n/d}$-equivariant,
that is $Q_{n,q}$ is the $C_{n}$-torsor induced by the $C_{n/d}$-torsor
$Q_{n/d,q/d}$. We obtain a quasi-inverse $\stZ_{G,q}\arr\stZ_{G_{d},q/d}$
mapping $P\arrdi{\phi}Q_{n,q}$ to the fiber product $P\times_{Q_{n,q}}Q_{n/d,q/d}\arr Q_{n/d,q/d}$.
\end{proof}
\begin{rem}
\label{rem:cyclic case for q coprime} If $(n,q)=1$ we have an isomorphism
of $B$-algebras
\[
B((s))\arr B((t))[X](X^{n}-t^{q})
\]
such that the $C_{n}$-action induced on the left is $s\longmapsto\xi^{\beta}s$
for $\xi\in C_{n}$, where $\beta\in\Z/n\Z$ is the inverse of $q$.
Indeed write $\beta q=1+\alpha n$ for some $\alpha,\beta\in\N$.
We have $(X^{\beta}/t^{\alpha})^{n}=t$ in $B((t))[X]/(X^{n}-t^{q})$
and isomorphisms   \[   \begin{tikzpicture}[xscale=3.9,yscale=-0.7]     
\node (A0_0) at (0.25, 0) {$s$};     
\node (A0_1) at (1, 0) {$X$};    
\node (A0_2) at (2, 0) {$X^{\beta}/t^{\alpha}$};     
\node (A1_0) at (0.25, 1) {$B((s))$};     
\node (A1_1) at (1, 1) {$B((t))[X]/(X^n-t)$};     
\node (A1_2) at (2, 1) {$B((t))[X]/(X^n-t^q)$};    
\path (A0_0) edge [|->]node [auto] {$\scriptstyle{}$} (A0_1);   
 \path (A1_0) edge [->]node [auto] {$\scriptstyle{\simeq}$} (A1_1);  
  \path (A0_1) edge [|->]node [auto] {$\scriptstyle{}$} (A0_2);    
\path (A1_1) edge [->]node [auto] {$\scriptstyle{\simeq}$} (A1_2);   \end{tikzpicture}   \] 
\end{rem}

\begin{defn}
Given a $p$-group $H$ and an autoequivalence $\phi\colon\Delta_{H}\arr\Delta_{H}$
we define $\stZ_{\phi}$ as the stack of pairs $(P,u)$ where $P\in\Delta_{H}$
and $u\colon P\arr\phi(P)$ is an isomorphism in $\Delta_{H}$.
\end{defn}

There are two natural autoequivalences of $\Delta_{H}$: $\phi_{\psi}\colon\Delta_{H}\arr\Delta_{H}$
obtained composing by an isomorphism $\psi\colon H\arr H$; $\phi_{\xi}\colon\Delta_{H}\arr\Delta_{H}$
induced by a $n$-th root of unity $\xi$ using the Cartesian diagram
  \[   \begin{tikzpicture}[xscale=2.7,yscale=-0.6]     \node (A0_0) at (0, 0) {$\phi_\xi(P)$};     \node (A0_1) at (1, 0) {$P$};     \node (A2_0) at (0, 2) {$\Spec B((t))$};     \node (A2_1) at (1, 2) {$\Spec B((t))$};     \node (A3_0) at (0, 3) {$\xi t$};     \node (A3_1) at (1, 3) {$t$};     \path (A3_1) edge [|->]node [auto] {$\scriptstyle{}$} (A3_0);     \path (A2_0) edge [->]node [auto] {$\scriptstyle{}$} (A2_1);     \path (A0_0) edge [->]node [auto] {$\scriptstyle{}$} (A0_1);     \path (A0_1) edge [->]node [auto] {$\scriptstyle{}$} (A2_1);     \path (A0_0) edge [->]node [auto] {$\scriptstyle{}$} (A2_0);   \end{tikzpicture}   \] 
\begin{prop}
\label{prop:passing to Zphi} Assume $(q,n)=1$ and let $\zeta\in C_{n}$
be a primitive $n$-th root of unity and $\psi\colon H\arr H$ be
the automorphism image of $\zeta$ under $C_{n}\arr\Aut(H)$. Set
$\xi=\zeta^{\beta}$ where $\beta\in\Z/n\Z$ is the inverse of $q$
and $\phi=\phi_{\psi}\circ\phi_{\xi}\colon\Delta_{H}\arr\Delta_{H}$.
Then $\stZ_{G,q}$ is an open and closed substack of $\stZ_{\phi}$.
\end{prop}

\begin{proof}
Let $P\in\Delta_{H}$. We have $\phi(P)=\phi_{\xi}(P)$ as schemes.
Thus an isomorphism $u\colon P\arr\phi(P)$, via $\phi(P)=\phi_{\xi}(P)\arr P$,
corresponds to an isomorphism $v\colon P\arr P$. The morphism $u$
is over $\Spec B((t))$ if and only if the following diagram commutes
  \[   \begin{tikzpicture}[xscale=2.7,yscale=-0.6]     \node (A0_0) at (0, 0) {$P$};     \node (A0_1) at (1, 0) {$P$};     \node (A2_0) at (0, 2) {$\Spec B((t))$};     \node (A2_1) at (1, 2) {$\Spec B((t))$};     \node (A3_0) at (0, 3) {$\xi t$};     \node (A3_1) at (1, 3) {$t$};     \path (A3_1) edge [|->]node [auto] {$\scriptstyle{}$} (A3_0);     \path (A2_0) edge [->]node [auto] {$\scriptstyle{}$} (A2_1);     \path (A0_0) edge [->]node [auto] {$\scriptstyle{v}$} (A0_1);     \path (A0_1) edge [->]node [auto] {$\scriptstyle{}$} (A2_1);     \path (A0_0) edge [->]node [auto] {$\scriptstyle{}$} (A2_0);   \end{tikzpicture}   \] Finally,
by going through the definitions, we see that $u$ is $H$-equivariant
if and only if $\psi(h)=vhv^{-1}$ in $\Aut(P)$ for all $h\in H$.
We identify $\stZ_{\phi}$ with the stack of pairs $(P,v)$ as above.

Set $S_{B}=\Spec B((s))$ with the $C_{n}$-action given by $s\mapsto\lambda^{\beta}s$
for $\lambda\in C_{n}$. By \ref{rem:cyclic case for q coprime} $S_{B}$
is isomorphic to $\Spec B((t))[X]/(X^{n}-t^{q})$ and therefore, by
construction, an object $(P,\delta)\in\stZ_{G,q}(B)$ is a $G$-torsor
over $B((t))$ with a $G$-equivariant map $\delta\colon P\arr S_{B}$.
In particular $P\arrdi{\delta}S_{B}$ is an $H$-torsor. Set $g_{0}=(1,\zeta)\in H\rtimes C_{n}=G$,
so that $\psi(h)=g_{0}hg_{0}^{-1}$ in $G$. Since $\delta$ is $G$-equivariant
it follows that $(P,g_{0})\in\stZ_{\phi}$. We therefore get a map
$\stZ_{G,q}\arr\stZ_{\phi}$. This map is fully faithful. Indeed if
$(P,\delta),(P',\delta')\in\stZ_{G,q}(B)$ the map on isomorphisms
is   \[   \begin{tikzpicture}[xscale=1.5,yscale=-1.2]     \node (A0_0) at (0, 0) {$\Iso_{\stZ_{G,q}}((P,\delta),(P',\delta'))=\Iso_{S_{B}}^{G}(P,P')$};     \node (A1_0) at (0, 1) {$\Iso_{\stZ_{\phi}}((P,g_{0}),(P',g_{0}))=\{\omega\in\Iso_{S_{B}}^{H}(P,P')\st g_{0}\omega=\omega g_{0}\}$};     \path (A0_0) edge [->]node [auto] {$\scriptstyle{}$} (A1_0);   \end{tikzpicture}   \] We
are going to show that the substack $\overline{\stZ}_{\phi}$ of pairs
$(P,v)$ where $v^{n}=\id$ is the essential image of $\stZ_{G,q}\arr\stZ_{\phi}$.
Since $g_{0}^{n}=1$ we have the inclusion $\subseteq$. Now let $(P,v)\in\overline{\stZ}_{\phi}$.
Since $v^{n}=1$ we get the map $G=H\rtimes C_{n}\arr\Aut(P)$ sending
$(h,\zeta^{m})$ to $hv^{m}$, defining a $G$-action on $P$. By
construction the map $\delta\colon P\arr S_{B}$ is $G$-equivariant.
Since $P/H\simeq S_{B}$ it remains to show that $G$ acts freely
on $P$. If $p(hv^{l})=p$ for some $p\in P$, then $\delta(ph)\zeta^{l}=\delta(p)\zeta^{l}=\delta(p)$
and therefore $n\mid l$ and $v^{l}=1$. Finally $ph=p$ implies $h=1$
in $H\subseteq\Aut(P)$ because $P$ is an $H$-torsor. Thus $G$
acts on $P$ freely.

We now show that $\overline{\stZ}_{\phi}$ is open and closed in $\stZ_{\phi}$.
If $(P,v)\in\stZ_{\phi}(B)$ we have that $v^{n}\in\Aut_{B((t))}^{H}(P)$
because $\psi(h)=vhv^{-1}$ and $\psi$ has order $n$. The group
scheme $\Autsh_{B((t))}^{H}(P)\arr\Spec B((t))$ is finite and étale,
thus the locus $W$ in $\Spec B((t))$ where $v^{n}=\id$ is open
and closed in $\Spec B((t))$. By \ref{lem: constant sheaves and t}
there is an open and closed subset $\widetilde{W}$ in $\Spec B$
inducing $W$. By construction the base change of $\overline{\stZ}_{\phi}\arr\stZ_{\phi}$
along $(P,v)\colon\Spec B\arr\stZ_{\phi}$ is $\widetilde{W}$, which
ends the proof.
\end{proof}
\begin{prop}
\label{prop:Zphi as limit} If $\phi\colon\Delta_{H}\arr\Delta_{H}$
is an equivalence then there exists a direct system of separated DM
stacks $\stZ_{*}$ with finite and universally injective transition
maps, with a direct system of finite and étale atlases $Z_{n}\arr\stZ_{n}$
of degree $(\sharp H)^{2}$ from affine schemes and an equivalence
$\varinjlim_{n}\stZ_{n}\simeq\stZ_{\phi}$.
\end{prop}

\begin{proof}
Consider a direct system of DM stacks $\stY_{*}$ as in Theorem \ref{A}
for the $p$-group $H$. Denote by $\Gamma_{\phi}\colon\Delta_{H}\arr\Delta_{H}\times\Delta_{H}$
be the graph of $\phi$ and by $\gamma_{u,v}\colon\stY_{u}\arr\stY_{v}$
the transition maps. By \ref{rem:objs plus iso} $\stZ_{\phi}$ is
the fiber product of $\Gamma_{\phi}$ and the diagonal of $\Delta_{H}$.
There exist an increasing function $\delta\colon\N\arr\N$ and $2$-commutative
diagrams   \[   \begin{tikzpicture}[xscale=1.5,yscale=-1.2]     \node (A0_0) at (0, 0) {$\stY_n$};     \node (A0_1) at (1, 0) {$\stY_{n+1}$};     \node (A0_2) at (2, 0) {$\Delta_H$};     \node (A1_0) at (0, 1) {$\stY_{\delta_n}$};     \node (A1_1) at (1, 1) {$\stY_{\delta_{n+1}}$};     \node (A1_2) at (2, 1) {$\Delta_H$};     \path (A0_0) edge [->]node [auto] {$\scriptstyle{}$} (A0_1);     \path (A0_1) edge [->]node [auto] {$\scriptstyle{}$} (A0_2);     \path (A1_0) edge [->]node [auto] {$\scriptstyle{}$} (A1_1);     \path (A1_1) edge [->]node [auto] {$\scriptstyle{}$} (A1_2);     \path (A0_2) edge [->]node [auto] {$\scriptstyle{\phi}$} (A1_2);     \path (A0_0) edge [->]node [auto] {$\scriptstyle{\phi_n}$} (A1_0);     \path (A0_1) edge [->]node [auto] {$\scriptstyle{\phi_{n+1}}$} (A1_1);   \end{tikzpicture}   \] Similarly
$\stY_{n}\arr\stY_{n}\times\stY_{n}\arrdi{\gamma_{n,\delta_{n}}\times\phi_{n}}\stY_{\delta_{n}}\times\stY_{\delta_{n}}$
and $\stY_{n}\arr\stY_{n}\times\stY_{n}\arrdi{\gamma_{n,\delta_{n}}\times\gamma_{n,\delta_{n}}}\stY_{\delta_{n}}\times\stY_{\delta_{n}}$approximate
$\Gamma_{\phi}$ and the diagonal of $\Delta_{H}$ respectively. By
\ref{prop:fiber product of direct systems} it follows that the fiber
product $\stZ_{n}=\stY_{n}\times_{\stY_{\delta_{n}}\times\stY_{\delta_{n}}}\stY_{n}$
of the two maps form a direct system of separated DM stacks whose
limit is $\stZ_{\phi}$. By \ref{rem:product of universally injective finite}
the transition maps are finite and universally injective. Let $Y_{n}\arr\stY_{n}$
be the finite and étale atlases of degree $\sharp H$ given in Theorem
\ref{A}. The induced map $Z_{n}=Y_{n}\times_{\stY_{\delta_{n}}\times\stY_{\delta_{n}}}Y_{n}\arr\stZ_{n}$
is finite and étale of degree $(\sharp H)^{2}$. Since $\stY_{\delta_{n}}\times\stY_{\delta_{n}}$
has affine diagonal it follows that $Z_{n}$ is affine. Finally, using
the usual properties of fiber products and the fact that $Y_{n}\arr\stY_{n}$
is a direct system of atlases, we see that the maps $Z_{n}\longrightarrow Z_{n+1}\times_{\shZ_{n+1}}\shZ_{n}$
are isomorphisms.
\end{proof}

\begin{proof}
[Proof of Theorem \ref{A}, \ref{thma-general}]\label{pf: A-general}
Recall that we have reduced the problem to the case of a constant
group $G=H\rtimes C_{n}$ at the beginning of this subsection \pageref{subsec:Semidirect-products}.
Consider $\pi\colon\Delta_{G}\arr\Delta_{C_{n}}$ and the decomposition
$\Delta_{C_{n}}=\bigsqcup_{q=1}^{n}\Bi G_{q}$, where $G_{q}=C_{n}$,
of Theorem \ref{cyclic}. The map
\[
\bigsqcup_{q=1}^{n}(\Bi G_{q}\times_{\Delta_{C_{n}}}\Delta_{G})\arr\Delta_{G}
\]
is well defined and an equivalence because given $k$-algebras $A_{1}$
and $A_{2}$ the map $\Delta_{G}(A_{1}\times A_{2})\arr\Delta_{G}(A_{1})\times\Delta_{G}(A_{2})$
is an equivalence. Thus in the statement of the theorem $\Delta_{G}$
can be replaced by $\widetilde{\stZ}_{G,q}=(\Bi G_{q}\times_{\Delta_{C_{n}}}\Delta_{G})$.
We must show that $\widetilde{\stZ}_{G,q}$ is a stack in the fpqc
topology if $\stZ_{G,q}$ is so. Let $B$ be a ring, $\shU=\{B\arr B_{i}\}_{i\in I}$
a covering and $\xi\in\widetilde{\stZ}_{G,q}(\shU)$ be a descent
datum. Given a $B$-scheme $Y$ we denote by $\shU_{Y}=\shU\times_{B}Y$
and by $\xi_{Y}\in\widetilde{\stZ}_{G,q}(\shU_{Y})$ the pullback.
Denote by $r\colon\widetilde{\stZ}_{G,q}\arr\Bi C_{n}$ the structure
map. The descent datum $r(\xi)$ yields a $C_{n}$-torsor $F\arr\Spec B$.
Let $Y\arr\Spec B$ a $B$-scheme with a factorization $Y\arr F$.
This factorization is a trivialization of the $C_{n}$-torsor over
$Y$ and therefore it induces a descent datum of $(\widetilde{\stZ}_{G,q}\times_{\Bi C_{n}}\Spec k)(\shU_{Y})=\stZ_{G,q}(\stU_{Y})$
which is therefore effective, yielding $\eta_{Y}\in\stZ_{G,q}(Y)$
and $\widetilde{\eta}_{Y}\in\widetilde{\stZ}_{G,q}(Y)$. By construction
$\widetilde{\eta}_{Y}\in\widetilde{\stZ}_{G,q}(Y)$ induces the descent
datum $\xi_{Y}\in\widetilde{\stZ}_{G,q}(\stU_{Y})$. In particular
we get $\widetilde{\eta}_{F}\in\widetilde{\stZ}_{G,q}(F)$. Since
$\widetilde{\stZ}_{G,q}$ is a prestack, the objects $\widetilde{\eta}_{F\times_{B}F}$
obtained using the two projections $F\times_{B}F\arr F$ are isomorphic
via a given isomorphism: they both induce $\xi_{F\times_{B}F}\in\widetilde{\stZ}_{G,q}(\shU_{F\times_{B}F})$
which does not depend on the projections being a pullback of $\xi\in\widetilde{\stZ}_{G,q}(\shU)$.
In conclusion $\widetilde{\eta}_{F}$ gives a descent datum for $\widetilde{\stZ}_{G,q}$
over the covering $F\arr\Spec B$. In order to get a global object
in $\widetilde{\stZ}_{G,q}(B)$ inducing the given descent datum $\xi$
it is enough to notice that, by \ref{lem:change of rings for series},
$\Delta_{G}$ satisfies descent along coverings $U\arr\Spec B$ which
are finite, flat and finitely presented.

Thanks to \ref{lem:descent of ind along torsors}, it is enough to
show that $\stZ_{G,q}$ is a limit as in the statement. Using \ref{prop:reducing to coprime numbers}
we can further assume $n$ and $q$ coprime and, using \ref{prop:passing to Zphi},
we can replace $\stZ_{G,q}$ by $\stZ_{\phi}$, where $\phi=\phi_{\psi}\circ\phi_{\xi}$
as in \ref{prop:passing to Zphi}. The conclusion now follows from
\ref{prop:Zphi as limit}.
\end{proof}

\appendix

\section{\label{sec:Limit} Limit of fibered categories}

In this appendix we discuss the notion of inductive limit of stacks.
To simplify the exposition and since general colimits were not needed
in this paper we will only talk about limit over the natural numbers
$\N$. General results can be found in \cite[Appendix A]{Tonini2016}.

A \emph{direct system }of categories $\shC_{*}$ (indexed by $\N$)
is a collection of categories $\shC_{n}$ for $n\in\N$ and functors
$\psi_{n}\colon\shC_{n}\arr\shC_{n+1}$. Given indexes $n<m$ we also
set 
\[
\psi_{n,m}\colon C_{n}\arrdi{\psi_{n}}\shC_{n+1}\arr\cdots\arr\shC_{m-1}\arrdi{\psi_{m-1}}\shC_{m}
\]
and $\psi_{n,n}=\id_{\shC_{n}}$. The limit $\varinjlim_{n\in\N}C_{n}$
or $\shC_{\infty}$ is the category defined as follows. Its objects
are pairs $(n,x)$ with $n\in\N$ and $x\in\shC_{n}$. Given pairs
$(n,x)$ and $(m,y)$ we set
\[
\Hom_{\shC_{\infty}}((n,x),(m,x))=\varinjlim_{q>n+m}\Hom_{\shC_{q}}(\psi_{n,q}(x),\psi_{m,q}(y))
\]
Composition is defined in the obvious way. There are obvious functors
$\shC_{n}\arr\shC_{\infty}$.

Given a category $\shD$ we denote by $\Hom(\shC_{*},\shD)$ the category
whose objects are collections $(F_{n},\alpha_{n})$ where $F_{n}\colon\shC_{n}\arr\shD$
are functors and $\alpha_{n}\colon F_{n+1}\circ\psi_{n}\arr F_{n}$
are natural isomorphisms of functors $\shC_{n}\arr\shD$. There is
an obvious functor $\Hom(\shC_{\infty},\shD)\arr\Hom(\shC_{*},\shD)$
and we have:
\begin{prop}
\cite[Remark A.3]{Tonini2016} The functor $\Hom(\shC_{\infty},\shD)\arr\Hom(\shC_{*},\shD)$
is an equivalence.
\end{prop}

This justifies calling $\shC_{\infty}$ the limit of the direct system
$\shC_{*}$.

Let $\shS$ be a category with fiber products. A direct system of
fibered categories $\stX_{*}$ over $\shS$ (indexed by $\N$) is
a direct system of categories $\shX_{*}$ together with maps $\stX_{n}\arr\shS$
making $\stX_{n}$ into a fibered category over $\shS$ and such that
the transition maps $\stX_{n}\arr\stX_{n+1}$ are maps of fibered
categories. Result \cite[Proposition A.4]{Tonini2016} translates
into what follows.

The induced functor $\stX_{\infty}\arr\shS$ makes $\stX_{\infty}$
into a fibered category over $\shS$ and the maps $\stX_{n}\arr\stX_{\infty}$
are map of fibered categories. Given an object $s\in\shS$ there is
an induced direct system of categories $\stX_{*}(s)$ and there is
a natural equivalence
\[
\varinjlim_{n\in\N}\stX_{n}(s)\arrdi{\simeq}\stX_{\infty}(s)
\]
In particular if all the $\stX_{n}$ are fibered in sets (resp. groupoids)
so is $\stX_{\infty}$.

If $\stY$ is another fibered category over $\shS$ denotes by $\Hom_{\shS}(\stX_{*},\stY)$
the subcategory of $\Hom(\stX_{*},\stY)$ of objects $(F_{n},\alpha_{n})$
where $F_{n}$ are base preserving functors and $\alpha_{n}$ are
base preserving natural transformations. Also the arrows in the category
$\Hom_{\shS}(\stX_{*},\stY)$ are required to be base preserving natural
transformations. There is an induced functor $\Hom_{\shS}(\stX_{\infty},\stY)\arr\Hom_{\shS}(\stX_{*},\stY)$
which is an equivalence of categories.

A direct check using the definition of fiber product yields the following. 
\begin{prop}
\label{prop:fiber product of direct systems} Let $\stX_{*},\stY_{*}$
and $\stZ_{*}$ be direct system of categories fibered in groupoids
over $\shS$ and assume they are given $2$-commutative diagrams   \[   \begin{tikzpicture}[xscale=1.9,yscale=-1.2]     \node (A0_0) at (0, 0) {$\stX_n$};     \node (A0_1) at (1, 0) {$\stX_{n+1}$};     \node (A0_2) at (2, 0) {$\stZ_{n}$};     \node (A0_3) at (3, 0) {$\stZ_{n+1}$};     \node (A1_0) at (0, 1) {$\stY_n$};     \node (A1_1) at (1, 1) {$\stY_{n+1}$};     \node (A1_2) at (2, 1) {$\stY_{n}$};     \node (A1_3) at (3, 1) {$\stY_{n+1}$};     \path (A0_0) edge [->]node [auto] {$\scriptstyle{}$} (A0_1);     \path (A0_1) edge [->]node [auto] {$\scriptstyle{a_{n+1}}$} (A1_1);     \path (A1_0) edge [->]node [auto] {$\scriptstyle{}$} (A1_1);     \path (A0_3) edge [->]node [auto] {$\scriptstyle{b_{n+1}}$} (A1_3);     \path (A0_2) edge [->]node [auto] {$\scriptstyle{b_n}$} (A1_2);     \path (A0_0) edge [->]node [auto] {$\scriptstyle{a_n}$} (A1_0);     \path (A0_2) edge [->]node [auto] {$\scriptstyle{}$} (A0_3);     \path (A1_2) edge [->]node [auto] {$\scriptstyle{}$} (A1_3);   \end{tikzpicture}   \] Then
the canonical map
\[
\varinjlim_{n\in\N}(\stX_{n}\times_{\stY_{n}}\stZ_{n})\arr\varinjlim_{n\in\N}\stX_{n}\times_{{\displaystyle \varinjlim_{n\in\N}}\stY_{n}}\varinjlim_{n\in\N}\stZ_{n}
\]
is an equivalence.
\end{prop}

\begin{cor}
\label{cor:about limits and atlas} Let $\stX_{*}$ and $\stY_{*}$
be direct systems of categories fibered in groupoids over $\shS$
and assume to have $2$-Cartesian diagrams   \[   \begin{tikzpicture}[xscale=1.5,yscale=-1.2]     \node (A0_0) at (0, 0) {$\stY_n$};     \node (A0_1) at (1, 0) {$\stY_{n+1}$};     \node (A1_0) at (0, 1) {$\stX_n$};     \node (A1_1) at (1, 1) {$\stX_{n+1}$};     \path (A0_0) edge [->]node [auto] {$\scriptstyle{}$} (A0_1);     \path (A1_0) edge [->]node [auto] {$\scriptstyle{}$} (A1_1);     \path (A0_1) edge [->]node [auto] {$\scriptstyle{}$} (A1_1);     \path (A0_0) edge [->]node [auto] {$\scriptstyle{}$} (A1_0);   \end{tikzpicture}   \] Then
the following diagrams are also $2$-Cartesian for all $n\in\N$:  \[   \begin{tikzpicture}[xscale=2.0,yscale=-1.2]     \node (A0_0) at (0, 0) {$\stY_n$};     \node (A0_1) at (1, 0) {$\varinjlim\stY_*$};     \node (A1_0) at (0, 1) {$\stX_n$};     \node (A1_1) at (1, 1) {$\varinjlim \stX_*$};     \path (A0_0) edge [->]node [auto] {$\scriptstyle{}$} (A0_1);     \path (A1_0) edge [->]node [auto] {$\scriptstyle{}$} (A1_1);     \path (A0_1) edge [->]node [auto] {$\scriptstyle{}$} (A1_1);     \path (A0_0) edge [->]node [auto] {$\scriptstyle{}$} (A1_0);   \end{tikzpicture}   \] 
\end{cor}

\begin{proof}
We have maps $\stY_{n}\arrdi{\simeq}\stX_{n}\times_{\stX_{m}}\stY_{m}\arr\stX_{n}\times_{\stX_{\infty}}\stY_{\infty}$
for all $m>n$. Passing to the limit on $m$ we get the result.
\end{proof}
\begin{lem}
\label{lem:check stack on finite coverings} Assume that $\shS$ has
a Grothendieck topology such that for all coverings $\shV=\{U_{i}\arr U\}_{i\in I}$
there exists a finite subset $J\subseteq I$ for which $\shV_{J}=\{U_{j}\arr U\}_{j\in J}$
is also a covering. Let also $\stY$ be a fibered category. Then $\stY$
is a stack (resp. pre-stack) if and only if given a finite covering
$\shU$ of $U\in\shS$ the functor $\stY(U)\arr\stY(\shU)$ is an
equivalence (resp. fully faithful), where $\stY(\shU)$ is the category
of descent data of $\stY$ over $\shU$.
\end{lem}

\begin{proof}
The ``only if'' part is trivial. We show the ``if'' part. Let
$\shV=\{U_{i}\arr U\}_{i\in I}$ be a general covering and consider
a finite subset $J\subseteq I$ for which $\shV_{J}=\{U_{i}\arr U\}_{i\in J}$
is also a covering. Thus the composition $\stY(U)\arr\stY(\shV)\arr\stY(\shV_{J})$
is an equivalence (resp. fully faithful) and it is enough to show
that $\stY(\shV)\arr\stY(\shV_{J})$ is faithful. This follows because
there is a $2$-commutative diagram   \[   \begin{tikzpicture}[xscale=3.3,yscale=-1.5]     \node (A0_0) at (0, 0) {$\stY(\shV)$};     \node (A0_1) at (1, 0) {$\displaystyle\prod_{i\in I}\stY(U_i)$};     \node (A1_0) at (0, 1) {$\stY(\shV_J)$};     \node (A1_1) at (1, 1) {$\displaystyle\prod_{i\in I}(\prod_{j\in J}\stY(U_i\times_U U_j))$};     \path (A0_0) edge [->]node [auto] {$\scriptstyle{a}$} (A0_1);     \path (A0_0) edge [->]node [auto] {$\scriptstyle{}$} (A1_0);     \path (A0_1) edge [->]node [auto] {$\scriptstyle{b}$} (A1_1);     \path (A1_0) edge [->]node [auto] {$\scriptstyle{}$} (A1_1);   \end{tikzpicture}   \] where
the functors $a$ and $b$ are faithful.
\end{proof}
\begin{prop}
\label{prop:limit of stacks is stack} In the hypothesis of \ref{lem:check stack on finite coverings},
if $\stX_{*}$ is a direct system of stacks (resp. pre-stacks) over
$\shS$ then $\stX_{\infty}$ is also a stack (resp. pre-stack) over
$\shS$.
\end{prop}

\begin{proof}
It is easy to prove descent (resp. descent on morphisms) and its uniqueness
along coverings indexed by finite sets. By \ref{lem:check stack on finite coverings}
this is enough.
\end{proof}
Clearly the site $\shS$ we have in mind in the above proposition
is a category fibered in groupoids over the category of affine schemes
$\Aff$ with any of the usual topologies, for instance $\Aff/X$,
the category of affine schemes together with a map to a given scheme
$X$.

\section{\label{sec:Rigidification}Rigidification revisited}

Rigidification is an operation that allows us to ``kill'' automorphisms
of a given stack by modding out stabilizers by a given subgroup of
the inertia. This operation is described in \cite[Appendix A]{Abramovich2007}
in the context of algebraic stacks, but one can easily see that this
is a very general construction. In this appendix we discuss it in
its general form so that we can apply it to non-algebraic stacks like
$\Delta_{G}$. 

Let $\shS$ be a site, $\stX$ be a stack in groupoids over $\shS$
and denote by $I(\stX)\arr\stX$ the inertia stack. The inertia stack
can be also thought as the sheaf $\stX^{\op}\arr(\text{Groups})$
mapping $\xi\in\stX(U)$ to $\Aut_{\stX(U)}(\xi)$. By a subgroup
sheaf of the inertia stack we mean a subgroup sheaf of the previous
functor. Notice that given a sheaf $F\colon\stX^{\op}\arr\sets$ and
an object $\xi\in\stX(U)$ one get a sheaf $F_{\xi}$ on $U$ by composing
$(\shS/U)^{\op}\arrdi{\xi}\stX^{\op}\arr\sets$, where the first arrow
comes from the $2$-Yoneda lemma. Concretely one has $F_{\xi}(V\arrdi gU)=F(g^{*}\xi)$.
If $f\colon V\arr U$ is any map in $\shS$ there is a canonical isomorphism
$F_{\xi}\times_{U}V\simeq F_{f^{*}\xi}$.

Notice moreover that a subgroup sheaf $\shH$ of $I(\stX)$ is automatically
normal: if $\xi\in\stX(U)$ and $\omega\in I(\stX)(\xi)=\Aut_{\stX(U)}(\xi)$
then $\omega$ induces the conjugation $I(\stX)(\xi)\arr I(\stX)(\xi)$
and, since $\shH$ is a subsheaf, the subgroup $\shH(\xi)$ is preserved
by the conjugation.

We now describe how to rigidify $\stX$ by any subgroup sheaf $\shH$
of the inertia. We define the category $\widetilde{\stX\rig\shH}$
as follows. The objects are the same as the ones of $\stX$. Given
$\xi\in\stX(U)$ and $\eta\in\stX(V)$ an arrow $\xi\arr\eta$ in
$\widetilde{\stX\rig\shH}$ is a pair $(f,\phi)$ where $f\colon U\arr V$
and $\phi\in(\Isosh_{U}(f^{*}\eta,\xi)/\shH_{\xi})(U)$. Given $\zeta\in\stX(W)$
and arrows $\xi\arrdi{(f,\phi)}\eta\arrdi{(g,\psi)}\zeta$ we have
that $(\Isosh_{U}(f^{*}g^{*}\zeta,f^{*}\eta)/\shH_{f^{*}\eta})\simeq(\Isosh_{V}(g^{*}\zeta,\eta)/\shH_{\eta})\times_{V}U$,
because the action of $\shH_{*}$ is free. Moreover composition induces
a map
\[
(\Isosh_{U}(f^{*}g^{*}\zeta,f^{*}\eta)/\shH_{f^{*}\eta})\times(\Isosh_{U}(f^{*}\eta,\xi)/\shH_{\xi})\arr(\Isosh_{U}(f^{*}g^{*}\zeta,\xi)/\shH_{\xi})
\]
One set $(g,\psi)\circ(f,\phi)=(gf,\omega)$ where $\omega$ is the
image of $(f^{*}\psi,\phi)$ under the above map. It is elementary
to show that this defines a category $\widetilde{\stX\rig\shH}$ together
with a map $\widetilde{\stX\rig\shH}\arr\shS$ making it into a category
fibered in groupoids. The map $\stX\arr\widetilde{\stX\rig\shH}$
is also a map of fibered categories.
\begin{defn}
The rigidification $\stX\rig\shH$ of $\stX$ by $\shH$ over the
site $\shS$ is a stackification of the category fibered in groupoids
$\widetilde{\stX\rig\shH}$ constructed above.
\end{defn}

Depending on the chosen foundation and the notion of category used,
a stackification does not necessarily exists. The usual workaround
is to talk about universes but in our case one can directly construct
a stackification $\stX\rig\shH$. We denote by $\stZ$ the category
constructed as follows. Its objects are pairs $(\shG\arr U,F)$ where
$U\in\shS$, $\shG\arr U$ is a gerbe and $F\colon\shG\arr\stX$ is
a map of fibered categories satisfying the following condition: for
all $y\in\shG$ lying over $V\in\shS$ the map $\Aut_{\shG(V)}(y)\arr\Aut_{\stX(V)}(F(y))$
is an isomorphism onto $\shH(F(y))$. An arrow $(\shG'\arr U',F')\arr(\shG\arr U,F)$
is a triple $(f,\omega,\delta)$ where   \[   \begin{tikzpicture}[xscale=1.6,yscale=-1.2]     \node (A0_0) at (0, 0) {$\shG'$};     \node (A0_1) at (1, 0) {$\shG$};     \node (A1_0) at (0, 1) {$U'$};     \node (A1_1) at (1, 1) {$U$};     \path (A0_0) edge [->]node [auto] {$\scriptstyle{\omega}$} (A0_1);     \path (A1_0) edge [->]node [auto] {$\scriptstyle{f}$} (A1_1);     \path (A0_0) edge [->]node [auto] {$\scriptstyle{}$} (A1_0);     \path (A0_1) edge [->]node [auto] {$\scriptstyle{}$} (A1_1);   \end{tikzpicture}   \] is
a $2$-Cartesian diagram and $\delta\colon F\circ\omega\arr F'$ is
a base preserving natural isomorphism. The class of arrows between
two given objects is in a natural way a category rather than a set.
On the other hand, since the maps from the gerbes to $\stX$ are faithful
by definition, this category is equivalent to a set: between two $1$-arrows
there exist at most one $2$-arrow. In particular $\stZ$ is a $1$-category.
 It is not difficult to show that $\stZ$ is fibered in groupoids
over $\shS$ and that it satisfies descent, i.e., it is a stack in
groupoids over $\shS$.

There is a functor $\Delta\colon\stX\arr\stZ$ mapping $\xi\in\stX(U)$
to $F_{\xi}\colon\Bi\shH_{\xi}\arr\stX\times U\arr\stX$. If $\psi\colon\xi'\arr\xi$
is an isomorphism in $\stX(U)$, then $\Delta(\psi)=(\Bi(c_{\psi}),\lambda_{\psi})$
where $c_{\psi}\colon\shH_{\xi'}\arr\shH_{\xi}$ is the conjugation
by $\psi$ and $\lambda_{\psi}$ is the unique natural transformation
$F_{\xi'}\arr F_{\xi}\circ\Bi(c_{\psi})$ that evaluated in $\shH_{\xi'}$
yields $\xi'\arrdi{\psi}\xi$. For the existence and uniqueness of
$\lambda_{\psi}$ recall that, by descent, a natural transformation
of functors $Q,Q'$ from a stack of torsors to a stack, is the same
datum of an isomorphism between the values of $Q$ and $Q'$ on the
trivial torsor which is functorial with respect to the automorphisms
of the trivial torsor. In our case a natural transformation $F_{\xi'}\arr F_{\xi}\circ\Bi(c_{\psi})$
is an isomorphism $\omega\colon\xi'\arr\xi$ (the values of the functors
on the trivial torsor $\shH_{\xi'}$) such that $c_{\psi}(u)=\omega u\omega^{-1}$
for all $\xi'\arrdi u\xi'\in\shH_{\xi'}$ (that is for all automorphisms
of the trivial torsor).

Given an object $z=(\shG,F)\in\stZ(U)$ there is a natural isomorphism
making the following diagram $2$-Cartesian:   \[   \begin{tikzpicture}[xscale=1.5,yscale=-1.2]     \node (A0_0) at (0, 0) {$\stG$};     \node (A0_1) at (1, 0) {$\stX$};     \node (A1_0) at (0, 1) {$U$};     \node (A1_1) at (1, 1) {$\stZ$};     \path (A0_0) edge [->]node [auto] {$\scriptstyle{F}$} (A0_1);     \path (A1_0) edge [->]node [auto] {$\scriptstyle{z}$} (A1_1);     \path (A0_1) edge [->]node [auto] {$\scriptstyle{\Delta}$} (A1_1);     \path (A0_0) edge [->]node [auto] {$\scriptstyle{}$} (A1_0);   \end{tikzpicture}   \] For
this reason we call $\Delta\colon\stX\arr\stZ$ the universal gerbe.
 The key point in proving this is that if we have a gerbe $\shG$
over $U$ and a section $x\in\shG(U)$ then the functor 
\[
\shG\arr\Bi(\Autsh_{\shG}(x))\comma y\longmapsto\Isosh_{\shG}(x,y)
\]
is well defined and an equivalence. In particular if $(\shG,F)\in\stZ(U)$
then $\Autsh_{\shG}(x)\simeq\shH_{F(x)}$ via $F$. 
\begin{prop}
\label{prop:explicit rigidification}The functor $\Delta\colon\stX\arr\stZ$
induces a fully faithful epimorphism $\widetilde{\stX\rig\shH}\arr\stZ$.
In particular $\stZ$ is a rigidification $\stX\rig\shH$ of $\stX$
by $\shH$.
\end{prop}

\begin{proof}
Given a functor $T\arr\stZ$ induced by $(\shG\arr T,F)\in\stZ(T)$
then $T\times_{\stZ}\stX$ is the stack of triples $(f,\xi,\omega)$
where $f\colon S\arr T$, $\xi\in\stX(S)$ and $\omega$ is an isomorphism
$(\Bi\shH_{\xi},F_{\xi})\simeq(\shG,F)$. Denote by $\Delta\colon\stX\arr\stZ$
the functor and let $\xi,\xi'\in\stX(U)$. Since $\stX\arr\stZ$ is
clearly an epimorphism, we have to prove that
\[
\Isosh_{U}(\xi',\xi)\arr\Isosh_{U}(\Delta(\xi'),\Delta(\xi))
\]
is invariant by the action of $\shH_{\xi}$ and an $\shH_{\xi}$-torsor.
Notice that a functor of the form $\Bi\shH_{\xi'}\arr\Bi\shH_{\xi}$
is locally induced by a group homomorphism $\shH_{\xi'}\arr\shH_{\xi}$.
Thus it is enough to prove that if $c\colon\shH_{\xi'}\arr\shH_{\xi}$
is an isomorphism of groups and $\lambda\colon F_{\xi'}\arr F_{\xi}\circ\Bi c$
is an isomorphism then the set $J$ of $\phi\colon\xi'\arr\xi$ inducing
$(\Bi(c),\lambda)\colon\Delta(\xi')\arr\Delta(\xi)$ is non empty
and $\shH_{\xi}(U)$ acts transitively of this set. 

The natural transformation $\lambda$ evaluated on the trivial torsor
$\shH_{\xi'}$ yields an isomorphism $\phi\colon\xi'\arr\xi$. The
fact that $\lambda$ is a natural transformation implies that $c=c_{\phi}$
and $\lambda=\lambda_{\phi}$, that is $\phi\in J$. Now let $\xi'\arrdi{\psi}\xi$
be an isomorphism. A natural isomorphism $\Bi(c_{\psi})\arr\Bi(c_{\phi})$
is given by $h\in\shH_{\xi}(U)$ (more precisely the multiplication
$\shH_{\xi}\arr\shH_{\xi}$ by $h$) such that $hc_{\psi}(\omega)=c_{\phi}(\omega)h$
for all $\omega\in\shH_{\xi'}(U)$. Such an $h$ induces a morphism
$\Delta(\psi)\arr\Delta(\phi)$ if and only if $h\psi=\phi$. Since
this condition implies the previous one we see that $J=\shH_{\xi}(U)\phi$.
\end{proof}
We denote by $\Bi_{\stX}\shH$ the stack of $\shH$-torsors over $\stX$
(thought of as a site). An object of $\Bi_{\stX}\shH$ is by definition
an object $\xi\in\stX(U)$ together with an $\shH_{|\stX/\xi}$-torsor
over $\stX/\xi$. Since $\stX$ is fibered in groupoids the forgetful
functor $\stX/\xi\arr\shS/U$ is an equivalence. Thus an object of
$\Bi_{\stX}\shH$ is an object $\xi\in\stX(U)$ together with a $\shH_{\xi}$-torsor
over $U$.
\begin{prop}
\label{prop:properties of rigidification}We have:

\begin{enumerate}
\item given $\xi,\eta\in\stX(U)$ we have $\Isosh_{U}(\Delta(\xi),\Delta(\eta))\simeq\Isosh(\xi,\eta)/\shH_{\eta}$;
\item the functor $\Delta\colon\stX\arr\stX\rig\shH$ is universal among
maps of stacks $F\colon\stX\arr\stY$ such that, for all $\xi\in\stX(U)$,
$\shH_{\xi}$ lies in the kernel of $\Autsh_{U}(\xi)\arr\Autsh_{U}(F(\xi))$;
\item if   \[   \begin{tikzpicture}[xscale=1.5,yscale=-1.2]     \node (A0_0) at (0, 0) {$\stY$};     \node (A0_1) at (1, 0) {$\stX$};     \node (A1_0) at (0, 1) {$\stR$};     \node (A1_1) at (1, 1) {$\stX\rig\shH$};     \path (A0_0) edge [->]node [auto] {$\scriptstyle{b}$} (A0_1);     \path (A1_0) edge [->]node [auto] {$\scriptstyle{}$} (A1_1);     \path (A0_1) edge [->]node [auto] {$\scriptstyle{\Delta}$} (A1_1);     \path (A0_0) edge [->]node [auto] {$\scriptstyle{a}$} (A1_0);   \end{tikzpicture}   \] is
a $2$-Cartesian diagram of stacks then for all $\eta\in\stY(U)$
the map 
\[
\Ker(\Autsh_{U}(\eta)\arr\Autsh_{U}(a(\eta)))\arr\Autsh_{U}(b(\eta))
\]
is an isomorphism onto $\shH_{b(\eta)}$, so that $b^{*}\shH$ is
naturally a subgroup sheaf of $I(\stY)$, and the induced map $\stY\rig b^{*}\shH\arr\stR$
is an equivalence;
\item there is an isomorphism $\stX\times_{\stX\rig\shH}\stX\simeq\Bi_{\stX}(\shH)$;
\item the map $\stX\to\stX\rig\shH$ is a relative gerbe (see \cite[Tag 06P1]{SP014}). 
\end{enumerate}
\end{prop}

\begin{proof}
Point $1)$ follows from \ref{prop:explicit rigidification}, while
point $2)$ is a direct consequence of the definition of rigidification.
Now consider point $3)$. The kernel in the statement corresponds
to the group of automorphisms of the object $\eta$ in the fiber product
$U\times_{\stR}\stY$. Using that $\stX\arr\stX\rig\shH$ is the universal
gerbe we get the the isomorphism in the statement. In particular there
is an epimorphism $\stY\rig b^{*}\shH\arr\stR$. This is fully faithful
because, given $\eta,\eta'\in\stY(U)$, by definition of fiber product
one get a Cartesian diagram   \[   \begin{tikzpicture}[xscale=4.0,yscale=-1.2]     \node (A0_0) at (0, 0) {$\Isosh_U(\eta',\eta)$};     \node (A0_1) at (1, 0) {$\Isosh_U(b(\eta'),b(\eta))$};     \node (A1_0) at (0, 1) {$\Isosh_U(a(\eta'),a(\eta))$};     \node (A1_1) at (1, 1) {$\Isosh_U(b(\eta'),b(\eta))/\shH_{b(\eta)}$};     \path (A0_0) edge [->]node [auto] {$\scriptstyle{}$} (A0_1);     \path (A0_0) edge [->]node [auto] {$\scriptstyle{}$} (A1_0);     \path (A0_1) edge [->]node [auto] {$\scriptstyle{}$} (A1_1);     \path (A1_0) edge [->]node [auto] {$\scriptstyle{}$} (A1_1);   \end{tikzpicture}   \] For
point $4)$, denotes by $\stY$ the fiber product in the statement.
It is the stack of triples $(\xi,\xi',\phi)$ where $\xi,\xi'\in\stX(U)$
and $\phi\in\Isosh_{U}(\xi',\xi)/\shH_{\xi}$. The functor $\stY\arr\Bi_{\stX}\shH$
which maps $(\xi,\xi',\phi)$ to $(\xi,P_{\phi})$, where $P_{\phi}$
is defined by the Cartesian diagram   \[   \begin{tikzpicture}[xscale=2.5,yscale=-1.2]     \node (A0_0) at (0, 0) {$P_\phi$};     \node (A0_1) at (1, 0) {$\Isosh_U(\xi',\xi)$};     \node (A1_0) at (0, 1) {$U$};     \node (A1_1) at (1, 1) {$\Isosh_U(\xi',\xi)/\shH_\xi$};     \path (A0_0) edge [->]node [auto] {$\scriptstyle{}$} (A0_1);     \path (A0_0) edge [->]node [auto] {$\scriptstyle{}$} (A1_0);     \path (A0_1) edge [->]node [auto] {$\scriptstyle{}$} (A1_1);     \path (A1_0) edge [->]node [auto] {$\scriptstyle{\phi}$} (A1_1);   \end{tikzpicture}   \] is
an equivalence. This is because the functor $\stX\arr\Bi_{\stX}\shH$
sending $\xi$ to $\xi$ with the trivial torsor is an epimorphism
and the base change $\stY\times_{\Bi_{\stX}\shH}\stX\arr\stX$ is
an equivalence since $\stY\times_{\Bi_{\stX}\shH}\stX$ is the stack
of triples $(\xi,\xi',\psi)$ where $\xi,\xi'\in\stX(U)$ and $\psi\colon\xi'\arr\xi$
is an isomorphism in $\stX(U)$.

For point $5)$, since $\Delta\colon\stX\to\stX\rig\shH$ is an epimorphism,
it has local sections. Moreover given two objects $\xi,\eta\in\stX(U)$
and an isomorphism $\Delta(\xi)\to\Delta(\eta)$, by point $1)$,
this isomorphism locally comes from an isomorphism $\xi\to\eta$ in
$\stX$, as required.
\end{proof}
\begin{prop}
\label{prop:trivial rigidification} Let $\stX$ be a stack in groupoid
over $\shS$ and $\shG\colon\shS^{\op}\to\Ab$ be a sheaf of abelian
groups. Then the map
\[
(\stX\times\Bi_{\shS}\shG)\rig\shG\to\stX
\]
 is an equivalence.
\end{prop}

\begin{proof}
Let $U\in\shS$ be an object and $P\in\Bi_{\shS}\shG(U)$ a $\shG$-torsor.
Since $\shG$ is commutative the action of $\shG$ on $P$ is $\shG$-equivariant
and therefore the map
\[
\shG\times U\to\Autsh_{\Bi_{\shS}\shG}(P)
\]
is well defined and, checking locally, an isomorphism. This implies
that $\shG$, more precisely the restriction $(\stX\times\Bi_{\shS}\shG)^{\op}\to\shS^{\op}\to\Ab$,
is a subgroup of the inertia stack of $(\stX\times\Bi_{\shS}\shG)$,
thought of as a sheaf of groups. Since the functor $\stX\times\Bi_{\shS}\shG\to\stX$
kills the automorphisms in $\shG$ we obtain the map in the statement
thanks to \ref{prop:properties of rigidification}, $2)$. By \ref{prop:properties of rigidification},
$3)$ we can assume $\stX=\shS$. In order to show that $\Bi_{\shS}\shG\rig\shG\to\shS$
is an equivalence, it is enough to show that the functor $\widetilde{\Bi_{\shS}\shG\rig\shG}\to\shS$
is fully faithful. By construction, the objects of the first category
are torsors $P,Q\in\Bi_{\shS}\shG(U)$ and the morphisms are $\Isosh^{\shG}(P,Q)/\shG$.
But this is a sheaf which is locally trivial and therefore it is trivial.
It follows that $\Iso_{\widetilde{\Bi_{\shS}\shG\rig\shG}}(P,Q)$
consists of just one element.
\end{proof}
\begin{prop}
\label{prop:rigidification as gerbe} Let $\stX$ be a stack in groupoid
over $\shS$ and $\shG\colon\shS^{\op}\to(\text{Groups})$ be a sheaf
of groups. Assume there is an isomorphism between the restriction
$\shG_{\stX}\colon\stX^{\op}\to\shS^{\op}\to(\text{Groups})$ and
the inertia $I(\stX)\colon\stX^{\op}\to(\text{Groups})$. Then $\shG_{\stX}$
is a sheaf of abelian groups, $\stX\rig\shG$ is the sheaf of isomorphism
classes of $\stX$ and $\stX\to\stX\rig\shG$ is a relative gerbe.
Moreover for any object $U\in\shS$ and map $U\to\stX\rig\shG$ the
fiber $\stX\times_{\stX\rig\shG}U\to U$ is locally of the form $\Bi_{U}\shG_{|U}\to U$.
\end{prop}

\begin{proof}
Let $h\in\shG_{\stX}(\xi)$ for $\xi\in\stX(U)$. The naturality of
the isomorphism $\shG_{\stX}(\xi)\to I(\stX)(\xi)=\Aut_{\stX(U)}(\xi)$
on the morphism $h\colon\xi\to\xi$ exactly implies that the conjugation
by $h$ on $I(\stX)(\xi)$ is the identity. Thus $\shG_{\stX}\simeq I(\stX)$
is abelian. By \ref{prop:properties of rigidification}, $2)$ any
map $\stX\to\shF$ to a sheaf factors through $\stX\rig\shG$ and
by \ref{prop:properties of rigidification},$1)$ the stack $\stX\rig\shG$
is a actually a sheaf. This implies that $\stX\rig\shG$ is the sheaf
of isomorphism classes of $\stX$. It is a gerbe thanks to \ref{prop:properties of rigidification},
$5)$. For the local form of $\stX\to\stX\rig\shG$ any map $U\to\stX\rig\shG$
locally factors through $\stX$ itself. In this case the fiber is
exactly $\Bi_{U}\shG_{|U}\to U$ thanks to \ref{prop:properties of rigidification},
$4)$.
\end{proof}
\bibliographystyle{alpha}
\bibliography{biblio}

\begin{thebibliography}{{Sta}17}

\bibitem[AGV64]{SGA4-2}
Michael Artin, Alexander Grothendieck, and Jean-Louis Verdier.
\newblock {\em {SGA4-2 - Th{\'{e}}orie des topos et cohomologie {\'{e}}tale des
  sch{\'{e}}mas - S{\'{e}}minaire de G{\'{e}}om{\'{e}}trie Alg{\'{e}}brique du
  Bois Marie - 1963-64 - vol. 2 (Lecture notes in mathematics '''270''')}}.
\newblock 1964.

\bibitem[AOV08]{Abramovich2007}
Dan Abramovich, Martin Olsson, and Angelo Vistoli.
\newblock {Tame stacks in positive characteristic}.
\newblock {\em Annales de l'institut Fourier}, 58(4):1057--1091, 2008.

\bibitem[Bat99]{Batyrev}
Victor~V. Batyrev.
\newblock Non-{A}rchimedean integrals and stringy {E}uler numbers of
  log-terminal pairs.
\newblock {\em J. Eur. Math. Soc. (JEMS)}, 1(1):5--33, 1999.

\bibitem[B{\v{C}}19]{BouthierCesnavicius}
Alexis Bouthier and Kestutis {\v{C}}esnavi{\v{c}}ius.
\newblock Torsors on loop groups and the hitchin fibration.
\newblock arXiv:1908.07480v1, 2019.

\bibitem[DL02]{Denef-Loeser}
Jan Denef and Fran{\c{c}}ois Loeser.
\newblock Motivic integration, quotient singularities and the {M}c{K}ay
  correspondence.
\newblock {\em Compositio Math.}, 131(3):267--290, 2002.

\bibitem[FM02]{Fried-Mezard}
Michael~D. Fried and Ariane M\'{e}zard.
\newblock Configuration spaces for wildly ramified covers.
\newblock In {\em Arithmetic fundamental groups and noncommutative algebra
  ({B}erkeley, {CA}, 1999)}, volume~70 of {\em Proc. Sympos. Pure Math.}, pages
  353--376. Amer. Math. Soc., Providence, RI, 2002.

\bibitem[Har80]{Harbater1980}
David Harbater.
\newblock Moduli of {$p$}-covers of curves.
\newblock {\em Comm. Algebra}, 8(12):1095--1122, 1980.

\bibitem[Kat86]{Katz}
Nicholas~M. Katz.
\newblock Local-to-global extensions of representations of fundamental groups.
\newblock {\em Ann. Inst. Fourier (Grenoble)}, 36(4):69--106, 1986.

\bibitem[OP10]{Obus-Pries}
Andrew Obus and Rachel Pries.
\newblock Wild tame-by-cyclic extensions.
\newblock {\em J. Pure Appl. Algebra}, 214(5):565--573, 2010.

\bibitem[Pri02]{Pries}
Rachel~J. Pries.
\newblock Families of wildly ramified covers of curves.
\newblock {\em Amer. J. Math.}, 124(4):737--768, 2002.

\bibitem[Rom05]{Romagny2005}
Matthieu Romagny.
\newblock {Group actions on stacks and applications}.
\newblock {\em The Michigan Mathematical Journal}, 53(1):209--236, 2005.

\bibitem[{Sta}17]{SP014}
The {Stacks Project Authors}.
\newblock {S}tacks {P}roject.
\newblock \url{http://stacks.math.columbia.edu}, 2017.

\bibitem[Swa80]{Swan}
Richard~G. Swan.
\newblock On seminormality.
\newblock {\em J. Algebra}, 67(1):210--229, 1980.

\bibitem[Ton14]{Tonini-diagonalizable}
Fabio Tonini.
\newblock Stacks of ramified covers under diagonalizable group schemes.
\newblock {\em Int. Math. Res. Not. IMRN}, (8):2165--2244, 2014.

\bibitem[Ton17]{Tonini-monoidal}
Fabio Tonini.
\newblock Ramified {G}alois covers via monoidal functors.
\newblock {\em Transform. Groups}, 22(3):845--868, 2017.

\bibitem[TY]{alpha_p}
Fabio Tonini and Takehiko Yasuda.
\newblock Notes on the motivic mckay correspondence for the group scheme
  $\alpha_p$.
\newblock arXiv:1803.09558, to appear in Proceedings of FJV2017 Kagoshima.

\bibitem[TY19]{Formal-Torsors-II}
Fabio Tonini and Takehiko Yasuda.
\newblock Moduli of formal torsors {II}.
\newblock arXiv:1909.09276, 2019.

\bibitem[TZ19]{Tonini2016}
Fabio Tonini and Lei Zhang.
\newblock Algebraic and {N}ori fundamental gerbes.
\newblock {\em J. Inst. Math. Jussieu}, 18(4):855--897, 2019.

\bibitem[WY15]{Wood-Yasuda-I}
Melanie~Matchett Wood and Takehiko Yasuda.
\newblock Mass formulas for local {G}alois representations and quotient
  singularities {I}: a comparison of counting functions.
\newblock {\em International Mathematics Research Notices},
  2015(23):12590--12619, 2015.

\bibitem[WY17]{Wood-Yasuda-II}
Melanie~Matchett Wood and Takehiko Yasuda.
\newblock Mass formulas for local {G}alois representations and quotient
  singularities {II}: dualities and resolution of singularities.
\newblock {\em Algebra Number Theory}, 11(4):817--840, 2017.

\bibitem[Yas]{Yasuda-motivic}
Takehiko Yasuda.
\newblock Motivic integration over wild {D}eligne-{M}umford stacks.
\newblock arXiv:1908.02932.

\bibitem[Yas14]{Yasuda-p-cyclic}
Takehiko Yasuda.
\newblock The {$p$}-cyclic {M}c{K}ay correspondence via motivic integration.
\newblock {\em Compos. Math.}, 150(7):1125--1168, 2014.

\bibitem[Yas17]{Yasuda-toward}
Takehiko Yasuda.
\newblock Toward motivic integration over wild {D}eligne-{M}umford stacks.
\newblock In {\em Higher dimensional algebraic geometry---in honour of
  {P}rofessor {Y}ujiro {K}awamata's sixtieth birthday}, volume~74 of {\em Adv.
  Stud. Pure Math.}, pages 407--437. Math. Soc. Japan, Tokyo, 2017.

\end{thebibliography}

\end{document}